\documentclass[11pt, a4paper]{amsart}
\usepackage{amsmath, amssymb, amsfonts, amsthm}
\newcommand{\cal}{\mathcal}

\newcommand{\E}{{\bf{E}}}

\newcommand{\Var}{{\bf{Var}}}

\def\II{\text{I\!I}}

\def\P{\hbox{\bf P}}
\def\E{\hbox{\bf E}}

\def\e{\varepsilon }

\def\Var{\hbox{\bf Var}}

\def\C{\cal C}

\def\A{\cal A}
\def\R{\cal R}
\def\D{\cal D}
\def\B{\cal B}

\def\Y{\cal Y}
\def\X{\cal X}
\def\U{\cal U}
\def\V{\cal V}

\def\F{\cal F}
\def\G{\cal G}

\def\MM{{\cal M}}

\def\VV{\mathbb V}
\def\YY{\mathbb Y}
\def\AA{\mathbb A}
\def\WW{\mathbb W}
\def\BB{\mathbb B}
\def\TT{\mathbb T}
\def\FF{\mathbb F}
\def\UU{\mathbb U}
\def\RR{\mathbb R}
\def\SS{{\mathbb S}}

\newtheorem{tm}{Theorem}
\newtheorem{lem}{Lemma}
\newtheorem{col}{Corollary}

\begin{document}

\parindent=0pt

\title[Symmetric Statistics]{Edgeworth  approximations for distributions\\
	of symmetric statistics}
\author[F. G\"otze]{Friedrich G\"otze$^1$}
\address{Faculty of Mathematics, University of Bielefeld, Germany}
\email{goetze@math.uni-bielefeld.de}

\author[M. Bloznelis]{Mindaugas Bloznelis$^1$}
\address{Department of Mathematics, Vilnius University, Lithuania}
\email{mindaugas.bloznelis@mif.vu.lt}
\thanks{$^1$Research supported by CRC 701}

 
\date{September 1, 2005}

\dedicatory{\centerline{In memoriam Willem Rutger van Zwet \; *March 31, 1934 \;\;  $\dagger$July 2, 2020}} 



\begin{abstract} We
	study the distribution of
 a general class of  asymptotically linear
statistics which are symmetric functions of $N$ independent
observations. The distribution functions of these statistics are
approximated by  an Edgeworth expansion with a remainder of order
$o(N^{-1})$. The Edgeworth expansion is based on Hoeffding's
decomposition which provides a stochastic expansion into a
linear part, a quadratic part as well as smaller higher
order parts. The validity of this
Edgeworth expansion is proved under Cram\'er's condition on the
linear part, moment assumptions for all parts of the statistic
and an optimal dimensionality requirement for the non linear part. 
\end{abstract}
\maketitle
	

\section{Introduction and Results}

{\bf 1.1.} By the classical central limit theorem the distributions of sums
$X_1+\dots+X_N$ of independent and identically distributed random variables
can be approximated by the normal distribution.
The accuracy of the normal
approximation  is of order $O(N^{-1/2})$
by the  well-known
Berry-Esseen theorem.

A function of observations
$X_1,\dots, X_N$ is called {\it linear statistic}
 if it can be
represented by 
a sum of functions depending on a single observation only.
Many important
statistics are non linear, but
 can be
approximated by  a linear statistic.
We call these statistics {\it asymptotically
	linear}.
The central limit theorem and the normal approximation 
with rate $O(N^{-1/2})$ extend to the class of
asymptotically linear statistics as well.
For comparisons of  performance of statistical procedures 
beyond efficiency of first order, that is  $o(N^{-1/2})$, at the level of Hodges and Lehmann deficiency, that is  $o(N^{-1})$,
more precise approximations beyond the normal approximation
are required.
Such a refinement is provided by Edgeworth expansions of  the distribution function. 
The one-term
Edgeworth expansion adds a
 correction term of order $O(N^{-1/2})$   to the
standard normal distribution function
$\Phi(x)=\int_{-\infty}^x(2\pi)^{-1/2}e^{-u^2/2}du$ and provides an
approximation with the error $o(N^{-1/2})$.
Similarly, the two-term Edgeworth expansion includes correction terms
of orders $O(N^{-1/2})$ and $O(N^{-1})$ as well with an approximation error of order $o(N^{-1})$.

Normal approximation theory including Edgeworth expansions is well studied for distribution function
$F_N(x)=\P\{X_1+\dots+X_N\le N\E X_1+N^{1/2}\sigma x\}$, $\sigma^2=\Var X_1$ of a sum of independent random variables,
see Petrov (1975) \cite{Petrov_1975}. Sums of independent random vectors are considered in Bhattacharya and Rao (1986)) \cite{Bhattacharya_Rao_1986}.
One distinguishes two cases. For summands taking values in an
arithmetic progression (lattice case) the distribution function of the sum
  has jumps of
 order $O(N^{-1/2})$.
 Corresponding asymptotic expansions
are
discontinuous functions designed to capture these jumps see
the seminal results by Esseen (1945) \cite{Esseen_1945}.
For sums of non-lattice random variables (non-lattice case) the one-term Edgeworth expansion
$$
\Phi(x)-\frac{1}{\sqrt N}\frac{\kappa_3}{6\sigma^3}(x^2-1)\Phi'(x)
$$
is a differentiable function.
 The
correcting term
reflects the skewness of
the distribution of a summand, $\kappa_3=\E (X_1-\E X_1)^3$.
More generally, a $k-$term Edgeworth expansion is a differentiable function with all derivatives bounded
(it is a sum of Hermite polynomials of increasing order with
scalar weights involve cumulants of order at least three of $X_1$ which vanish for $X_1$ being Gaussian).
Therefore, in order to establish  the validity of such an expansion, that is to prove the bound
$o(N^{-k/2})$ for the remainder, one should assume
that the underlying distribution is sufficiently smooth (see Bickel and Robinson 1982 \cite{Bickel_Robinson_1982}).
A convenient  condition to ensure that is Cram\'er's condition (C):
\begin{displaymath}\nonumber
\limsup_{|t|\to \infty} |\E \exp\{itX_1\}|<1.
\qquad
\qquad (C)
\end{displaymath}

In this paper we establish the two-term Edgeworth expansion for a general asymptotically linear statistic
$T=T(X_1,\dots, X_N)$ with non-lattice distribution.

There is a rich literature devoted
to  normal approximation and Edgeworth expansions
for
various classes of asymptotically linear  statistics
(see e.g. Babu and Bai (1993) \cite{Babu_Bai_1993},
Bai and Rao (1991) \cite{Bai_Rao_1991},
 Bentkus, G\"otze and van Zwet (1997)
 \cite{Bentkus_Goetze_vanZwet_1997}, 
 Bhattacharya and Ghosh (1978, 1980) \cite{Bhattacharya_Ghoshai_1978,Bhattacharya_Ghoshai_1980},
 Bhattacharya and Rao  (1986) \cite{Bhattacharya_Rao_1986}, 
 Bickel (1974) \cite{Bickel_1974},
 Bickel, G\"otze and van Zwet (1986) \cite{Bickel_Goetze_vanZwet_1986},
Callaert, Janssen and Veraverbeke (1980) \cite{Callaert_Janssen_Veraverbeke_1980},
Chibisov (1980) \cite{Chibisov_1980},
Hall (1987) \cite{Hall_1987},
Helmers (1982) \cite{Helmers_1982},
 Petrov (1975) \cite{Petrov_1975}, Pfanzagl (1985) \cite{Pfanzagl_1985}, Serfling (1980) \cite{Serfling_1980}, etc.

A wide class of statistics can be represented as functions of sample means of vector variables.
Edgeworth expansions of such statistics can be obtained by applying
 the multivariate expansion to corresponding functions, see Bhattacharya and Ghosh (1978, 1980) \cite{Bhattacharya_Ghoshai_1978,Bhattacharya_Ghoshai_1980}.
In their work the crucial  Cram\'er condition (C)  is assumed on the joint distribution of all the
components of a vector  which may be too restrictive
in cases where some components have a negligible influence
on the statistic.
More often only one or a few of the components satisfy   a conditional version of condition (C). Bai and Rao (1991) \cite{Bai_Rao_1991},
Babu and Bai (1993) \cite{Babu_Bai_1993} established Edgeworth expansions for functions of sample means under such a conditional
Cram\'er condition. This approach exploits the smoothness of the distribution  of vector as well as
 the smoothness
of the function defining the statistic.
In particular this approach needs a class of statistics which are smooth functions of observations or can
be approximated by such
functions via Taylor's expansion, see also Chibisov (1980) \cite{Chibisov_1980}.

Let us note that the smoothness  of the distribution function of a statistic, say $T=\phi(X_1,\dots, X_N)$
may have little to do  with the smoothness of the kernel $\phi$. Just take Gini's mean difference
 $\sum_{i<j}|X_i-X_j|$ with absolutely continuous $X_i$ for example.
The aim of this paper is to establish the validity of the  two-term Edgeworth
expansion for general asymptotically linear symmetric statistics in the following setup.
We require a minimal Cram\'er type condition for the linear part of the statistic  $T$ (given by its $L^2$ projection
to the space of linear statistics of $X_1, \ldots, X_N$),
and require instead of smoothness of $\phi$  of these
observations that the $L^2$-complementary non linear part of $T$ does not lie in a linear subspace in $L^2$ generated  by the linear part.  

\vskip 0.1cm

{\bf Remark.} {\it Determining in this setup  the actual influence of the non linear terms of $T$
on the approximation error represents a considerable challenge.
The crucial problem, which required  new techniques, was to control the quasi periodic behavior
of upper bounds for the characteristic function of $T$ at frequencies $t \sim N$,
These upper bounds involve  bounds on conditional characteristic functions of the non linear part of $T$,
say $g(t, \mathbf Y)^m$, given a subset $\mathbf Y:=\mathbf X_I$ of size $N-m$
 of the  observations. Understanding the separation of large maxima
of $(t, \mathbf Y )\to  g(t, \mathbf Y)^m$ in terms of both arguments in a function space
setup was finally achieved by  a combinatorial argument of Kleitman on symmetric partitions
(see Section 4)   controlling the concentration of sums like in  the Littlewood-Offord problem
in Banach spaces. 
The separation of large maxima of $g(t, \mathbf Y)^m$ then allowed to prove the
  desired bounds when averaging over $t$ and $\mathbf Y$. Note however, that more
    standard analytical arguments for concentration bounds would not work in the required
    generality here.       
 }

{\bf 2.1.}  Let $X,\, X_1,X_2,\dots, X_N$ be independent and
identically distributed random variables taking values in a
measurable space $(\X,\B)$. Let $P_X$ denotes the distribution of $X$ on $(\X,\B)$.
We assume that $\TT(X_1,\dots, X_N)$ is a symmetric
function of its arguments (symmetric statistic, for short). Furthermore, we assume that
the moments $\E \TT$ and
$\sigma_{\TT}^2:=\Var \TT$ are finite.

Our approach
is based on Hoeffding's
decomposition of $\TT$, see Hoeffding (1948) \cite{Hoeffding_1948}, Efron and Stein (1981) \cite{Efron_Stein_1981}
and
van Zwet (1984) \cite{vanZwet_1984}.
 Hoeffding's decomposition
expands $\TT$ into  the series of centered and mutually uncorrelated $U$-statistics of
increasing order

\begin{align*}
{\mathbb T}
=
\E {\mathbb T} 
&
+ \frac{1}{N^{1/2}} \sum_{1\le i\le N}g(X_i)
+
 \frac{1}{N^{3/2}}\sum_{1\le i<j\le N}\psi(X_i,X_j)
\\
&
 +
 \frac{1}{N^{5/2}}\sum_{1\le i<j<k\le N}\chi(X_i,X_j,X_k) +\dots .
\end{align*}
Let $L, Q$ and $K$ denote the first, the second and the third sum.
We call $L$ the linear part, $Q$ the
quadratic part and $K$ the cubic part
 of the decomposition.

We shall assume that the linear part does not vanish, that is, $\Var L^2>0$.
If, for large $N$,  the linear part dominates the statistic we call $\TT$
asymptotically linear.
The distribution of an asymptotically linear statistic can be
approximated by the normal distribution, via
the central limit theorem.
An improvement over the
normal approximation is obtained by using  Edgeworth
expansions for the
distribution function $\FF(x)=\P\{\TT-\E\TT\le \sigma_{\TT}x\}$. For this purpose we write
Hoeffding's decomposition in the form
\begin{equation}\label{1.1}
\TT-\E \TT=L+Q+K+R,
\end{equation}
 where
$R$ denotes the remainder.
For
a number of important examples of
asymptotically
linear statistics we have
$R/\sigma_{\TT}=o_P(N^{-1})$ (in probability) as $N\to \infty$.
Therefore,  the $U$-statistic $\sigma_{\TT}^{-1}(L+Q+K)$
can be viewed as a stochastic expansion of
$(\TT-\E\TT)/\sigma_{\TT}$ up to the order $o_P(N^{-1})$.
Furthermore, an Edgeworth expansion of $\sigma_{\TT}^{-1}(L+Q+K)$
can be used to approximate $\FF(x)$.

Introduce  the following
 two term Edgeworth expansion of the distribution function of $\sigma_{\TT}^{-1}(L+Q+K)$
\begin{eqnarray}
\label{1.2}
G(x)
&=&
\Phi(x)
-\frac{1}{\sqrt N}\frac{\kappa_3}{6}(x^2-1)\Phi'(x)
\\
\nonumber
&
-
&\frac{1}{N}
\Bigl(
\frac{\kappa_3^2}{72}(x^5-10x^3+15x)\Phi'(x)+\frac{\kappa_4}{24}(x^3-3x)\Phi'(x)
\Bigr).
\end{eqnarray}
Here $\Phi$ respectively $\Phi'$ denote the standard normal distribution
function and its derivative. Furthermore, we introduce $\sigma^2=\E g^2(X_1)$ and
\begin{align*}
\kappa_3
&
=
\sigma^{-3}\Bigl(\E g^3(X_1)+3\E g(X_1)g(X_2)\psi(X_1, X_2)\Bigr),
\\
\kappa_4
&
=\sigma^{-4}\Bigl(\E g^4(X_1)-3\sigma^4+12\E g^2(X_1)g(X_2)\psi(X_1,X_2)
\\
&
+12\E g(X_1)g(X_2)\psi(X_1,X_3)\psi(X_2,X_3)
\\
&
+4\E g(X_1)g(X_2)g(X_3)\chi(X_1,X_2,X_3)
\Bigr).
\end{align*}
Our main result, Theorem 1 below, establishes a bound
$o(N^{-1})$ for the Kolmogorov distance
$$
\Delta=\sup_{x\in \mathbb R}|\FF(x)-G(x)|.
$$

We shall consider a general situation where the kernel
$\TT=\TT^{(N)}$, the  space $(\X,\B)=(\X^{(N)},\B^{(N)})$ and the distribution
$P_X=P_X^{(N)}$  all depend on $N$ as $N\to \infty$. In order to
keep the notation simple we drop the subscript $N$ in what follows.

\bigskip

{\bf 3.1.} Let us introduce the conditions  we need in order to prove  the bound
$\Delta=o(N^{-1})$.

\bigskip

(i) {\it Moment conditions.} Assume that, for some absolute
constants $0<A_*<1$ and $M_*>0$ and numbers $r>4$ and $s>2$, we
have
\begin{eqnarray}\label{1.3}
&&
\E g^2(X_1)>A_*\sigma_\TT^2,
\quad
\E |g(X_1)|^r<M_*\sigma_\TT^r,
\\
\nonumber
&&
\E|\psi(X_1,X_2)|^r<M_*\sigma_\TT^r,
\quad
\E |\chi(X_1,X_2,X_3)|^s<M_*\sigma_\TT^s.
\end{eqnarray}

These moment conditions  refer to  the linear, the quadratic and the cubic part of $\TT$.
In order to control the remainder $R$ of the approximation (\ref{1.1}) we use moments of differences
introduced in Bentkus, G\"otze and van Zwet (1997) \cite{Bentkus_Goetze_vanZwet_1997},  
see also van Zwet (1984) \cite{vanZwet_1984} .
Define, for $1\le i\le N$,
$$
D_i\TT=\TT-\E_i\TT,
\qquad
\E_i\TT:=\E(\TT|X_1,\dots, X_{i-1},X_{i+1},\dots, X_N).
$$
A subsequent application of difference operations $D_i$, $D_j$, $\dots$,
(the indices  $i,j$, $\dots$, are all distinct) produce  higher
order  differences, like
$$
D_iD_j\TT:=D_i(D_j\TT)=\TT-\E_i\TT-\E_j\TT+\E_i\E_j\TT.
$$
For $m=1,2,3,4$ write $\Delta_m^2=\E|N^{m-1/2}D_1D_2\cdots D_m \TT|^2$.

We shall assume that for some absolute constant $D_*>0$ and number $\nu_1\in (0,1/2)$
we have
\begin{equation}\label{1.4}
\Delta_4^2/\sigma_\TT^2\le N^{1-2\nu_1}D_*
\end{equation}
For a number of important examples of asymptotically linear
statistics the moments $\Delta_m^2$ are evaluated or estimated in
\cite{Bentkus_Goetze_vanZwet_1997}. Typically we have
$\Delta_m^2/\sigma_\TT^2=O(1)$ for some $m$.
Therefore, assuming that (\ref{1.4}) holds uniformly in $N$ as $N\to \infty$,
we obtain from the inequality $\E R^2\le N^{-3}\Delta_4^2$,
see (\ref{A1.2}) (see Appendix), that $R/\sigma_\TT=O_P(N^{-1-\nu_1})$.
Furthermore, assuming that (\ref{1.3}), (\ref{1.4}) hold uniformly in $N$ as $N\to \infty$,
we obtain from (\ref{A1.2}), (\ref{A1.1}), see Appendix, that
$\sigma^2/\sigma_\TT^2=(1-O(N^{-1}))$.

(ii) {\it Cram\'er type smoothness condition.} Introduce the
function
$$
\rho(a,b)=1-\sup\{|\E\exp\{itg(X_1)/\sigma\}|: \, a\le |t|\le b
\}.
$$

We shall assume that, for some $\delta >0$ and $\nu_2>0$, we have
\begin{equation}\label{1.5}
\rho(\beta_3^{-1}, N^{\nu_2+1/2})
\ge
\delta.
\end{equation}
Here  $\beta_3=\sigma^{-3}\E |g(X_1)|^3$.

It was shown in G\"otze and van Zwet (1992) \cite{Goetze_vanZwet_1992}, see as well Theorem
1.4 of \cite{Bentkus_Goetze_vanZwet_1997}, that  moment conditions (like (\ref{1.3}) and (\ref{1.4}))
together with Cram\'er's condition (on  the summand  $g(X_i)$ of
the linear part) do not suffice to establish the bound
$\Delta=o(N^{-1})$. For convenience we state this result in
Example 1 below.
\bigskip

{\bf Example 1.} Let $X_1,X_2,\dots$ be independent random variables
 uniformly distributed on the interval $(-1/2,1/2)$. Define
$T_N=(W_N+N^{-1/2}V_N)(1-N^{-1/2}V_N)$,
where $V_N=N^{-1/2}\sum \{N^{1/2}X_j\}$ and $W_N=N^{-1}\sum [N^{1/2}X_j]$.
Here $[x]$ denotes the nearest integer to $x$ and $\{x\}=x-[x]$.

Assume that  $N=m^2$, where $m$ is odd. We have, by the local limit theorem,
$$
\P\{W_N=1\}\ge cN^{-1}
\qquad
{\text{ and}}
\qquad
\P\{|V_N|<\delta\}>c\delta,
\qquad 0<\delta<1,
$$
 where $c>0$ is  an absolute constant. From these inequalities it follows
by the independence of $W_N$ and $V_N$, that $\P\{1-\delta^2
N^{-1}\le T_N\le 1\}\ge c^2\delta N^{-1}$.

\medskip

The example defines a sequence of $U$-statistics
$\TT_N$ whose distribution functions $\FF_N$ have $O(N^{-1})$ sized
increments in a particular interval of length $o(N^{-1})$. These
fluctuations of magnitude $O(N^{-1})$ appear as a result of a
nearly lattice structure induced by the interplay between the
(smooth) linear part and the quadratic part.
 In order to avoid examples with such a (conditional) lattice structure
 a simple moment condition was introduced in G\"otze and van Zwet (1992)
 \cite{Goetze_vanZwet_1992}
 which, separates (in $L^2$ distance)
 the random variable
 $\psi(X_1,X_2)$
 from  any random variable of the form
 $\psi_h(X_1,X_2)=h(X_1)g(X_2)+g(X_1)h(X_2)$, $h-$measurable.
Note that the $L^2$ distance $\E(\psi(X_1,X_2)-\psi_h(X_1,X_2))^2$
is minimized by $h(x)=b(x)$, where
$$
b(x)=\sigma^{-2}\E\bigl(\psi(X_1,X_2)g(X_2)\bigr|X_1=x\bigr)-(\kappa/2\sigma^4)g(x).
$$
Here $\kappa=\E\psi(X_1,X_2)g(X_1)g(X_2)$.

Therefore, we assume that, for some absolute constant $\delta_{*}>0$, we have
\begin{equation}\label{1.6}
\E\Bigl(\psi(X_1,X_2)-\bigl(b(X_1)g(X_2)+b(X_2)g(X_1)\bigr)\Bigr)^2\ge
\delta_{*}^2\sigma_\TT^2.
\end{equation}

Define $\nu=600^{-1}\min\{\nu_1,\nu_2,s-2,r-4,\}$.

\begin{tm}\label{Theorem 1} Let $N\ge 4$. Assume that for some absolute constants
$A_*,M_*,D_*>0$ and numbers $r>4, s>2$, $\nu_1,\nu_2>0$ and $\delta,\delta_{*}>0$,
the conditions (\ref{1.3}), (\ref{1.4}), (\ref{1.5}), (\ref{1.6})  hold.
 Then there exists a constant $C_*>0$
 depending only
on $A_*$, $M_*$, $D_*$, $r$, $s$, $\nu_1,\nu_2,\delta, \delta_{*}$
such that
$$
\Delta\le C_*N^{-1-\nu}\bigl(1+\delta_{*}^{-1}N^{-\nu}\bigr).
$$
\end{tm}

In particular case of $U$ statistics of degree three (the case where $R\equiv 0$ in (\ref{1.1}))
the proof of Theorem 1 has been outlined in an unpublished paper by G\"otze and van Zwet (1992)
\cite{Goetze_vanZwet_1992}.
 We provide a complete and more readable version of the arguments sketched in that preprint and
 extend them to a general class of symmetric statistics.
\vskip 0.1cm
{\bf Remark 1.} Condition (\ref{1.6}) can be relaxed.  Assume that
for some absolute constant $G_*$ we have
\begin{equation}\label{1.7}
\E\Bigl(\psi(X_1,X_2)-\bigl(b(X_1)g(X_2)+b(X_2)g(X_1)\bigr)\Bigr)^2\ge
N^{-2\nu} G_*\sigma_\TT^2.
\end{equation}
The bound of Theorem 1 holds if we replace (\ref{1.6}) by this weaker
condition. In this case we have $\Delta\le C_*N^{-1-\nu}$,
where the constant $C_*$ depends on $A_*,D_*, G_*$,
$M_*,r,s,\nu_1,\nu_2, \delta$.
\vskip 0.1cm
{\bf Remark 2.}
Consider a sequence of statistics $T^{(N)}=T^{(N)}(X_{N1},\dots,
X_{NN})$ based on independent observations $X_{N1},\dots,
X_{NN}$ taking values in
$(\X^{(N)},\B^{(N)})$ and with the common distribution
$P_X^{(N)}$. Assume that conditions (\ref{1.3}), (\ref{1.4}) and (\ref{1.6}) (or
(\ref{1.7})) hold uniformly in $N=N_0,N_0+1,\dots$, for some  $N_0$.
Theorem 1 implies the bound $\Delta=o(N^{-1})$ as $N\to \infty$.
\vskip 0.1cm
{\bf Remark 3.}
 The value of $\nu=600^{-1}\min\{\nu_1,\nu_2,s-2,r-4,\}$ is far from
being optimal.
 Furthermore, the
 moment conditions (\ref{1.3}) and (\ref{1.4}) are not the weakest possible that
would ensure the approximation of order $o(N^{-1})$. The condition (\ref{1.3}) can likely
be reduced to the moment conditions that are necessary to define
 Edgeworth expansion terms $\kappa_3$ and $\kappa_4$, similarly,  (\ref{1.4}) can be reduced to
$\Delta_4^2/\sigma_\TT^2= o(N^{-1})$.
 No effort was made to obtain the result under the optimal conditions.
  This would increase
 the complexity of the proof which is already
 rather involved.

\bigskip
{\bf 4.1.} In order to compare Theorem 1 with earlier results of similar nature let
us consider the  case of $U$-statistics of degree
two
\begin{equation}\label{1.8}
\UU=\frac{\sqrt N}{2} \binom{N}{2}^{-1}\sum_{1\le i<j\le N}h(X_i,X_j),
\end{equation}
 where $h(\cdot,\cdot)$ denotes a (fixed)
symmetric kernel. Assume for simplicity of notation and without loss of
generality
that $\E h(X_1,X_2)=0$. Write $h_1(x)=\E(h(X_1,X_2)|X_1=x)$ and assume that
$\sigma_h^2>0$, where $\sigma_h^2=\E h_1^2(X_1)$.
In this case Hoeffding's decomposition (\ref{1.1}) reduces
to $\UU=L+Q$, where, by the assumption $\sigma_h^2>0$, we have $\Var L>0$.
 Since the cubic part vanishes
 we may remove the moment $\E g(X_1)g(X_2)g(X_3)\chi(X_1,X_2,X_3)$ from the
expression for $\kappa_4$. In this way we obtain the two term
Edgeworth expansion (\ref{1.2}) for the distribution function
$\FF_U(x)=\P\{\UU\le \sigma_\UU x\}$ with
$\sigma^2_\UU:=\Var \UU$.

We call a kernel $h$ reducible if
 for some measurable functions
$u,v:\X\to R$ we have $h(x,y)-\E h(X_1,X_2)=v(x)u(y)+v(y)u(x)$
for $P_X\times P_X$
almost sure $(x,y)\in \X\times \X$. A simple calculation shows that for a
sequence of $U$-statistics (\ref{1.8}) with a fixed non-reducible kernel
the condition
(\ref{1.6}) is satisfied, for some $\delta_{*}>0$, uniformly in $N$.
A straightforward consequence of Theorem 1 is the following
corollary. Write ${\tilde \nu}=600^{-1}\min\{\nu_2,r-4,1\}$.

\begin{col}\label{ Corollary 1} Let $N\ge 4$.
Assume that for some $r>4$
\begin{equation}\label{1.9}
\E |h(X_1,X_2)|^r<\infty.
\end{equation}
Assume that $\sigma_h^2>0$ and the kernel $h$ is non-reducible and that for some $\delta>0$
\begin{equation}\label{1.10}
\sup\{|\E e^{it\sigma_h^{-1}h_1(X_1)}|:\, |t|\ge \beta_3^{-1}\}
\le
1-\delta.
\end{equation}
 Then there exist a constant $C_*>0$
such that
$$
\sup_{x\in \RR}|\FF_U(x)-G(x)|\le C_*N^{-1-{\tilde \nu}}
$$
\end{col}

For $U$-statistics with fixed kernel $h$ the validity of the Edgeworth expansion (\ref{1.2}) up to
the order $o(N^{-1})$ was established by Callaert, Janssen and
Veraverbeke (1980) \cite{Callaert_Janssen_Veraverbeke_1980}
and 
Bickel, G\"otze  and van Zwet (1986) \cite{Bickel_Goetze_vanZwet_1986}.
In addition to the moment conditions (like (\ref{1.9})) and Cram\'er's
condition (like (\ref{1.10})) they  imposed  the following rather implicit
conditions which ensure the regularity of $\FF_U(x)$.
Callaert, Janssen and
Veraverbeke (1980) \cite{Callaert_Janssen_Veraverbeke_1980}
assumed that for some $0<c<1$ and
$0<\alpha<1/8$ the event
\begin{equation}\label{1.11}
\Bigl|\E\bigl(\exp\{it\sigma_\UU^{-1}\sum_{j=m+1}^Nh(X_1,X_j)\}\,\bigl|
X_{m+1},\dots, X_N\bigr)\Bigr|\le c
\end{equation}
has probability $1-o(1/N\log N)$ uniformly for all $t\in
[N^{3/4}/\log N,\, N\log N]$. Here $m\approx N^\alpha$.

Bickel, G\"otze  and van Zwet (1986) \cite{Bickel_Goetze_vanZwet_1986}
more explicitly required that the linear operator,
$f(\cdot)\to  \E \psi(X,\cdot)f(X)$ defined by $\psi$
has sufficiently large number of non-zero eigenvalues (depending on
the existing moments).

Both of these  conditions correspond to some techniques used in the
parts of our proof. Recall that in the first step of the proof,
one reduces the problem
of bounding
 $|\FF_U(x)-G(x)|$ by means of Berry-Esseen's smoothing inequality, 
  to that of bounding the difference between the corresponding Fourier
 transforms
$|{\hat \FF}_U(t)-{\hat G}(t)|$ in the region $|t|<N\e_N^{-1}$, for some $\e_N\downarrow 0$.
For $|t|\approx N$ one writes
$|{\hat \FF}_U(t)-{\hat G}(t)|\le |{\hat \FF}_U(t)|+|{\hat G}(t)|$
and bounds every summand separately. The condition (\ref{1.11}) applies more or less directly
and shows exponential decay of $|{\hat \FF}_U(t)|$ as $N\to \infty$, for $|t|\approx N$.
 The eigenvalue condition
achieves the same goal, but
 refers to a more sophisticated approach based on a symmetrization technique introduced by G\"otze (1979) \cite{Goetze_1979}.

 The condition (\ref{1.6}) provides
  an alternative to (\ref{1.11}) and the eigenvalue
 condition. It is aimed to exclude  the situations,
  where the interplay between the linear and the quadratic part produces a nearly
  lattice structure of $\UU$, which in turn results in
  fluctuations of $\FF_U(x)$
  of magnitude $O(N^{-1})$. The  proof of Theorem \ref{Theorem 1} uses a result of Kleitman 
which provides a  solution to multidimensional Littlewood-Offord problem.
This result establishes
  bounds
  for probabilities of the concentration of sums of
 random variables with values in
  multidimensional spaces.
\vskip 0.1cm
{\bf Remark 4.} The $U$-statistic (\ref{1.8})
  with a kernel $h(x,y)=v(x)u(y)+v(y)u(x)$,
where $\E u(X_1)=0$, violates (\ref{1.6}).
 Let us note that in this case one can
establish the validity  of an Edgeworth expansion with the remainder $o(N^{-1})$,
 under the following bivariate Cram\'er's condition,
which is essentially more restrictive than
condition  (ii),
$$
1
-
\sup\{|\E\exp\{i\frac{t}{\sigma}u(X_1)+i\frac{s}{\sigma}v(X_1)\}|:
\, \beta_3^{-1}<|t|\le N^{\nu+1/2},
\
|s|\le N^{\nu}\}>\delta,
$$
for some $\delta,\nu>0$. Note that from this condition we immediately obtain
the desired exponential decay of  ${\hat{\FF}}(t)$ as $N\to \infty$ for
$|t|\approx N$.

 The remaining parts of the paper (Sections 2---5) contain
the proof of Theorem 1.
Auxiliary results are placed in the Appendix.

\section {Proof of Theorem 1}

The proof
 combines various techniques developed in earlier
papers by Callaert, Janssen and Veraverbeke (1980)
\cite{Callaert_Janssen_Veraverbeke_1980}, 
Bickel,
G\"otze and van Zwet (1986) \cite{Bickel_Goetze_vanZwet_1986}. 
It is based on the manuscript of G\"otze and van Zwet (1992) \cite{Goetze_vanZwet_1992}.
The later paper introduces the
condition (\ref{1.6}), provides the crucial counter example, see Example 1 above, and contains
an outline of the proof
in the particular case of $U$-statistics of degree three
($\TT-\E\TT=L+Q+K$).
In order to extend these arguments to general symmetric statistics
we combine stochastic expansions by means of Hoeffding's
decomposition and bounds for various parts of the
decomposition. This approach was introduced in van Zwet (1984) \cite{vanZwet_1984}  and
further developed in \cite{Bentkus_Goetze_vanZwet_1997}.

\bigskip

{\bf 2.1.} Let us start with an outline of the proof.
Firstly, using the linear structure induced by Hoeffding's
decomposition we replace $\TT/\sigma_\TT$ by the statistic ${\tilde \TT}$
which is conditionally linear given $X_{m+1},\dots, X_N$.
Secondly, invoking  a smoothing inequality we pass
from distribution functions to Fourier transforms.
In the remaining steps we  bound the difference $\delta(t)=\E e^{it {\tilde \TT}}- {\hat G}(t)$,
for
$|t|\le N^{1+\nu}$. For "small frequencies" $|t|\le C N^{1/2}$, we expand the
characteristic function $\E e^{it {\tilde \TT}}$  in order to show that $\delta(t)=o(N^{-1})$.
For remaining range of frequencies, that is $C N^{1/2}\le
|t|\le N^{1+\nu}$, we bound the summands $\E e^{it {\tilde \TT}}$
and ${\hat G}(t)$ separately. The cases of "large frequencies",
that is
$N^{1-\nu}\le |t|\le N^{1+\nu}$, and  "medium frequencies", that
is
$C\sqrt N\le |t|\le N^{1-\nu}$, are treated in a different manner.
For medium frequencies the Cram\'er type condition (\ref{1.5})
ensures an
exponential decay of $|\E e^{it {\tilde \TT}}|$ as $N\to \infty$.
For large frequencies we combine conditions (\ref{1.5}) and (\ref{1.6}). Here we apply
a {\it combinatorial} concentration bound due to Kleitman as described in the introduction.

{\bf 2.2.} Before starting the  proof we  introduce some notation.
By $c_*$ we shall denote a positive constant which may depend only
on $A_*,D_*,M_*, r, s, \nu_1,\nu_2, \delta$, but it does not depend on $N$.
 In different
places the values of $c_*$ may be different.

  It
is convenient to write the decomposition in the form
\begin{equation}\label{2.1}
\TT=\E\TT+\sum_{1\le k\le N}\UU_k,
\qquad
\UU_k=\sum_{1\le i_1<\cdots<i_k\le N}g_k(X_{i_1},\dots, X_{i_k}),
\end{equation}
where, for every $k$, the symmetric kernel $g_k$ is centered, i.e.,
$\E g_k(X_1,\dots, X_k)=0$, and
satisfies, see, e.g., \cite{Bentkus_Goetze_vanZwet_1997},
\begin{equation}\label{2.2}
\E\bigl(g_k(X_1,\dots, X_k)\bigr|X_2,\dots, X_k)=0
\qquad
{\text{almost surely}}.
\end{equation}
Here we write $g_1:= N^{-1/2}g$, $g_2:= N^{-3/2}\psi$ and $g_3:= N^{-5/2}\chi$.
Furthermore, for an integer $k>0$ write $\Omega_k=\{1,\dots, k\}$.
Given a subset
$A=\{i_1,\dots, i_k\}\subset \Omega_N$ we write, for short,
$T_A:=g_k(X_{i_1},\dots, X_{i_k})$. Put $T_{\emptyset}:=\E \TT$.
Now the decomposition (\ref{2.1}) can be written as follows
$$
{\TT}
=
\E \TT+\sum_{1\le k\le N}{\mathbb U}_k
=
\sum_{A\subset \Omega_N}T_A,
\qquad
{\mathbb U}_k=\sum_{|A|=k,\, A\subset \Omega_N}T_A.
$$

\bigskip

{\bf 2.3.} {\it Proof of Theorem 1.}
Throughout the proof we assume without loss of generality that
\begin{equation}\label{2.3}
4<r\le 5,
\qquad
2<s\le 3
\qquad
{\text{and}}
\qquad
\E\TT=0,
\qquad
\sigma_\TT^2=1.
\end{equation}
Denote, for $t>0$,
$$
\beta_t=\sigma^{-t}\E|g(X_1)|^t,
\qquad
\gamma_t=\E|\psi(X_1,X_2)|^t,
\qquad
\zeta_t=\E|\chi(X_1,X_2,X_3)|^t.
$$

{\it The linearization step.}
Choose number $\nu>0$ and integer $m$ such that
\begin{equation}\label{2.4}
\nu=600^{-1}\min\{\nu_1,\, \nu_2,\, s-2, \, r-4\},
\qquad
m\approx N^{100\nu}.
\end{equation}
Split
$$
{\mathbb T}={\mathbb T}_{[m]}+{\mathbb W},
\qquad
{\mathbb T}_{[m]}=\sum_{A:\, A\cap \Omega_m\neq\emptyset}T_A,
\qquad
{\mathbb W}=\sum_{A:\, A\cap \Omega_m=\emptyset}T_A.
$$
Furthermore, write
\begin{align*}
{\mathbb T}_{[m]}
&
=
{\mathbb U}_1^*+{\mathbb U}_2^*+\Lambda,
\qquad
\Lambda=\Lambda_1+\Lambda_2+\Lambda_3+\Lambda_4+\Lambda_5,
\\
{\mathbb U}_1^*
&
=
\sum_{i=1}^mT_{\{i\}},
\qquad
{\mathbb U}_2^*
=
\sum_{i=1}^m\sum_{j=m+1}^N T_{\{i,j\}},
\\
\Lambda_1
&=
\sum_{1\le i<j\le m} T_{\{i,j\}},
\qquad
\Lambda_2
=
\sum_{|A|\ge 3, |A\cap \Omega_m|=2} T_A,
\\
\Lambda_3
&
=
\sum_{A:\, |A\cap \Omega_m|\ge 3} T_A,
\qquad
\Lambda_4=\sum_{|A|=3,\, |A\cap \Omega_m|=1}T_A,
\\
\Lambda_5
&
=
\sum_{i=1}^m\eta_i,
\qquad
\eta_i=\sum_{|A|\ge 4,\, A\cap\Omega_m=\{i\}}T_A.
\end{align*}

Before applying a smoothing inequality  we  replace
$\FF(x)$ by
$$
{\tilde \FF}(x):=\P\{{\tilde {\mathbb T}}\le x\},
\qquad{\text{ where}}
\qquad
{\tilde {\mathbb T}}
=
{\mathbb U}_1^*+{\mathbb U}_2^*+{\mathbb W}= {\mathbb T}-\Lambda.
$$

In order to  show that $\Lambda$ can be neglected we apply a simple Slutzky type argument.
 Given $\e>0$,  we have
\begin{equation}\label{2.5}
\Delta
\le
\sup_{x\in\RR}|{\tilde \FF}(x)-G(x)|
+
\e\,\sup_{x\in\RR}|G'(x)|+\P\{|\Lambda|>\e\}.
\end{equation}
From  Lemma \ref{LA1.2} we obtain via Chebyshev's
inequality, for $\e=N^{-1-\nu}$,
\begin{align*}
\P\{|\Lambda|>\e\}
&
\le
\sum_{i=1}^5\P\{|\Lambda_i|>\frac{\e}{5}\}
\\
&
\le
\bigl(\frac{5}{\e}\bigr)^3\E|\Lambda_1|^3
+
\bigl(\frac{5}{\e}\bigr)^2(\E\Lambda_2^2+\E\Lambda_3^2+\E \Lambda_5^2)
+
\bigl(\frac{5}{\e}\bigr)^s\E|\Lambda_4|^s
\\
&
\le c_*N^{-1-\nu}.
\end{align*}
In the last step we used conditions (\ref{1.3}), (\ref{1.4}) and the inequality
(\ref{A1.3}).
Furthermore, using (\ref{1.3}) and (\ref{1.4}) one can show that
\begin{equation}\label{2.6}
\sup_{x\in\RR}|G'(x)|\le c_*.
\end{equation}
 Therefore, (\ref{2.5}) implies
 $$
 \Delta\le {\tilde \Delta}+c_*N^{-1-\nu},
 \qquad
 {\tilde \Delta}:=\sup_{x\in\RR}|{\tilde \FF}(x)-G(x)|.
 $$
 It remains to show that 
 ${\tilde \Delta}\le c_*N^{-1-\nu}$.

 {\it A smoothing inequality.} Given $a>0$ and even integer $k\ge 2$
 consider the probability density function, see (10.7) in
 Bhattacharya and Rao (1986) \cite{Bhattacharya_Rao_1986},
\begin{equation}\label{2.7}
 x\to g_{a,k}(x)=a \,c(k)(ax)^{-k}\sin^k(ax),
\end{equation}
  where $c(k)$ is the normalizing constant. Its
  characteristic function
  $$
  {\hat g}_{a,k}(t)
  =
  \int_{-\infty}^{+\infty} e^{itx}g_{a,k}(x)dx
  =
  2\pi\,  a \,
  c(k)u^{*k}_{[-a,a]}(t)
 $$
 vanishes outside the interval $|t|\le ka$. Here $u^{*k}_{[-a,a]}(t)$ denotes the
probability  density function of the sum of $k$ independent random variables
 each uniformly distributed in $[-a,a]$. It is easy to show that the function
 $t\to {\hat g}_{a,k}(t)$ is unimodal and symmetric around $t=0$.

Let $\mu$ be the probability distribution with the density
$g_{a,2}$, where $a$ is chosen to satisfy $\mu([-1,1])=3/4$.
Given $T>1$ define $\mu_T({{\A}})=\mu(T{{\A}})$, for ${{\A}}\subset \RR-$ measurable. Let
${\hat \mu}_T$ denote the characteristic function corresponding to
$\mu_T$.

We apply Lemma 12.1 of Bhattacharya and Rao (1986) \cite{Bhattacharya_Rao_1986}. It follows
from (\ref{2.6}) and the identity $\mu_T([-T^{-1}, T^{-1}])= 3/4$ that
\begin{equation}\label{2.8}
{\tilde \Delta}
\le
 2\sup_{x\in \RR}\bigl|({\tilde{{{\F}}}}-{{\G}})*\mu_T(-\infty,x]\bigr|
 +
 c_*T^{-1}.
\end{equation}
Here ${\tilde {\F}}$ respectively ${\G}$ denote the probability distribution of
${\tilde \TT}$ respectively the signed measure with density
function $G'(x)$. Furthermore, $*$ denotes the convolution
operation.
Proceeding as in the proof of Lemma 12.2 ibidem we obtain
\begin{equation}\label{2.9}
({\tilde{\F}}-{\G})\ast\mu_T(-\infty,x]
=
\frac{1}{2\pi}\int_{-\infty}^{+\infty}
e^{-itx}
\Bigl( \E e^{it{\tilde \TT}}-{\hat G}(t)\Bigr)
\frac{{\hat \mu}_T(t)}{-it}dt,
\end{equation}
where ${\hat G}$ denotes the Fourier transform of $G(x)$. Note that
${\hat \mu}_T(t)$ vanishes outside the interval $|t|\le 2aT$.
Finally, we obtain from (\ref{2.8}) and (\ref{2.9}) that
\begin{equation}\label{2.10}
{\tilde \Delta}
\le
\frac{1}{\pi}
\sup_{x\in \RR}|I(x)|
+c_*\frac{2a}{T},
\qquad
I(x):=
\int_{-T}^{T}
e^{-itx}
\bigl(\E e^{ it {\tilde \TT} } - {\hat G} (t)\bigr)
\frac{{\hat \mu}_{T'}(t)}{-it}dt,
\end{equation}
where $T'=T/2a$. Here we use the fact that ${\hat \mu}_{T'}(t)=0$ for $|t|>T$.

Choose $T=N^{1+\nu}$ and denote
$K_N(t)= {\hat \mu}_{T'}(t)$. Note that $|K_N(t)|\le 1$ (since $\mu_{T'}$
is a probability measure). Write
\begin{align*}
&
|I(x)|\le c\, I_1+c \,I_2+|I_3|+|I_4|,
\\
&
I_1
=
\int_{|t|\le t_1}
\bigl|\E e^{ it {\tilde \TT} } - {\hat G}(t)\bigr|
\frac{dt}{|t|},
\quad
I_2=\int_{t_1<|t|<T}|{\hat G}(t)|\frac{dt}{|t|},
\\
&
I_3=\int_{t_1<|t|<t_2}e^{-itx}\E e^{it{\tilde \TT}}\frac{K_N(t)}{-it}dt,
\quad
I_4=\int_{t_2<|t|<T}e^{-itx}\E e^{it{\tilde \TT}}\frac{K_N(t)}{-it}dt.
\end{align*}
Here we denote $t_1=N^{1/2}10^{-3}/\beta_3$
  and $t_2=N^{1-\nu}$.

In view of (\ref{2.10}) the bound ${\tilde \Delta}\le c_*N^{-1-\nu}$
follows from the bounds
\begin{equation}\label{2.11}
|I_k|\le c_*N^{-1-\nu},
\quad
k=1,2,3,
\quad
{\text{and}}
\quad
|I_4|\le c_*N^{-1-\nu}(1+\delta_*^{-1}N^{-\nu}).
\end{equation}
The bound $I_2\le c_*N^{-1-\nu}$ is a consequence of the
exponential decay of $|{\hat G}(t)|$ as $|t|\to \infty$. For
$k=3,4$ the bound (\ref{2.11}) is shown in Section 3. The proof
of (\ref{2.11}), for $k=1$ is based on careful expansions and
is given
Section 5.

\section{Large frequencies}

Here we prove the bounds (\ref{2.11}) for $I_3$ and $I_4$. The proof of the bound
$|I_3|\le c_*N^{-1-\nu}$ is relatively simple and it is deferred to
the end of the section.

Let us show that
\begin{equation}\label{3.1}
\Bigl|
\int_{N^{1-\nu}<|t|<N^{1+\nu}}e^{-itx}\E e^{it{\tilde \TT}}\frac{K_N(t)}{-it}dt
\Bigr|
\le c_*\frac{1+\delta_{*}^{-1}}{N^{1+2\nu}}.
\end{equation}
In what follows we assume that $N$ is sufficiently large, say
$N>C_*$, where $C_*$ depends only on $A_*,D_*,M_*,r,s,\nu_1,\nu_2,\delta$.
We use this inequality in several places below, where the constant $C_*$
can be easily specified. Note that for small $N$ such that $N\le C_*$ the
inequality (\ref{3.1}) becomes trivial.

\bigskip
{\bf 3.1.} {\bf Notation.} Let us first introduce some notation.
Introduce the number
$$
\alpha=3/(r+2).
$$
For $r\in (4,5]$ and $\nu$, defined by (\ref{2.4}), we have
$$
2/r<\alpha<1/2
\qquad
{\text{and}}
\qquad
80\nu <\min\{r\alpha-2,\, 1-2\alpha\}.
$$
 Given $N$ introduce the integers
\begin{equation}\label{3.2}
n\approx N^{50\nu},
\qquad
M=\lfloor(N-m)/n\rfloor.
\end{equation}
We have $N-m=M\,n+s$, where
the integer $0\le s<n$.
Observe, that the inequalities $\nu<600^{-1}$ and  $m<N^{1/2}$, see (\ref{2.4}),
imply $M>n$. Therefore, $s<M$. Split the index set
\begin{eqnarray}\label{3.3}
&&
\{m+1,\dots, N\}=O_1\cup O_2\cup \dots\cup O_{n},
\\
\nonumber
&&
O_i=\{j:\, m+(i-1)M< j\le m+iM\},
\quad
1\le i\le n-1,
\\
\nonumber
&&
O_n=\{j:\, m+(n-1)M<j\le N\}
\end{eqnarray}
Clearly, $O_1,\dots, O_{n-1}$ are of equal size (=$M$) and
$|O_n|=M+s<2M$.

We shall assume that the random variable $X:\Omega\to \X$ is defined on the
probability space $(\Omega, P)$ and   $P_X$ is the
probability distribution on $\X$ induced by $X$.
Given $p\ge 1$ let $L^p=L^p(\X,P_X)$ denote the space of real
functions $f:\X\to\RR$ with $\E|f(X)|^p<\infty$. Denote
$\|f\|_p=(\E|f(X)|^p)^{1/p}$.
With  a random variable $g(X)$ we associate an element $g=g(\cdot)$ of
$L^p$, $p\le r$.
Let
$p_g:L^2\to L^2$
denotes the projection onto the subspace orthogonal
 to the vector $g(\cdot)$ in $L^2$.
Given $h\in L^2$, decompose
\begin{equation}\label{3.4}
 h=a_hg+h^*,
 \qquad
 a_h=\left< h,g\right>\|g\|_2^{-2},
 \qquad
 h^*:=p_g(h),
\end{equation}
 where $\left<h,g\right>=\int h(x)g(x)P_X(dx)$.
For $h\in L^r$
 we
 have
 \begin{equation}\label{3.5}
 \|h\|_r\ge \|h\|_2\ge \|h^*\|_2.
\end{equation}
Furthermore, for $r^{-1}+v^{-1}=1$ (here $r\ge 2\ge v>1$) we have
$$
|\left<h,g\right>|\le \|h\|_r\|g\|_v\le \|h\|_r\|g\|_2.
$$
In particular,
\begin{equation}\label{3.6}
|a_h|\le \|h\|_r/\|g\|_2.
\end{equation}
It follows from the decomposition (\ref{3.4}) and (\ref{3.6}) that
\begin{eqnarray}
\|h^*\|_r
&
\le
& 
\|h\|_r+|a_h|\,\|g\|_r\le \|h\|_r(1+\|g\|_r/\|g\|_2)
\\
&
=
\label{3.7}
&
 c_g\|h\|_r,
\qquad
\quad
c_g:=1+\|g\|_r/\|g\|_2.
\end{eqnarray}

Note that $c_g\le c_g^*:=1+M_*^{1/r}A_*^{-1/2}$.
Introduce the numbers
\begin{equation}\label{3.8}
a_1
=
\frac{1}{4}
\min
\bigl\{
\frac{1}{12c_g^*}, \, \frac{(c_rA_*/2^rM_*)^{1/(r-2)}}{1+4A_*^{-1/2}}
\bigr\},
\qquad
c_r=\frac{7}{24}\frac{1}{2^{r-1}}.
\end{equation}

We shall show  that there exist $\delta',\delta''>0$ depending on
$A_*, M_*, \delta$ only such that (uniformly in $N$) Cram\'er's
characteristic $\rho$, see (\ref{1.5}), satisfies
\begin{equation}\label{3.9}
\rho(a_1,2N^{-\nu + 1/2})\ge \delta',
\qquad
\rho((2\beta_3)^{-1},N^{\nu_2+1/2})\ge \delta''.
\end{equation}
We shall prove  the first inequality only.
In view of (\ref{1.5}) it suffices to show that
$\rho(a_1,\beta_3^{-1})\ge \delta'$.
Invoking the simple inequality, see, e.g., proof of (\ref{A2.11}) below,
$$
|\E e^{it\sigma^{-1}g(X_1)}|\le 1-2^{-1}t^2(1-3^{-1}|t|\beta_3)
$$
we obtain, for $|t|\le \beta_3^{-1}$,
$$
|\E e^{it\sigma^{-1}g(X_1)}|\le 1-t^2/3.
$$
Therefore, $\rho(a_1,\beta_3^{-1})\ge a_1^2/3$ and we can choose
$\delta'=\min\{\delta, a_1^2/3\}$ in (\ref{3.9}).

Introduce the constant (depending only on $A_*,M_*,\delta$)
\begin{equation}\label{3.10}
\delta_1= \delta'/(10 c_g^*).
\end{equation}
Note that $0<\delta_1<1/10$.

Given $f\in L^r$ and  $T_0\in {\mathbb R}$ such that
\begin{equation}\label{3.11}
N^{-\nu+1/2}\le|T_0|\le N^{\nu+1/2},
\end{equation}
denote
\begin{eqnarray}
\nonumber
&&
I(T_0)=[T_0, \, T_0+\delta_1N^{-\nu+1/2}],
\\
\label{3.12}
&&
\tau (f)=1-v^2(f),
\qquad
v(f)
=
\sup_{t\in I(T_0)}|u_t(f)|,
\\
\nonumber
&&
u_t(f)=\int\exp\bigr\{it\bigl(g(x)+N^{-1/2}f(x)\bigr)\bigr\}P_X(dx).
\end{eqnarray}

Given a random variable $\eta$ with values in
$L^r$ and number $0<s< 1$ define
\begin{equation}\label{3.13}
d_s(\eta,I(T_0))
=
\II_{\{v^2(\eta)>1-s^2\}}
\II_{\{\|\eta\|_r\le N^{\nu}\}},
\qquad
\delta_s(\eta,I(T_0))=\E d_s(\eta, I(T_0)).
\end{equation}

Introduce the function
\begin{equation}\label{3.14}
\psi^{**}(x,y)=\psi(x,y)-b(x)g(y)-b(y)g(x)
\end{equation}
 and the number
$$
\delta_3^2=\E|\psi^{**}(X_1,X_2)|^2.
$$
It follows from (\ref{1.6}) and our assumption that $\sigma_\TT^2=1$, see (\ref{2.3}),
 that $\delta_3^2\ge \delta_{*}^2$.

\bigskip
{\bf 3.2.} {\bf Proof of (\ref{3.1}).}
The proof of (\ref{3.1}) is rather technical and therefore will be illustrated by
an outline.
In the
first step we truncate random variables $X_{m+1},\dots, X_N$  in a
special way using conditioning  and replace them by
corresponding "truncated" random variables $Y_{m+1},\dots, Y_N$.
Correspondingly the
statistic ${\tilde {\mathbb T}}$ is, then, replaced by $T'$, see (\ref{3.20}).
In the second step we split the interval of frequencies
$N^{1-\nu}\le |t|\le N^{1+\nu}$
into non overlapping intervals $\cup_{p}J_p$ of  sizes $\approx N^{1-\nu}$
so that the integral  (\ref{3.22})
splits into the sum (\ref{3.23}).
Conditionally, given $Y_{m+1},\dots, Y_N$, the statistic $T'$ is linear in
observations $X_1,\dots, X_m$, since in 
${\tilde {\mathbb T}}$ we have removed
the higher order terms (in $X_1,\dots, X_m$) from $\mathbb T$.
Let $\E_{\mathbb Y}$ denote the conditional
expectation given $Y_{m+1},\dots, Y_N$. The  conditional characteristic
function
$\E_{\mathbb Y}\exp\{itT'\}=\alpha_t^m\exp\{itW'\}$
contains the multiplicative component $\alpha_t^m$, where
$$
\alpha_t
=
\E_{\mathbb Y}
\exp\{itN^{-1/2}g(X_1)+itN^{-3/2}\sum_{j=m+1}^N\psi(X_1,Y_j)\}
$$
and where the real valued statistic $W'$ is obtained from $W$
replacing $X_j$ by $Y_j$, for $m+1\le j\le N$.
In order to bound $|\E_{\mathbb Y}e^{itT'}|$ one would like to show
exponential decay (in $m$) of the product $|\alpha_t^m|$ using a
Cram\'er type
condition like (\ref{1.5}) above.
 For $|t|=o( N)$ (the case of medium frequencies)   the size of the quadratic part
$N^{-3/2}\sum_{j=m+1}^N\psi(X_1,Y_j)$ can be
neglected and Cram\'er's condition
implies  $|\alpha_t|\le 1-v'$ for some $v'>0$. Thus we obtain
$|\alpha_t^m|\le e^{-mv'}$.
For large
frequencies $|t|\approx  N$, the contribution of the quadratic part becomes significant and we
introduce  an extra moment condition (\ref{1.6}).
Using (\ref{1.6}), we show that, for a large set of values $t\in J_p$,  Cram\'er's
condition (\ref{1.5})  yields the desired decay of $|\alpha_t^m|$.
Furthermore, the measure of the remaining $t$ is small with high probability.

\smallskip
{\it Step 1.} {\it Truncation.}
Recall that the random variable $X:\Omega\to \X$ is defined on the
probability space $(\Omega, P)$.
Let $X'$ be an independent copy so that
$(X,X')$ is defined on $(\Omega\times \Omega', P\times P)$, where $\Omega'=\Omega$.
It follows from $\E |\psi(X,X')|^r<\infty$,
by Fubini, that for $P$ almost all $\omega'\in \Omega'$ the function
$\psi(\cdot, X'(\omega'))=\{x\to \psi(x,X'(\omega')),
\, x\in \X\}$ is an element of
 $L^r$.
Furthermore, one can define an $L^r$-valued  random variable
  $Z':\Omega'\to L^r$ such that $Z'(\omega')=\psi(\cdot,  X'(\omega'))$,
  for $P$ almost all $\omega'$.
 Consider the event ${\tilde \Omega}=\{\|Z'\|_r\le N^{\alpha}\}\subset \Omega'$
 and denote $q_N=P({\tilde \Omega})$.
Here $\|Z'\|_r=(\int|\psi(x,X'(w'))|^rP_X(dx))^{1/r}$ denotes
the $L^r$ norm of the random vector $Z'$.
Let $Y:{\tilde \Omega}\to \X$ denote the random variable $X'$
conditioned on the event ${\tilde \Omega}$. Therefore  $Y$
is defined on the probability space $({\tilde \Omega}, {\tilde
P})$, where ${\tilde P}$ denotes the restriction of $q_N^{-1}P$ to the set
${\tilde \Omega}$ and,  for every
$\omega'\in {\tilde \Omega}$, we have
$Y(\omega')=X'(\omega')$. Let $Z$ denote the $L^r-$ valued random element
$\{x\to\psi(x,\, Y(\omega'))\}$ defined on the probability space $({\tilde \Omega}, {\tilde
P})$.

We can assume that ${\mathbb X}:=(X_1,\dots, X_N)$ is a
sequence of independent copies of $X$
 defined on the probability space
$(\Omega^N, P^N)$. Let ${\overline \omega}=(\omega_1,\dots,
\omega_N)$ denote an element of $\Omega^N$. Every   $X_j$
defines random vector $Z_j'=\psi(\cdot, X_j)$
taking values in $L^r$.
Introduce  events $A_j:=\{\|Z_j'\|_r\le N^{\alpha}\}\subset \Omega^N$ and
 let ${\mathbb X}'=(X_1,\dots, X_m,Y_{m+1},\dots, Y_N)$ denote the sequence
${\mathbb X}$ conditioned on the event
$\Omega^*=\cap_{j=m+1}^NA_j=\Omega^{m}\times {\tilde \Omega}^{N-m}$. Clearly,
${\mathbb X}'({\overline \omega})={\mathbb X}({\overline \omega})$ for every
 ${\overline\omega}\in \Omega^*$ and  ${\mathbb X}'$ is
 defined on the space
$\Omega^{m}\times{\tilde \Omega}^{N-m}$ equipped with the probability
measure $P^m\times{\tilde P}^{N-m}$. In particular, the random variables
$X_1,\dots, X_m, Y_{m+1},\dots, Y_{N}$ are independent and
$Y_j$, for $m+1\le j\le N$, has the same distribution as $Y$.
Let $Z_j$ denote the $L^r-$ valued random element $\{x\to
\psi(x,Y_j)\}$, for $m+1\le j\le N$.

We are going to replace $\E e^{it{\tilde{\mathbb T}}}$ by $\E e^{itT'}$.
For $s>0$ we have almost surely
\begin{equation}\label{3.15}
1-{\mathbb I}_{A_j}\le N^{-\alpha \,s}\|Z_j'\|_r^{s},
\qquad
\|Z_j'\|_r^r=\E \bigl( |\psi(X,X_j)|^r \bigl|\, X_j\bigr).
\end{equation}
Therefore,
 by  Chebyshev's inequality, for $s=r$,
\begin{equation}\label{3.16}
0
\le
1-q_N
\le
N^{-r\alpha}\E|\psi(X,X_j)|^r
\le
N^{-r\alpha}M_*
\le
c_* N^{-2-3\nu}.
\end{equation}
We have for $k\le N$
\begin{eqnarray}
\nonumber
q_N^{-k}
&
\le
&
(1-N^{-r\, \alpha}M_*)^{-k}
\le
(1-N^{-2}M_*)^{-N}
\le
c_*,
\\
\label{3.17}
q_N^{-k}-1
&
\le
&
 kq_N^{-k}(1-q_N)\le c_*kN^{-2-3\nu}\le c_*N^{-1-3\nu}.
\end{eqnarray}
For a measurable function $f:\X^N\to {\mathbb R}$, we have
\begin{equation}\label{3.18}
\E f(X_1,\dots, X_m, Y_{m+1},\dots, Y_N)
=
\E f(X_1,\dots, X_N)
\frac
{ {\mathbb I}_{A_{m+1}}\dots {\mathbb I}_{A_N} }
{q_N^{(N-m)}}.
\end{equation}
Therefore, for
$f\ge 0$, (\ref{3.17}) and (\ref{3.18}) imply
\begin{equation}\label{3.19}
\E f(X_1,\dots, X_m, Y_{m+1},\dots, Y_N)
\le
c_*\E f(X_1,\dots, X_N)
\end{equation}
Furthermore, for
\begin{equation}\label{3.20}
T':={\tilde {\mathbb T}}(X_1,\dots, X_m,Y_{m+1},\dots, Y_N)
\end{equation}
we have, by (\ref{3.17}) and (\ref{3.18}),
\begin{eqnarray}
\nonumber
|\E e^{it(T'-x)}-\E e^{it({\tilde {\mathbb T}}-x)}|
&
\le
&
\bigl(q_N^{-(N-m)}-1\bigr)
+
\bigl(1-\P\{ A_{m+1}\cap\dots \cap A_N \}\bigr)
\\
\label{3.21}
&
=
&
(q_N^{-(N-m)}-1)+(1-q_N^{N-m})
\le
c_*N^{-1-3\nu}.
\end{eqnarray}

We replace ${\tilde {\mathbb T}}$
by $T'$ in the exponent in (\ref{3.1}). The error of this
replacement is $c_*N^{-1-2\nu}$, by (\ref{3.21}) and
 the simple inequality $|K_N(t)|\le 1$, for every $t$.
In order to prove (\ref{3.1}) we shall show that
\begin{eqnarray}\label{3.22}
&&
I
:=
\int_{N^{1-\nu}
\le
|t|\le N^{1+\nu}} \E e^{it{\hat T}}v_N(t)dt\le
c_*\frac{1+\delta_3^{-1}}{N^{1+2\nu}},
\\
\nonumber
&&
v_N(t)=t^{-1} K_N(t),
\qquad
{\hat T}=T'-x.
\end{eqnarray}
{\it Step 2.} Here we prove (\ref{3.22}).
Split the integral
\begin{equation}\label{3.23}
I=
\sum_pI_p,
\qquad
I_p=\E\int_{t\in J_p} e^{it{\hat T}}v_N(t)dt,
\end{equation}
where $\{J_p,\, p=1,2,\dots\}$ denote a sequence of
consecutive intervals of length $\approx \delta_1N^{1-\nu}$ each.
Here $\delta_1$ is a constant defined by (\ref{3.10}).
In order to prove (\ref{3.22}) we show that for every $p$,
\begin{equation}\label{3.24}
|I_p|
\le
c_*N^{-2}+c_*N^{-1-4\nu}\bigl(1+\delta_3^{-1}\bigr).
\end{equation}

Given $p$ let us prove (\ref{3.24}). Firstly, we replace $I_p$ by $\E J_*$, where
$$
J_*=\int {\mathbb I}_{\{t\in I_*\}}v_N(t)\E_{\mathbb Y}e^{it{\hat T}}dt.
$$
Here
$I_*=I_*(Y_{m+1}, \dots, Y_{N})\subset J_p$ is a random subset
defined by
\begin{equation}\label{3.25}
I_*=\{t\in J_p:\, |\alpha_t|^2>1-\e_m^2\},
\qquad
\e_m^2=m^{-1}\ln^{2}N.
\end{equation}
Since, for $t\notin I_*$, we have
$$
|\E_{\mathbb Y}e^{itT'}|
\le
|\alpha_t|^m
\le
(1-\e_m^2)^{m/2}
\le
c_*N^{-3},
$$
the error of this replacement is given by
\begin{equation}\label{3.26}
|I_p-\E J_*|\le c_*N^{-2}.
\end{equation}

Secondly, we shall show that with a high probability the set $I_*\subset J_p$
is a (random) interval. This fact and the fact that $v_N(t)$ is
monotone will be used latter to bound the integral $J_*$.
Introduce the $L^r-$ valued random element
\begin{equation}\label{3.27}
S=N^{-1/2}(Z_{m+1}+\dots+Z_N)=N^{-1/2}\sum_{j=m+1}^N\psi(\cdot, Y_j).
\end{equation}
We apply Lemma \ref{LA4.1} to the set $N^{-1/2}I_*$ conditionally on  the event
 ${\SS}=\{\|S\|_r<N^{\nu/10}\}$. This lemma shows that
 $N^{-1/2} I_*$ is an interval of size at
 most $c_*\e_m$. That is, we can write $I_*=(a_N,a_N+b_N^{-1})$
 and
\begin{equation}\label{3.28}
{\mathbb I}_{\SS} J_*={\mathbb I}_{\SS}\E_{\mathbb Y}{\tilde J}_*,
\qquad
{\tilde J}_*=\int_{a_N}^{a_N+b_N^{-1}}v_N(t)e^{it{\hat T}}dt,
\end{equation}
where the random variables (functions
of $Y_{m+1},\dots, Y_N$) satisfy
$$
a_N\in J_p
\qquad
{\text{and}}
\qquad
b_N^{-1}\le c_*\e_m\sqrt N=c_*\sqrt N m^{-1/2}\ln N.
$$
By Lemma \ref{LA5.1}, $\SS$ has at least a probability  $\P\{\SS\}\ge 1-c_*N^{-3}$. Therefore,
\begin{equation}\label{3.29}
|\E J_*-\E {\mathbb I}_{\SS} J_*|\le c_*N^{-2}.
\end{equation}
Clearly, $I_*\not=\emptyset$ if and only if
$
\alpha^2>1-\e_m^2$,
where
$$
\alpha=\sup\{|\alpha_t|:\, t\in J_p\}.
$$
Therefore,  we can write, see also (\ref{3.28}),
$$
{\mathbb I}_{\SS}J_*={\mathbb I}_{\BB}J_*={\mathbb I}_{\BB}\E_{\mathbb Y}{\tilde J}_*,
\qquad
{\text{where}}
\qquad
\BB=\{\alpha^2>1-\e_m^2\}\cap \SS.
$$
This identity together with (\ref{3.26}) and (\ref{3.29}) shows
\begin{equation}\label{3.30}
|I_p|\le |\E {\mathbb I}_{\BB}\E_{\mathbb Y}{\tilde J}_*|+c_*N^{-2}.
\end{equation}

Using the integration by parts formula we shall show that
\begin{eqnarray}
\label{3.31}
&&
|\E {\mathbb I}_{\BB}\E_{\mathbb Y}{\tilde J}_*|
\le
\frac{c} {N^{1-\nu}}
\Bigl(
\P\{\BB\}
+
\int_{b_N}^1\frac{\P\{\BB_{\e}\}}{\e^2}d\e
\Bigr),
\\
\nonumber
&
&
\BB_{\e}=\BB\cap\{|{\hat T}|\le \e\}.
\end{eqnarray}
This inequality in combination with (\ref{3.30}) and  (\ref{3.32}), see below, shows
(\ref{3.24}),
\begin{equation}\label{3.32}
\int_{b_N}^1\frac{\P\{\BB_{\e}\}}{\e^2}d\e
\le
c_*\frac{1+\delta_3^{-1}}{N^{5\nu}},
\qquad
\P\{\BB\}
\le
c_*\frac{1+\delta_3^{-1}}{N^{5\nu}}.
\end{equation}
Proof of (\ref{3.32}) is rather technical. It is given in subsection {\bf 3.3}.

Let us prove (\ref{3.31}). Firstly, we show that
\begin{equation}\label{3.33}
|{\tilde J}_*|\le c(|{\hat T}|+b_N)^{-1}a_N^{-1}.
\end{equation}
 The
 integration by parts formula shows
\begin{equation}\label{3.34}
i{\hat T}{\tilde J}_*
=
v_N(t)e^{it{\hat T}}\bigr|_{a_N}^{a_N+b_N^{-1}}
-\int_{a_N}^{a_N+b_N^{-1}}v'_N(t)e^{it{\hat T}}dt=:a'+a''.
\end{equation}
By our choice of the smoothing kernel, $v_N(t)$ is monotone
on $J_p$. Therefore,
$$
|a''|
\le
\int_{a_N}^{a_N+b_N^{-1}}|v'_N(t)|dt
=
|\int_{a_N}^{a_N+b_N^{-1}}v'_N(t)dt|
=
|v_N(a_N)-v_N(a_N+b_N^{-1})|.
$$
Invoking the simple inequality $|a'|\le|v_N(a_N)|+|v_N(a_N+b_N^{-1})|$ and using
$|v_N(t)|\le |t|^{-1}$ we obtain from
(\ref{3.34})
$$
|{\hat T}{\tilde J}_*|
\le
c
\,
\bigl(a_N^{-1}+ (a_N+b_N^{-1})^{-1}\bigr)
\le
c
\,
a_N^{-1}.
$$
For $|{\hat T}|>b_N$, this inequality implies (\ref{3.33}). For $|{\hat T}|\le b_N$ the
 inequality (\ref{3.33}) follows from the inequalities
$$
|{\tilde J}_*|
\le
\int_{a_N}^{a_N+b_N^{-1}}|v_N(t)|dt
\le
\int_{a_N}^{a_N+b_N^{-1}}\frac{c}{|t|} dt
\le
c \,
a_N^{-1}b_N^{-1}.
$$
Furthermore, it follows from (\ref{3.33}) and the
 inequality $a_N \ge N^{1-\nu}$ that
$$
|{\tilde J}_*|\le c(|{\hat T}|+b_N)^{-1}N^{-1+\nu}.
$$
Finally, we apply the inequality (which holds for arbitrary real number $v$)
$$
\frac{1}{|v|+b_N}
\le
2+2\int_{b_N}^1\frac {d\e}{\e^2}{\mathbb I}_{\{|v|\le \e\}}
$$
 to derive
 $$
 |{\tilde J}_*|
 \le
 \frac{c_*}{N^{1-\nu}}
 \bigl(
 1
 +
 \int_{b_N}^1\frac{d\e}{\e^2}{\mathbb I}_{\{|{\hat T}|\le \e\}}
 \bigr).
 $$
This inequality gives (\ref{3.31}).

\bigskip
{\bf 3.3.}
Here we prove (\ref{3.32}). The first (respectively second) inequality is proved in {\it Step A}
(respectively {\it Step B}).

\smallskip

{\it Step A.} Here we prove the first inequality of (\ref{3.32}).
 Split
 \begin{align*}
 &
 W=W_1+W_2+W_3,
 \qquad
W_1=\frac{1}{N^{1/2}}\sum_{j=m+1}^Ng(X_j),
\\
&
 W_2=\frac{1}{N^{3/2}}\sum_{m<i<j\le N}\psi(X_i,X_j),
 \qquad
 W_3=\sum_{|A|\ge 3:A\cap \Omega_m=\emptyset}T_A.
 \end{align*}
 Replacing $X_j$ by $Y_j$, for $m+1\le j\le N$, we obtain $W'=W_1'+W_2'+W_3'$.
Therefore, we can write  ${\hat T}=L+\Delta+W_3'$, where
\begin{eqnarray}\label{3.35}
&&
L
=
\frac{1}{\sqrt N}\sum_{j=1}^m g(X_j)
+
\frac{1}{\sqrt N}\sum_{j=m+1}^N g(Y_j)
-x,
\\
\nonumber
&&
\Delta
=
\frac{1}{N^{3/2}}\sum_{j=1}^m\sum_{l=m+1}^N\psi(X_j,Y_l)
+
\frac{1}{N^{3/2}}\sum_{m+1\le j<l\le N}\psi(Y_j,Y_l).
\end{eqnarray}
The inequalities
$|{\hat T}|\le \e$ and $|L|\ge 2\e$ imply $|\Delta+W_3'|>\e$. Therefore,
$$
\P\{\BB_{\e}\}
\le
\P\{\BB\cap\{|L|\le 2\e\}\, \}
+
\P\{|{\hat T}|\le \e,\, |\Delta+W_3'|\ge \e\}
=:I_1(\e)+I_2(\e).
$$
In order to prove (\ref{3.32}) we show that
\begin{equation}\label{3.36}
\int_{b_N}^1\frac{d\e}{\e^2}I_1(\e)\le c_*N^{-5\nu}(1+\delta_3^{-1}),
\qquad
\int_{b_N}^1\frac{d\e}{\e^2}I_2(\e)\le c_*N^{-5\nu}.
\end{equation}

{\it Step A.1.} Here we prove
(\ref{3.36}), for the integral over $I_2(\e)$.
We have
\begin{eqnarray}
\label{3.37}
&&
I_2(\e)\le \P\{|W_3'|>\e/2\}+I_3(\e),
\\
\nonumber
&&
I_3(\e):=\P\{|L+\Delta|<3\e/2, |\Delta|>\e/2\}.
\end{eqnarray}
It follows from (\ref{3.19}), by Chebyshev's inequality,
$
\P\{|W_3'|>\e/2\}\le c_*\e^{-2}\E W_3^2$. Furthermore, invoking the
inequalities, see (\ref{A1.2}), (\ref{A1.3}),
$$
\E W_3^2
=
\sum_{|A|\ge 3: A\cap \Omega_m=\emptyset}\E T_A^2
\le
\sum_{|A|\ge 3}\E T_A^2\le N^{-2}\Delta_3^2
\le
c_*N^{-2}
$$
we obtain from (\ref{3.37}) $I_2(\e)\le I_3(\e)+c_*\e^{-2}N^{-2}$.
Since
$$
\int_{b_N}^1\frac{d\e}{\e^2}\Bigl(\frac{1}{\e^2N^2}\Bigr)
\le
c_*
b_N^{-3}N^{-2}
\le c_*N^{-5\nu},
$$
it suffices to show (\ref{3.36}) for $I_3(\e)$.

We have, for $\Lambda_1=N^{-3/2}\sum_{1\le i<j\le m}\psi(X_i,X_j)$,
$$
I_3(\e)\le \P\{|\Lambda_1|>\e/4\}+I_4(\e),
\qquad
I_4(\e):=\P\{|L+U|<2\e,\, |U|>\e/4\},
$$
where we denote $U=\Lambda_1+\Delta$.
By Chebyshev's inequality
$$
\P\{|\Lambda_1|>\e/4\}
\le
16\e^{-2}\E\Lambda_1^2
\le
c_*\e^{-2}m^2N^{-3}.
$$
Furthermore,
$$
\int_{b_N}^1\frac{d\e}{\e^2}\Bigl(\frac{m^2}{\e^2N^3}\Bigr)
\le
c_*
b_N^{-3}m^2N^{-3}
\le c_*N^{-5\nu}.
$$
Therefore, it suffices to show (\ref{3.36}) for $I_4(\e)$.

Let $I_4'(\e)$ be the same probability as $I_4(\e)$ but with
$X_i$ replaced by $Y_i$, for $1\le i\le m$,
\begin{align*}
&
I_4'(\e)=\P\{|L'+U'|<2\e,\, |U'|>\e/4\},
\\
&
L'=\frac{1}{N^{1/2}}\sum_{1\le i\le N}g(Y_i)-x,
\qquad
U'=\frac{1}{N^{3/2}}\sum_{1\le i<j\le N}\psi(Y_i,Y_j).
\end{align*}
We have $|I_4(\e)-I_4'(\e)|\le c_*N^{-1-3\nu}$, cf. (\ref{3.21}). Since
$$
\int_{b_N}^1\frac{d\e}{\e^2}N^{-1-3\nu}\le c_*b_N^{-1}N^{-1-3\nu}\le c_*N^{-5\nu},
$$
 it
suffices to show (\ref{3.36}) for $I_4'(\e)$.

In what follows we show the bound (\ref{3.36}) for $I_4'(\e)$. Split the sample
$$
{\YY}
:=
\{Y_1,\dots,  Y_N\}
=
{\YY}_1\cup {\YY}_2\cup {\YY}_3,
$$
into
three
groups
 of nearly equal size.
Split
$U'=\sum_{i\le j}U'_{ij}$
 so that the sum $U'_{ij}$ depends on observations from
 ${\YY}_i$ and
${\YY}_j$ only.
We have
\begin{equation}\label{3.38}
I_4'(\e)
\le
\sum_{i\le j}
\P\{|L'+U'|\le 2\e,\,|U'_{ij}|\ge \e/24\}.
\end{equation}
In order to prove (\ref{3.36}) we
shall show  this bound for every summand in the
right of (\ref{3.38}).
Let ${\tilde U}$ denote a summand $U'_{ij}$, say, not
depending on ${\YY}_3$.
We shall prove (\ref{3.36}) for
\begin{eqnarray}
\label{3.39}
&&
{\tilde I}(\e)
:=
\P\{|L'+U'|\le 2\e,\, |{\tilde U}|\ge \e/24\}
=
\E{\mathbb I}_{\U}{\mathbb I}_{\V},
\\
\nonumber
&&
{\U}=\{|{\tilde U}|\ge \e/24\},
\quad
{\V}=\{|L'+U'|\le 2\e\}.
\end{eqnarray}
By the definition of $Y_1,\dots, Y_N$, the random function
$$
x
\to
{\overline S}(x)
=
N^{-1/2}
\sum_{Y_i\in {\YY}\setminus{\YY}_3}\psi(x,Y_i),
\qquad
x\in \X,
$$
defines a random variable with values in $L^r$  such that, for
every $i$,
$\|\psi(\cdot,Y_i)\|_r\le N^{\alpha}$ for almost all values of $Y_i$. An
application of Lemma \ref{LA5.1} gives
$$
\P\{\|{\overline S}(\cdot)\|_r> N^{\nu}\}\le N^{-3}.
$$
Therefore, in (\ref{3.39}) we can replace the
event $\V$ by ${\V}_1={\V}\cap\{\|{\overline S}\|_r\le N^{\nu}\}$.

Since ${\tilde U}$ does not depend on ${\YY}_3$, the
concentration bound for the conditional probability (proof of this bound is
given below)
\begin{equation}\label{3.40}
p':=\E\bigl({\mathbb I}_{{\V}_1}|{\YY}_1, {\YY}_2\bigr)
\le
c_*(\e+N^{-1/2})
\end{equation}
implies
\begin{equation}\label{3.41}
{\tilde I}(\e)
\le
c_*(\e+N^{-1/2})\P\{{\U}\}
\le
c_*(\e+N^{-1/2})\e^{-r}N^{-r/2}.
\end{equation}
In the last step we applied Chebyshev's inequality
$$
\P\{{\U}\}
\le
(24/\e)^{r}N^{-r/2}\E|N^{1/2}{\tilde U}|^r
$$
and the bound $\E|N^{1/2}{\tilde U}|^r\le c_*\E|N^{1/2}U_{ij}|^r\le c_*$, which follows from
(\ref{3.19}) and routine moment inequalities for $U$-statistics, see
Dharmadhikari, S. W., Fabian, V. Jogdeo, K. (1968).
Here the random variable $U_{ij}$ is obtained from
${\tilde U}$ after we replace $Y_j$ by $X_j$ for every $j$.

It follows from (\ref{3.41}) and the simple inequality  $\e\ge b_N\ge c_* N^{-1/2}$ that
\begin{align*}
\int_{b_N}^{1}\frac{d\e}{\e^2}{\tilde I}(\e)
&
\le \frac{c_*}{N^{r/2}}
\int_{b_N}^1\frac{d\e}{\e^{1+r}}
\le
\frac{c_*}{N^{r/2}b_N^r}
\\
&
=
c_*m^{-r/2}\ln^rN
\le
c_*
N^{-5\nu},
\end{align*}
provided that
$
m^{r/2}\ge N^{6\nu}$. The latter inequality is ensured by
(\ref{2.4}).
Thus we have shown (\ref{3.36}) for ${\tilde I}(\e)$.

It remains to prove (\ref{3.40}).
Write $L'+U'=L_*+U_*+b-x$, where
$$
L_*
=\frac{1}{N^{1/2}}\sum_{Y_j\in {\YY}_3}(g(Y_j)+N^{-1/2}{\overline S}(Y_j)),
\quad
U_*=
\frac{1}{N^{3/2}}\sum_{\{Y_j,Y_k\}\subset {\YY}_3}\psi(Y_j,Y_k),
$$
and where  $b$ is a function of
$\{Y_i\in {\YY}\setminus{\YY}_3\}$.
Introduce the random variables ${\overline L}$ and ${\overline U}$
which are obtained from $L_*$ and $U_*$ after we replace every $Y_j\in
{\mathbb Y}_3$ by the corresponding observation $X_j$.
We have
\begin{align*}
p'
&
\le \sup_{v\in R}
\E
\bigl(
{\mathbb I}_{\{L_*+U_*\in [v,v+2\e]\}}
\bigl|
{\mathbb Y}_1,
{\mathbb Y}_2
\bigr)
{\mathbb I}_{\{\|{\overline S}\|_r\le N^{\nu}\}}
\\
&
\le c_*\sup_{v\in R}
\E
\bigl(
{\mathbb I}_{\{{\overline L}+{\overline U}\in [v,v+2\e]\}}
\bigl|
 {\mathbb Y}_1,
{\mathbb Y}_2
\bigr)
{\mathbb I}_{\{\|{\overline S}\|_r\le N^{\nu}\}}.
\end{align*}
In the last step we applied (\ref{3.19}).
An application of the Berry-Esseen bound due to van Zwet (1984)
 shows (\ref{3.40}).

{\it Step A.2.} Here we prove
(\ref{3.36}) for $I_1(\e)$. Write $I_1(\e)$ in the form
\begin{align*}
I_1(\e)
&
=\E \, \II_{\AA}\II_{\SS}\II_{\WW}
\le
\E \, \II_{\AA}\II_{\VV}\II_{\WW},\\
 \AA=\{\alpha^2>1-\e_m^2\},
\quad
\VV
&
=\{\|S\|_r\le N^{\nu}\},
\quad
\WW=\{|L|<2\e\},
\end{align*}
where $\e_m$ is defined in (\ref{3.25}).
Note that, by the Berry--Esseen inequality,
\begin{equation}\label{3.42}
\P\{\WW\}\le c_*(\e+N^{-1/2}).
\end{equation}
Furthermore, one can show that the
probability of the event $\AA$ is
small, like $\P\{\AA\}=O(N^{-6\nu})$.
We are going to make use  of both of these bounds while
constructing an upper bound  for $I_1(\e)$.
Since the events $\AA$ and $\WW$ refer to the same set of random variables
$Y_{m+1},\dots, Y_N$, we cannot argue directly
that
$\E {\mathbb I}_{\AA}{\mathbb I}_{\WW}\approx \P\{\AA\}\P\{\WW\}$.
Nevertheless, invoking  a  complex
conditioning argument  we are able to  show  that
\begin{equation}\label{3.43}
I_1(\e)\le c_*{\R}(\e+N^{-1/2})+c_*N^{-2},
\qquad
{\R}:=N^{-6\nu}(1+\delta_3^{-1}).
\end{equation}
Since $\e\ge b_N>N^{-1/2}$, the inequality (\ref{3.43}) implies (\ref{3.36}).

Let us prove (\ref{3.43}). Since the proof is rather involved we start by
providing an outline,
Let the integers $n$ and $M$ be defined by (\ref{3.2}). Split
$\{1,\dots, N\}=O_0\cup O_1\cup\dots\cup O_n$, where
$O_0=\{1,\dots, m\}$ and where the sets $O_i$, for $1\le i\le n$, are defined in
(\ref{3.3}).
 Split $L$, see (\ref{3.35}),
\begin{equation}\label{3.44}
L=
\sum_{k=0}^nL_k-x,
\qquad
L_k=N^{-1/2}\sum_{j\in O_k}g(Y_j),
\qquad
{\text{for}}
\qquad
k=1,\dots, n,
\end{equation}
and $L_0=N^{-1/2}\sum_{j\in O_0}g(X_j)$.
Observe, that $\II_{\WW}$ is a function of  $L_0, L_{1},\dots, L_n$.
The random variables $\II_{\AA}$ and $\II_{\VV}$ are functions of
$Y_{m+1},\dots, Y_N$ and do not depend on $X_1,\dots, X_m$.
Therefore,
denoting
$$
m(l_1,\dots, l_n)=\E(\II_{\AA}\II_{\VV}|L_{1}=l_{1},\dots, L_n=l_n)
$$
we obtain from (\ref{3.42})
\begin{equation}\label{3.45}
\E \, \II_{\AA}\II_{\VV}\II_{\WW}
=
\E \, \II_{\WW} m(L_1,\dots, L_n)
\le
c_*(\e+N^{-1/2})\MM,
\end{equation}
where
$
\MM=ess\sup m(l_{1},\dots, l_n)$. Clearly, the  bound $\MM\le c_*{\R}$
would imply (\ref{3.43}).
Unfortunately, we are not able to establish such a bound directly.
 In what follows we prove (\ref{3.43}) using the argument outlined
 above. But we shall use  a more delicate conditioning which  allow to
estimate quantities like $\MM$.

\smallskip
{\it Step A.2.1.} Firstly we  replace $L_k$, $1\le k\le n$, by smooth random variables
\begin{equation}\label{3.46}
g_k=\frac{1}{N}\frac{\xi_k}{n^{1/2}}+L_k,
\end{equation}
where
$\xi_{1},\dots, \xi_n$
are symmetric  i.i.d. random variables with the density function
defined by (\ref{2.7}) with $k=6$ and $a=1/6$ so that
the characteristic function
$t\to \E \exp\{it\xi_1\}$ vanishes outside the
unit interval $\{t:\, |t|<1\}$.
Note that $\E \xi_1^4<\infty$.

We assume that the sequences $\xi_1,\, \xi_2, \dots $ and $X_1,\dots, X_m,Y_{m+1},\dots, Y_N$
are independent. In particular, $\xi_k$ and $L_k$ are independent.
Introduce the event
$$
{\tilde \WW}=\{| L_0+\sum_{k=1}^ng_k-x|<3\e\}.
$$
Note that
$$
\II_{\WW}
\le
\II_{\tilde \WW}
+
\II_{\{|\xi|\ge\e N\}},
\qquad
{\text{where}}
\qquad
\xi=\frac{1}{n^{1/2}}\sum_{k=1}^n\xi_k.
$$
By Chebyshev's inequality and the  inequality
$\E \xi^4\le c$,
$$
\P\{|\xi|\ge\e N\}
\le
\frac{\E\xi^4}{\e^4 N^4}
\le
\frac{c}{\e^4N^4}
\le \frac{c_*}{N^2}.
$$
Here we used the inequality $\e^2N\ge b_N^2N\ge c'_*$.
Therefore, we obtain
\begin{equation}\label{3.47}
\E\II_{\AA}\II_{\VV}\II_{\WW}
\le
\E\II_{\AA}\II_{\VV}\II_ {\tilde \WW}+c_*N^{-2}.
\end{equation}
In subsequent  steps of the proof we
 replace the conditioning on
$L_1,\dots, L_n$ (in (\ref{3.45}))  by the conditioning on the random variables
$g_1,\dots, g_n$. Since the latter random
 variables have densities (as it is shown in Lemma \ref{LA2.1} below)
the corresponding conditional distributions are much easier to handle.
Moreover,
we restrict the conditioning on the event where these densities are positive.

\smallskip
{\it Step A.2.2.} Given $w>0$, consider the events $\{|g_k|\le n^{-1/2}w\}$ and
their indicator functions $\II_k=\II_{\{|g_k|\le n^{-1/2}w\}}$. Using the simple inequality
$n\E g^2_k\le c_*$ (where $c_*$ depends on $M_*$ and $r$)  we obtain from
Chebyshev's inequality that
\begin{equation}\label{3.48}
\P\{\II_k=1\}
=
1-\P\{|g_k|>n^{-1/2}w\}
\ge
1-w^{-2}n\E|g_k|^2>7/8,
\end{equation}
where the last inequality holds for a sufficiently large constant $w$ (depending on
$M_*,\, r$).
Fix  a number $w$ such that (\ref{3.48}) holds and introduce the event
$\BB^*=\{\sum_{k=1}^{n}\II_k> n/4\}$.
Hoeffding's inequality shows $\P\{\BB^*\}\ge 1-\exp\{-n/8\}$. Therefore,
\begin{equation}\label{3.49}
\E\II_{\AA}\II_{\VV}\II_{\tilde \WW}
\le
\E\II_{\AA}\II_{\VV}\II_{\tilde \WW}\II_{\BB^*}+c_*N^{-2}.
\end{equation}

Given a binary vector $\theta=(\theta_1,\dots,\theta_{n})$ (with
$\theta_k\in \{0;1\}$) write $|\theta|=\sum_k\theta_k$.
Introduce the event $\BB_\theta=\{\II_k=\theta_k, \,
1\le k\le n\}$ and the conditional expectation
$$
m_{\theta}(z_1,\dots,z_{n})
=
\E(\II_{\AA}\II_{\VV}\II_{\BB_{\theta}}\,|\,g_1=z_1,\dots,
g_{n}=z_{n}).
$$
Note that $\II_{\BB_{\theta}}$, the indicator of
the event $\BB_{\theta}$, is a function of $g_1,\dots, g_{n}$.
%
%
It follows from the identities
$$
\BB^*=\cup_{|\theta|> n/4}\BB_{\theta}
\qquad
{\text{and}}
\qquad
\II_{\BB^*}=\sum_{|\theta|> n/4}\II_{\BB_{\theta}}
$$
(here $\BB_{\theta}\cap \BB_{\theta'}=\emptyset$, for
$\theta\not=\theta'$)
that
\begin{align*}
\E\II_{\AA}\II_{\VV}\II_{\tilde \WW}\II_{\BB^*}
&
=
\sum_{|\theta|> n/4}\E\II_{\AA}\II_{\VV}\II_{\tilde \WW}\II_{\BB_{\theta}}.
\\
&
=
\sum_{|\theta|> n/4}\E\II_{\BB_\theta}\II_{\tilde \WW}m_{\theta}(g_1,\dots, g_{n}).
\end{align*}
Assume  that we have already shown  that uniformly in $\theta$, satisfying $|\theta|> n/4$,
we have
\begin{equation}\label{3.50}
M_{\theta}\le c_*{\R},
\qquad
{\text{where}}
\qquad
M_{\theta}
:=
{\text{ess sup}}
\
m_{\theta}(z_1,\dots, z_{n}).
\end{equation}
This bound in combination with (\ref{3.42}), which extends to ${\tilde \WW}$ as
well, implies
\begin{align*}
\E\II_{\AA}\II_{\VV}\II_{\tilde \WW}\II_{\BB^*}
&
\le
c_*{\R}\sum_{|\theta|> n/4}\E\II_{\BB_\theta}\II_{\tilde \WW}
=
c_*{\R}\E\II_{\BB^*}\II_{\tilde \WW}
\\
&
\le
c_*{\R} \P\{\tilde \WW\}
\le c_*{\R}(\e+N^{-1/2}).
\end{align*}
Combining this inequality, (\ref{3.47}) and (\ref{3.49}) we obtain (\ref{3.43}).

\smallskip
{\it Step A.2.3.}
Here we show (\ref{3.50}).
Fix $\theta=(\theta_1,\dots, \theta_{n})$ satisfying
$|\theta|> n/4$. Denote, for brevity, $h=|\theta|$ and assume without loss of generality
  that
$\theta_i=1$, for $1\le i\le h$, and
$\theta_j=0$, for $h+1\le j\le n$.

Consider the $h-$dimensional random vector
${\overline g}_{[\theta]}=(g_1,\dots, g_h)$.
The  random vector ${\overline g}_{[\theta]}$ and
the sequences of random variables
$$
{\mathbb Y}_{\theta}
=
\bigl\{
Y_j:\, m+hM<j\le N
\bigr\},
\qquad
\xi_{\theta}=\{\xi_j:\, h<j\le n\}
$$
are independent.
Furthermore  the summands $S_{\theta}$
and $S'_{\theta}$ of the decomposition
$$
S=S_{\theta}+S'_{\theta},
\qquad
S_{\theta}(\cdot)=\frac{1}{\sqrt N}\sum_{1\le k\le h} \sum_{j\in O_k}\psi(\cdot,Y_j),
$$
are independent, see (\ref{3.27}). Moreover,  we have
$
m_{\theta}(z_1,\dots, z_n)
\le
{\tilde m}_{\theta}({\overline z}_{[\theta]})$,
where
$$
{\tilde m}_{\theta}({\overline z}_{[\theta]})
=
{\text{ess sup}}_{\theta}
\E
\bigl(
\II_{\AA}\II_{\VV}\II_{\BB_{\theta}}
\,
\bigr|
\,
{\overline g}_{[\theta]}={\overline z}_{[\theta]},
\,
{\mathbb Y}_{\theta},
\,
\xi_{\theta}
\bigr)
$$
denotes the  "ess sup" taken with respect to almost all  values of
${\mathbb Y}_{\theta}$ and $\xi_{\theta}$.
Here ${\overline z}_{[\theta]}=(z_1,\dots, z_h)\in {\mathbb R}^h$.
In order to prove (\ref{3.50}) we show that
\begin{equation}\label{3.51}
{\tilde m}_{\theta}({\overline z}_{[\theta]})
\le
c_*{\R}.
\end{equation}

Let us prove (\ref{3.51}). Given
${\mathbb Y}_{\theta}$,
denote $f_{\theta}=S'_{\theta}$ (note that $S'_{\theta}$ is a function of $\YY_{\theta}$).
Using the notation (\ref{3.13}), we have for the interval $J'_p=N^{-1/2}J_p$,
\begin{equation}\label{3.52}
\E
\bigl(
\II_{\AA}\II_{\VV}\II_{\BB_{\theta}}
\,
\bigr|
\,
{\overline g}_{[\theta]}={\overline z}_{[\theta]},
\,
{\mathbb Y}_{\theta},
\,
\xi_{\theta}
\bigr)
=
{\mathbb I}_{{\mathbb B}_{\theta}}
\E\bigl(
d_{\e_m}(f_{\theta}+S_{\theta}, J'_p)\,
\bigr|
\,
{\overline g}_{[\theta]}={\overline z}_{[\theta]},\, \YY_{\theta},
\,\xi_{\theta}
\bigr).
\end{equation}
Note that the factor ${\mathbb I}_{{\mathbb B}_{\theta}}$ in the right
hand side is non zero only in the case where
 ${\overline z}_{[\theta]}=(z_1,\dots, z_h)$ satisfies
$|z_i|\le w/\sqrt n$, for $i=1,\dots, h$.

Introduce the $L^r$ valued random variables
$$
U_i=N^{-1/2}\sum_{j\in O_{i}}\psi(\cdot,Y_j),
\qquad
i=1,\dots, h,
$$
and the regular conditional probability
$$
P({\overline z}_{[\theta]};{{\A}})
=
\E
\bigl(
{\mathbb I}_{\{(U_{1},\dots, U_h)\in {{\A}}\}}
\,
\bigr|
\,
{\overline g}_{[\theta]}={\overline z}_{[\theta]}%
\bigr).
$$
Here ${\A}$ denotes a Borel subset of
$L^r\times\dots\times L^r$ ($h$-times).
By independence, there exist regular conditional probabilities
\begin{equation}\label{3.53}
P_i(z_i;\,{{\A}}_i)
=
\E({\II}_{U_i\in {{\A}}_i}\,\bigr|\, g_i=z_i),
\qquad
i=1,\dots, h,
\end{equation}
such that for Borel subsets ${{\A}}_i$ of $L^r$ we have
$$
P({\overline z}_{[\theta]}; {{\A}}_1\times\cdots\times {{\A}}_h)
=
\prod_{1\le i\le h}P_i(z_i; {{\A}}_i).
$$
In particular, for every ${\overline z}_{[\theta]}$, the regular conditional
probability
$P({\overline z}_{[\theta]};\cdot)$ is the (measure theoretical)
 extension of the product of the regular
conditional probabilities (\ref{3.53}).
Therefore, denoting by
$\psi_i$ a random variable with values in $L^r$ and
with the distribution
\begin{equation}\label{3.54}
\P\{\psi_i\in {\B} \}=P_i(z_i;{\B}),
\qquad
{\B}\subset L^r - {\text{Borel set}},
\end{equation}
we obtain that the distribution of the sum
\begin{equation}\label{3.55}
\zeta=\psi_1+\dots+\psi_h
\end{equation}
of independent
 random variables $\psi_1,\dots, \psi_h$
is the regular conditional distribution of $S_{\theta}$, given
${\overline g}_{[\theta]}={\overline z}_{[\theta]}$.
In particular, the expectation in the right hand side of (\ref{3.52}) equals
$\delta_{\e_m}(f_{\theta}+\zeta)$, where
\begin{equation}\label{3.56}
\delta_s(f_{\theta}+\zeta)
:=
\E_\zeta
d_s(f_{\theta}+\zeta, J_p'),
\qquad
s>0,
\end{equation}
and where $\E_\zeta$
denotes the conditional expectation given all the random variables, but $\zeta$.
Note that the inequality
\begin{equation}\label{3.57}
\e_m\le \e_*
\end{equation}
 implies
\begin{equation}\label{3.58}
 \delta_{\e_m}(f_{\theta}+\zeta)\le \delta_{\e_*}(f_{\theta}+\zeta).
\end{equation}
We shall apply Lemma \ref{L4.1} to  construct an upper bound  for
$\delta_{\e_*}(f_{\theta}+\zeta)$, where
 $\e_*=$$\mu_* |T_0| N^{-1/2}/20$. The quantity $\mu_*$ is
defined in (\ref{4.4}) and satisfies
$c_*\delta_3^2/n\le \mu_*^2\le c_*'\delta_3^2/n$, by the inequality (\ref{A3.22}).
Note that for $T_0$ satisfying (\ref{3.11}), for some integers $m$ and $n$ as in (\ref{2.4})
and (\ref{3.2}), and for the quantity
$\delta_3$ (see (\ref{3.14})) which  satisfies
\begin{equation}\label{3.59}
\delta_3^2\ge N^{-8\nu},
\end{equation}
 the inequality (\ref{3.57}) holds with $\e_m$ defined by (\ref{3.25}),
 provided that $N$ is sufficiently large ($N>C_*$). Moreover, we have
\begin{equation}\label{3.60}
 \e_*^2\le c_*\delta_3^2N^{-48\nu}.
\end{equation}
In order to apply Lemma \ref{L4.1} we invoke the
moment inequalities of Lemma \ref{LA3.3}. Now Lemma \ref{L4.1} shows that
\begin{equation}\label{3.61}
\delta_{\e_*}(f_{\theta}+\zeta)
\le
c_*\kappa_*^{1/2}\e_*^{(r-2)/(2r)}+c_*N^{-2},
\end{equation}
where the number $\kappa_*$, defined in (\ref{4.4}), satisfies
 $\kappa_*\le c_*\delta_3^{-r/(r-2)}$, by (\ref{A3.23}).
Denote ${\tilde r}=r^{-1}+(r-2)^{-1}$.
It follows from (\ref{3.61}), (\ref{3.60}) and (\ref{3.58}), for $r>4$, that
\begin{eqnarray}\nonumber
\delta_{\e_m}(f_{\theta}+\zeta)
&&
\le
c_*\delta_3^{-{\tilde r}}N^{-6\nu}+c_*N^{-2}
\\
\label{3.62}
&&
\le
c_*(1+\delta_3^{-{\tilde r}})N^{-6\nu}
\le
c_*{\R}.
\end{eqnarray}
In the last step we used the simple bound $\delta_3^2\le c_*$, see (\ref{A3.5}), and the inequality
$1+\delta_3^{-{\tilde r}}\le 2+\delta_3^{-1}$, which follows from ${\tilde r}<1$.
Note that (\ref{3.62}) and  (\ref{3.52}), (\ref{3.56})
 imply (\ref{3.51}) thus completing the proof of the first inequality
(\ref{3.32}).

\smallskip
{\it Step B.} Here we prove the second bound of (\ref{3.32}).
It is convenient to write the $L^r$-valued random variable (\ref{3.27}) in
the form
\begin{eqnarray}\nonumber
&&
S=U_1+\dots+U_{n-1}+U_n=:S'+U_n,
\\
\label{3.63}
&&
U_i=N^{-1/2}\sum_{j\in O_i}\psi(\cdot,Y_j).
\end{eqnarray}
Observe that $U_1,\dots, U_{n-1}$ are independent and identically
distributed  $L^r$-valued random variables.

 We are going to apply
Lemma \ref{L4.1} conditionally, given $U_n$,  to the probability
$$
\P\{\BB\}=\E {\tilde p}(U_n),
\qquad
{\tilde p}(f)=\E\bigl(d_{\e_m}(S'+f,\,N^{-1/2}J_p)\bigl| U_n=f\bigr).
$$
Lemma \ref{LA3.2} shows that $U_1,\dots, U_{n-1}$ satisfy the moment
conditions of Lemma \ref{L4.1}, but now the corresponding quantity
$\mu_*$, see (\ref{4.4}),
satisfies $c_*\delta_3^2/n\le \mu^2_*\le c_*'/n$, by (\ref{A3.6}). This
implies the bound $\e_*\le c_*N^{-48\nu}$ instead of (\ref{3.60}). As a
result we obtain a different  power of $\delta_3$ in the upper
bound below.
Proceeding as in proof of (\ref{3.62}), see (\ref{3.58}), (\ref{3.60}), (\ref{3.61}),
we obtain
$$
{\tilde p}(f)\le c_*(1+\delta_3^{-r/2(r-2)})N^{-6\nu}
\le
c_*{\R}.
$$
In the last step we used the inequality $1+\delta_3^{-r/2(r-2)}\le 2+\delta_3^{-1}$.
This inequality follows from $r/2(r-2)< 1$, for $r>4$.
Therefore, we have $\P\{\BB\}\le \E {\tilde p}(U_n)\le c_*{\R}$, where
${\R}$ is defined in (\ref{3.43}).
This completes the proof of the second inequality  in (\ref{3.32}).

\bigskip
{\bf 3.4.} Here we prove the bound $|I_3|\le c_*N^{-1-\nu}$, see (\ref{2.11}).
It follows from (\ref{3.21}) that
\begin{equation}\label{3.64}
|I_3|\le \int_{t\in {\mathbb J}}\frac{\E |\alpha_t^m|}{|t|}dt+c_*N^{-1-\nu},
\end{equation}
where ${\mathbb J}=\{t:\, \sqrt N/10^3\beta_3\le |t|\le N^{1-\nu}\}$.
For the $L_r-$valued random element $S$ defined in (\ref{3.27}) and the event
${\mathbb S}=\{\|S\|_r<N^{\nu/10}\}$ write
\begin{equation}\label{3.65}
\E|\alpha_t^m|
\le
\E{\mathbb I}_{{\mathbb S}}|\alpha_t^m| +\E(1-{\mathbb I}_{\mathbb S}).
\end{equation}
Estimate the  second summand  as
$\P\{\|S\|_r\ge N^{\nu/10}\}\le c_*N^{-3}$, by
Lemma \ref{LA5.1}. Furthermore, expanding the exponent in $\alpha_t$  we obtain
\begin{equation}\nonumber
{\mathbb I}_{\mathbb S}|\alpha_t|
\le
|\E \exp\{itN^{-1/2}g(X_1)\}|+ {\mathbb I}_{\mathbb S}|t|N^{-1}\|S\|_1.
\end{equation}
It follows from (\ref{1.5}) that the first summand is bounded from above by
$1-v$,  for some $v>0$ depending on $A_*,M_*,D_*, \delta$ only,
 see the proof of (\ref{3.9}).
Furthermore,  the second summand is bounded from above by
$N^{-9\nu/10}$ almost surely. Therefore, for sufficiently large $N>C_*$
we have ${\mathbb I}_{\mathbb S}|\alpha_t|\le 1-v/2$ uniformly in $N$.
Invoking this bound in (\ref{3.65}) we obtain
$\E|\alpha_t^m|\le (1-v/2)^m+c_*N^{-3}\le c_*N^{-3}$, for $m$ satisfying (\ref{2.4}).
Finally,  we obtain that the integral in (\ref{3.64}) is bounded from above by
$c_*N^{-2}$ thus completing the proof.


\section{Combinatorial concentration bound}

We start the section by introducing some notation and collecting
auxiliary inequalities. Then we formulate and prove Lemmas \ref{L4.1}
and \ref{L4.2}.

Introduce the number
\begin{equation}\label{4.1}
\delta_2
=
\min
\bigl\{
 \frac{1}{12c_g},
\frac{(c_r\|g\|_2^2/2^r\|g\|_r^r)^{1/(r-2)}}{1+4/\|g\|_2}
\bigr\},
\end{equation}
where $c_g=1+\|g\|_r/\|g\|_2$ and $c_r=(7/24)2^{-(r-1)}$.
Denote
$$
\rho^*=1-\sup\{|\E e^{itg(X_1)}|: 2^{-1}\delta_2\le |t|\le N^{-\nu+1/2}\}.
$$
It follows from the identity
$\rho^*
=
\rho(2^{-1}\sigma \delta_2,\, \sigma N^{-\nu+1/2})$
and the simple inequality $a_1\le \delta_2/4$, see (\ref{3.8}),
that
$\rho^*
\ge
\rho(2 \sigma \,a_1, \sigma N^{-\nu+1/2})$.
Furthermore, it follows from (\ref{A1.4}) and the assumption
$\sigma_\TT^2=1$ that $1/2<\sigma<2$ for sufficiently large $N$
($N>C_*$). Therefore,
$\rho^*
\ge
\rho(a_1,2N^{-\nu+1/2})
\ge \delta'$,
where the last inequality follows from (\ref{3.9}).
We obtain, for $N>C_*$,
\begin{equation}\label{4.2}
1-\sup\{|\E e^{itg(X_1)}|: 2^{-1}\delta_2\le |t|\le N^{-\nu+1/2}\}
\ge \delta',
\end{equation}
where the number $\delta'$ depends on $A_*,D_*,M_*,\nu_1, r,s,\delta$ only.
In what follows we use the notation $c_0=10$.
Let $L_0^r= \{y\in L^r: \int_{\X} y(x)P_X(dx)=0\}$
 denotes a subspace of $L^r$.
Observe, that $\E g(X_1)=0$ implies $y^*(=p_g(y))\in L_0^r$, for
 every $y\in L_0^r$.

\bigskip
{\bf 4.1.} Let $\psi_1,\dots, \psi_n$ denote independent random vectors
with values in $L_0^r$. For $k=1,\dots, n$, write
$$
\zeta_k
=
\psi_1+\dots+\psi_k \qquad {\text{and}}
\qquad
\zeta=\zeta_n.
$$
Let $\overline \psi_i$ denote an
independent copy of $\psi_i$. Write $\psi_i^*=p_g(\psi_i)$  and
$\overline \psi_i^*=p_g({\overline \psi}_i)$, see (\ref{3.4}). Introduce  random
vectors
$$
\tilde \psi_i
=
2^{-1}(\psi_i-\overline \psi_i),
\qquad
\tilde \psi_i^*
=
2^{-1}(\psi_i^*-\overline \psi_i^*),
\qquad
\hat \psi_i
=
2^{-1}(\psi_i+\overline \psi_i).
$$
We shall assume that, for some
$c_A\ge c_D\ge c_B>0$,
\begin{equation}\label{4.3}
n^{r/2}\E\|{\tilde \psi}_i\|_r^r\le c_A^r,
\qquad
c_B^2
\le
n\, \E\|{\tilde \psi}^*_i\|_2^2
\le
c_D^2,
\end{equation}
for every $1\le i\le n$. Furthermore, denote
$\mu_i^2=\E\|{\tilde \psi}^*_i\|_2^2$
and
${\tilde \kappa}_i^{r-2}=\frac{8}{3}\frac{\E\|{\tilde \psi}_i\|_r^r}{\mu_i^r}$,
\begin{equation}\label{4.4}
\mu_*=\min_{1\le i\le n}\mu_i,
\qquad
\kappa_*=\max_{1\le i\le n}{\tilde \kappa}_i.
\end{equation}
Observe that, by H\"older's inequality and (\ref{3.5}), we have
$\kappa_i>1$, for $i=1,\dots, n$.

\begin{lem}\label{L4.1}
Let $4<r\le 5$  and $0<\nu<10^{-2}(r-4)$. Assume that $n\ge N^{5\nu}$.
 Suppose that
\begin{equation}\label{4.5}
\kappa_*^4\le\frac{9}{256}\frac{n}{\ln N}.
\end{equation}
Assume that (\ref{4.2}), (\ref{4.3}) as well as (\ref{4.13}), (\ref{4.19}) (below) hold.
There exist a constant $c_*>0$ which depends on $r,s,\nu, A_*, D_*, M_*, \delta$ only
such that
for every  $T_0$ satisfying (\ref{3.11}) we have
\begin{equation}\label{4.6}
\delta_{\e_*}(f+\zeta, I(T_0))
\le
c_*(C_D/C_B)^{1/2}\kappa_*^{1/2}{\e_*}^{(r-2)/2r}+c_*N^{-2},
\end{equation}
for an arbitrary non-random element $f\in L_0^r$. Here
$
{\e_*}=\frac{\mu_*}{2c_0}\frac{|T_0|}{\sqrt N}$. The function
$\delta_s(\cdot, I(T_0))$, is defined in (\ref{3.13}).
\end{lem}

In {\it Step A.2.3} of Section 3 we apply this lemma to random
vector $\zeta=\psi_1+\dots+\psi_h$, see (\ref{3.55}). In {\it Step B} of
Section 3 we apply this lemma to the random vector $S'$, see
(\ref{3.63}).


{\it Proof of Lemma \ref{L4.1}.}
We shall consider the case where $T_0>0$. For $T_0<0$ the proof
is the same. We can assume without loss of generality that
$c_0<N^{\nu}$.
Denote
$X=\Vert \tilde \psi_i^*\Vert_2$ and
$Y=\Vert \tilde \psi_i\Vert_r$ and $\mu=\mu_i$,
$\kappa={\tilde \kappa}_i$.
By (\ref{3.5}), we have $Y\ge X$.

{\it Step $1$.} Here we construct the bound (\ref{4.7}), see below, for the probability
$\P\{B_i\}$,
where
$$
B_i=\{X\ge \mu/2,\, Y<\kappa\mu\}.
$$
Write
\begin{align*}
&
\mu^2=\E X^2
=
\E X^2I_A+\E X^2I_{B_i}+\E X^2I_D,
\\
&
A
=\{X<\mu/2\},
\qquad
D=\{X\ge \mu/2,\, Y\ge \kappa\mu\}.
\end{align*}
Substitution of the bounds
\begin{align*}
\E X^2I_A
&
\le \frac{\mu^2}{4},
\\
\E X^2I_{B_i}
&
\le \E Y^2I_{B_i}\le (\kappa\mu)^2\P\{B_i\},
\\
\E X^2I_D
&
\le
 \E Y^2I_{\{Y\ge \kappa\mu\}}
\le
(\kappa\mu)^{2-r}\E Y^r
\end{align*}
gives
$$
\mu^2\le 4^{-1}\mu^2+\kappa^2\mu^2\P\{B_i\}+(\kappa\mu)^{2-r}\E Y^r.
$$
Finally, invoking the identity $\kappa^{r-2}=(8/3)\E Y^r/\mu^r$ we obtain
\begin{equation}\label{4.7}
\P\{B_i\}
\ge
\frac{3}{4\kappa^2}-\frac{\E Y^r}{(\kappa\mu)^{r}}
=
\frac{3}{4\kappa^2}\bigl(1-\frac{4\E Y^r}{3\mu^r\kappa^{r-2}}\bigr)
=
\frac{3}{8\kappa^2}
\ge
\frac{3}{8\kappa_*^2}=:p
\end{equation}

Introduce the  (random) set
$J=\{i:\, B_i\ {\text{occurs}}\}\subset \{1,\dots, n\}$.
 Hoeffding's inequality applied to the random variable $|J|=\II_{B_1}+\dots+\II_{B_n}$
 shows
\begin{equation}\label{4.8}
\P\{|J|\le \rho n\}
\le
 \exp\{-np^2/2\}
\le
N^{-2},
\qquad
\rho:=p/2=(3/16)\kappa_*^{-2}.
\end{equation}
In the last step we invoke  (\ref{4.5}) and use (\ref{4.7}).

{\it Step $2$}. Here we introduce randomization.
Note that for any $\alpha_i\in \{-1, +1\}$, $i=1,\dots, n$,
the distributions of the random vectors
$$
(\psi_1,\dots, \psi_n)
\qquad
{\text{ and}}
\qquad
\bigl(
\alpha_1\tilde\psi_1+\hat \psi_1,
\dots,
\alpha_n\tilde\psi_n+\hat \psi_n
\bigr)
$$
coincide.
Therefore, denoting
$$
{\tilde \zeta}_n=\alpha_1\tilde\psi_1+\dots+\alpha_n\tilde \psi_n,
\qquad
{\hat \zeta}_n={\hat \psi}_1+\dots+{\hat \psi}_n,
$$
we have for $s>0$,
$$
\delta_{s}(f+ \zeta, I(T_0))
=
 \delta_{s}(f+{\tilde \zeta}_n+{\hat \zeta}_n,
I(T_0)),
$$
for every choice of $\alpha_1,\dots, \alpha_n$. From now on let
$\alpha_1,\dots, \alpha_n$ denote a sequence of independent identically distributed
 Bernoulli random variables independent of ${\tilde \psi}_i, {\hat \psi}_i$, $1\le i\le n$,
and with probabilities $\P\{\alpha_1=1\}=\P\{\alpha_1=-1\}=1/2$. Denoting by
$\E_\alpha$
the expectation with respect to the sequence $\alpha_1,\dots,
\alpha_n$ we obtain
\begin{equation}\label{4.9}
\delta_{s}(f+\zeta,I(T_0))
=
\E_{\alpha}\delta_{s}(f+{\tilde \zeta}_n+{\hat \zeta}_n,I(T_0)).
\end{equation}
We are going to condition on ${\tilde \psi}_i$ and ${\hat
\psi}_i$, $1\le i\le n$, while taking expectations with respect to $\alpha_1,\dots,
\alpha_n$.
It follows from (\ref{4.8}), (\ref{4.9}) and the fact that the random variable $|J|$ does not
depend on $\alpha_1,\dots, \alpha_n$ that
\begin{equation}\label{4.10}
\delta_{s}(f+\zeta, I(T_0))
\le
\E\II_{\{|J|\ge \rho n\}}
\gamma_s({\tilde \psi}_i,\,  {\hat \psi}_i, \, 1\le i\le n)+N^{-2},
\end{equation}
where
$$
\gamma_s( {\tilde \psi}_i,\,  {\hat \psi}_i, \, 1\le i\le n)
=
\E_{\alpha}\II_{\{|J|\ge \rho n\}}
\II_{ \{ v^2(f+{\tilde \zeta}_n+{\hat \zeta}_n) > 1-s^2    \} }
\II_{ \{ \|f+{\tilde \zeta}_n+{\hat \zeta}_n\|_r\le N^{\nu} \} }
$$
denotes the conditional expectation given ${\tilde \psi}_i,\, {\hat \psi}_i$, $1\le i\le n$.
Note that (\ref{4.6}) is a consequence of (\ref{4.10}) and of the bound
\begin{equation}\label{4.11}
\gamma_{\e_*}({\tilde \psi}_i,\,  {\hat \psi}_i, \, 1\le i\le n)
\le
c_*\kappa_*^{1/2}\e_*^{(r-2)/(2r)}.
\end{equation}

Let us prove this bound.
Introduce the integers
$$
n_0=l-1
\qquad
\qquad
l=\lfloor\delta_2\varkappa^{-1}\e_*^{-(r-2)/r}\rfloor,
\qquad
\qquad
\varkappa=2c_0(C_D/C_B)\kappa_*.
$$
Let us show that
\begin{equation}\label{4.12}
n_0\le \rho n.
\end{equation}
It follows from the inequalities
$$
\e_*^{-1}\le 2\frac{c_0}{c_B}N^{ \nu}n^{1/2},
\quad
N^{\nu(r-2)/r}\le N^{\nu}\le n^{1/r},
\quad
\delta_2\le \frac{3}{16}\bigl(\frac{3}{8}\bigr)^{1/(r-2)}
$$
that
$$
l
\le
\frac{\delta_2}{C_D}\frac {1}{k_*}
\bigl(\frac{C_B}{2c_0}\bigr)^{2/r}
\bigl(N^{\nu}n^{1/2}\bigr)^{(r-2)/r}
\le
\frac{3}{16}\frac{1}{k_*}\frac{C_B^{2/r}}{C_D}n^{1/2}.
$$
Note that (\ref{4.5}) implies $k_*\le n^{1/4}$. Therefore, the
inequality
\begin{equation}\label{4.13}
C_B^{2/r}C_D^{-1}\le n^{1/4}
\end{equation}
implies $l\le (3/16)k_*^{-2}n=\rho n$. We obtain (\ref{4.12}).

Given ${\tilde \psi}_i,\, {\hat \psi}_i$, $1\le i\le n$,
consider the corresponding set $J$, say $J=\{i_1,\dots, i_k\}$.
Assume that  $k\ge \rho n$. From the  inequality $\rho n\ge n_0$, see
(\ref{4.12}), it follows
that we can choose a subset $J'\subset J$ of size $|J'|=n_0$.
Split
$$
{\tilde \zeta}_n
=
\sum_{i\in J'}\alpha_i{\tilde \psi}_i
+
\sum_{i\in J\setminus J'}\alpha_i{\tilde \psi}_i
=:\zeta_*+\zeta'
$$
and denote $f+\zeta'+{\hat \zeta}_n=f_*$. Note that $f_*\in L_0^r$ almost surely.
The bound (\ref{4.11}) would follow if we show that
\begin{eqnarray}\label{4.14}
{\tilde \delta}
&
\le
& 
c_*\kappa_*^{1/2}\e_*^{(r-2)/(2r)},
\\
\nonumber
{\tilde\delta} & :=
&
\E' \II_{ \{ v^2(f_*+\zeta_*) > 1-\e_*^2    \}}
 \II_{ \{ \|f_*+\zeta_*\|_r\le N^{\nu} \} }.
\end{eqnarray}
Here $\E'$ denotes the conditional expectation given all the random
variables,
\linebreak
 but $\{\alpha_i,\, i\in J'\}$.
 
\smallskip
{\it Step $3$}. Here we prove (\ref{4.14}).
Note that for $j\in J'$ the vectors
$$
x_j=T_0N^{-1/2}{\tilde \psi}_j
\qquad
{\text{ and}}
\qquad
x_j^*=p_g(x_j)=T_0N^{-1/2}{\tilde \psi}_j^*
$$
satisfy
\begin{equation}\label{4.15}
\Vert x_j^*\Vert_2\ge c_0\e_*,
\qquad
\Vert x_j\Vert_r\le \varkappa\e_*,
\qquad
\varkappa=2c_0(C_D/C_B)\kappa_*.
\end{equation}
Given $A\subset J'$ denote
$$
x_A=\sum_{i\in A}x_i-\sum_{i\in J'\setminus A}x_i,
\qquad
x_A^*=p_g(x_A).
$$
 We are going to
apply Kleitman's theorem on  symmetric partitions
(see, e.g. the proof of Theorem 4.2, Bollobas (1986))
to the sequence $\{x_j^*,\, j\in J'\}$ in $L^2$.
Since for $j\in J'$ we have
$\Vert x_j^*\Vert_2\ge c_0\e_*$, it follows from Kleitman's theorem
that the collection ${\cal P}(J')$ of all subsets of $J'$ splits into
non-intersecting non-empty classes
${\cal P}(J')={\cal D}_1\cup\cdots\cup{\cal D}_s$, such that the
corresponding sets of linear combinations
$
V_t
=
\bigl\{
x^*_A,\, A\in {\cal D}_t
\bigr\}$,
$t=1,2,\dots, s$,
are sparse, i.e., given $t$, for $A,A'\in {\cal D}_t$ and $A\not=A'$ we have
\begin{equation}\label{4.16}
\|x_A^*-x_{A'}^*\|_2\ge c_0\e_*.
\end{equation}
Furthermore, the
number of classes $s$ is bounded from above by $\binom{n_0} {\lfloor n_0/2\rfloor}$.


Next, using Lemma \ref{L4.2} we shall show that  given $f_*$ the class ${\cal D}_t$
 may contain at most
one element $A\in {\cal D}_t$  such that
\begin{equation}\label{4.17}
v^2(f_*+{\tilde x}_A)> 1-\e_*^2,
\qquad
\|f_*+{\tilde x}_A\|_r\le N^{\nu},
\qquad
{\tilde x}_A:=N^{1/2}T_0^{-1}x_A.
\end{equation}
This means that there are at
most $\binom{n_0}{\lfloor n_0/2\rfloor}$ different subsets $A\subset J'$
for which (\ref{4.17}) holds. This implies (\ref{4.14})
$$
{\tilde \delta}
\le
2^{-n_0}\binom{n_0}{\lfloor n_0/2\rfloor}\le cn_0^{-1/2}
=
c\delta_2^{-1/2} \varkappa^{1/2}\e_*^{\frac{r-2}{2r}}.
$$
Finally, (\ref{4.6}) follows  from (\ref{4.10}), (\ref{4.11}),  (\ref{4.14}).

Given $f_*\in L_0^r$ let us show that
there is no pair $A,\,A'$ in ${\cal D}_t$ which satisfy (\ref{4.17}).
Fix $A,A'\in {\cal D}_t$. We have, by (\ref{4.15}) and the choice of $n_0$,
$$
\|x_A-x_{A'}\|_r
\le
2\sum_{i\in J'}\|x_i\|_r
\le
2n_0\varkappa\e_*<2\delta_2\e_*^{2/r}.
$$
Denoting $S_{A}=f_*+{\tilde x}_A$ and
$S_{A'}=f_*+{\tilde x}_{A'}$ we obtain
\begin{equation}\label{4.18}
\|S_A-S_{A'}\|_r
=
N^{1/2}T_0^{-1}\|x_A-x_{A'}\|_r
\le
2\delta_2\e_*^{2/r}N^{1/2}T_0^{-1}.
\end{equation}
Assume that $S_A$ and $S_{A'}$ satisfy the second inequality of (\ref{4.17}), i.e.,
$\|S_A\|_r\le N^{\nu}$ and $\|S_{A'}\|_r\le N^{\nu}$.
We are going to apply Lemma \ref{L4.2} to the vectors $S_A$ and $S_{A'}$.
In order to check the conditions of
Lemma \ref{L4.2} note that  (\ref{4.21}) and (\ref{4.22}) are verified by
(\ref{4.15}), (\ref{4.16}) and (\ref{4.18}). Furthermore, the inequalities $c_0<N^{\nu}$  and
\begin{equation}\label{4.19}
c_B\ge 2N^{4\nu}(n/N)^{1/2},
\end{equation}
 imply
$ N^{2\nu-1/2}\le \e_*$. Finally, we can assume without loss
of generality that $\e_*\le c_*':=\min\bigl\{( \delta'/4)^{r/2}, (A_*^{1/2}/6)^{r/2}\bigr\}$.
Otherwise (\ref{4.6})
follows from trivial inequalities
$\delta_{\e_*}\le 1\le (\e_*/c_*')^{(r-2)/2r}\le c_*\e_*^{(r-2)/2r}$
and the inequality $\kappa_*>1$.

Now Lemma \ref{L4.2} shows that
$\min\{v^2(S_A), \, v^2(S_{A'})\}\le 1-\e_*^2$ thus completing the
proof of Lemma \ref{L4.1}.

%

\bigskip
{\bf 4.2.} Here we formulate and prove Lemma \ref{L4.2}. Let us introduce
first
some notation. Given $y\in L^r(=L^r(\X,P_X))$ define the
symmetrization $y_{s}\in L^r(\X\times\X,P_X\times P_X)$ by
$y_{s}(x,x')=y(x)-y(x')$, for $x,x'\in \X$. In what follows $X_1,X_2$ denote independent random
variables with values in $\X$ and with the common distribution $P_X$. By $\E$ we denote the
 expectation taken with respect to $P_X$. 
 For $h\in L^r$ we write
 $$
 \E h=\E h(X_1)=\int_{\X} h(x)P_X(dx),
 \qquad
 \E e^{ih}=\E e^{ih(X_1)}=\int_{\X} e^{ith(x)}P_X(dx).
 $$
 Furthermore, for $2\le p\le r$, denote
 $$
 \|y_s\|_p^p=\E|y(X_1)-y(X_2)|^p,
 \qquad
 \|y\|_p^p=\E|y(X_1)|^p.
 $$

Note that for $y\in L_0^r$ we have $y^*(=p_g(y))\in L_0^r$ and,
therefore,
\begin{equation}\label{4.20}
 \E|y^*(X_1)-y^*(X_2)|^2=2\E|y^*(X_1)|^2.
\end{equation}

Let $y_1,\dots, y_k, f$ be non-random vectors in $L^r$.
We shall assume that these vectors belong to the linear subspace
$L_0^r$.
 Given non random vectors
$\alpha=\{\alpha_i\}_{i=1}^k$ and $\alpha'=\{\alpha'_i\}_{i=1}^k$,
with $\alpha_i, \alpha'_i\in \{-1, +1\}$,
denote
$$
S_{\alpha}=f+\sum_{i=1}^k\alpha_iy_i,
\qquad
S_{\alpha'}=f+\sum_{i=1}^k\alpha'_iy_i.
$$

\bigskip
\begin{lem}\label{L4.2}
Let $\varkappa>0$.
Assume that (\ref{4.2}) holds and suppose that
$$
N^{\nu-1/2}
\le
\e
\le
\min\bigl\{( \delta'/4)^{r/2}, (\|g\|_2/6)^{r/2}
\bigr\}.
$$
Given $T_0$, satisfying (\ref{3.11}), write $T^*=N^{1/2}T_0^{-1}$ and assume that
\begin{equation}\label{4.21}
\Vert y_j^*\Vert_2>c_0T^*\e,
\quad
\Vert y_j\Vert_r\le \varkappa\, T^*\e,
\quad
j=1,\dots, k.
\end{equation}
Suppose that $\Vert S_{\alpha}\Vert_r\le N^{\nu}$ and $\Vert S_{\alpha'}\Vert_r\le
N^{\nu}$ and
\begin{equation}\label{4.22}
\Vert S^*_{\alpha}-S^*_{\alpha'}\Vert_2\ge c_0T^*\e,
\qquad
\|S_{\alpha}-S_{\alpha'}\|_r
\le
2\delta_2T^*\e^{2/r}.
\end{equation}
Then
$\min\{v^2(S_{\alpha}),\, v^2(S_{\alpha'}\}\le 1-\e^2$.
\end{lem}
Recall that the functionals $\tau(\cdot)$, $|u_t(\cdot)|$ and the interval
$I=I(T_0)$ used in proof below are defined in (\ref{3.12}).

{\it Proof.}
Note that $\delta_1<1/10$ and $\delta_2<1/12$.
In particular, we have
\begin{equation}\label{4.23}
9/10\le  1-\delta_1\le |s/T_0|\le 1+\delta_1\le 11/10,
\quad
{\text{for}}
\quad
|s-T_0|<\delta_1N^{-\nu+1/2}.
\end{equation}

\bigskip
{\bf Step 1.}
Assume that the inequality $\min\{v^2(S_{\alpha}),\, v^2(S_{\alpha'}\}\le
1-\e^2$ fails. Then for some $s,t\in I$ we have
\begin{equation}\label{4.24}
1-|u_t(S_{\alpha})|^2< \e^2,
\qquad
1-|u_s(S_{\alpha'})|^2< \e^2,
\end{equation}
see also (\ref{3.12}). Fix these $s,t$ and denote
$$
{\tilde X}=s(g+N^{-1/2}S_{\alpha'})-t(g+N^{-1/2}S_{\alpha}).
$$

 We are going to apply the inequality (\ref{A4.4}),
$$
1-|\E e^{i(Y+Z)}|^2
\ge
2^{-1}(1-|\E e^{iZ}|^2)
-
(1-|\E e^{iY}|^2)
$$
to  $Z=-{\tilde X}$ and $Y=s(g+N^{-1/2}S_{\alpha'})$.
It follows from this inequality and (\ref{4.24}) that
$$
\e^2
>
1-|u_t(S_{\alpha})|^2
=
1-|\E e^{i(Y+Z)}|^2
\ge
2^{-1}(1-|\E e^{-i{\tilde X}}|^2)-\e^2.
$$
In view of the identity $|\E e^{-i{\tilde X}}|=|\E e^{i{\tilde X}}|$ we have
\begin{equation}\label{4.25}
1-|\E e^{i{\tilde X}}|^2< 4\e^2.
\end{equation}

\bigskip
{\bf Step 2.}
Here we shall show  that (\ref{4.25}) contradicts the second inequality of (\ref{4.22}).
Firstly, we collect some auxiliary inequalities.
Write the decomposition (\ref{3.4}) for $S_{\alpha}$ and $S_{\alpha'}$,
\begin{equation}\label{4.26}
S_{\alpha}=a\,g+S_{\alpha}^*,
\qquad
S_{\alpha'}=a'\,g+S_{\alpha'}^*.
\end{equation}
Decompose
\begin{align*}
&
{\tilde X}
=vg+h,
\\
&
v=(s-t)(1+a\,N^{-1/2})+(a'-a)sN^{-1/2},
\\
&
h=
(s-t)N^{-1/2}S_{\alpha}^*+sN^{-1/2}(S_{\alpha'}^*-S_{\alpha}^*),
\end{align*}
where $v\in {\mathbb R}$ and where $h\in L^r$ is $L^2$-orthogonal to
$g$.  An application of (\ref{3.7}) to $S_{\alpha}^*$ and $S_{\alpha'}^*-S_{\alpha}^*$
gives
\begin{equation}\label{4.27}
\|h\|_r
\le
c_g
N^{-1/2}
\bigl(
|s|\, \|S_{\alpha'}-S_{\alpha}\|_r+|s-t|\,\|S_{\alpha}\|_r
\bigr).
\end{equation}
Furthermore, it follows from the simple  inequality
$$
\Vert x+y\Vert_2^2\ge 2^{-1}\Vert x\Vert_2^2-\Vert y\Vert_2^2
$$
that
\begin{equation}\label{4.28}
\Vert h\Vert_2^2
\ge
2^{-1}s^{2}N^{-1}\Vert S_{\alpha'}^*-S_{\alpha}^*\Vert_2^2
-
(s-t)^2N^{-1}\Vert S_{\alpha}^*\Vert_2^2.
\end{equation}

Note that for $a$ and $a'$ defined in (\ref{4.26}) we obtain from (\ref{3.6}) and (\ref{4.22}) that
\begin{eqnarray}
\label{4.29}
&&
|a|
\le
\Vert S_{\alpha}\Vert_r \Vert g\Vert_2^{-1}
\le
 N^{\nu}\|g\|_2^{-1},
\\
\label{4.30}
&
&
|a'-a|
\le
\Vert S_{\alpha'}-S_{\alpha}\Vert_r \Vert g\Vert_2^{-1}
\le
2 \delta_2 \e^{2/r} N^{1/2}T_0^{-1}\|g\|_2^{-1}.
\end{eqnarray}

\bigskip
{\bf Step 4.2.1.} Consider the case where,
$
|s-t|<\delta_2$.
Invoking the inequalities $\Vert S_{\alpha}\Vert_r\le N^{\nu}$ and
(\ref{4.22}) we obtain from (\ref{4.27}) that
$$
\Vert h\Vert_r^r
\le
(4c_g)^r \delta_2^r\,
\Bigl(
N^{\nu r-r/2}+\e^2|s|^rT_0^{-r}
\Bigr).
$$
Furthermore,
using (\ref{4.23}), (\ref{4.1}),
and $N^{\nu-1/2}\le \e$,
we obtain for $4\le r\le 5$
\begin{equation}\label{4.31}
\Vert h\Vert_r^r\le
3^{-r}(\e^r+\e^2(11/10)^r)
\le
3^{1-r}\e^2.
\end{equation}

Note that (\ref{3.5}) implies $\|S_{\alpha}^*\|_2\le \|S_{\alpha}\|_r\le N^{\nu}$.
This inequality in combination
with (\ref{4.22}) and (\ref{4.28})
gives
$$
\Vert h\Vert_2^2
\ge
2^{-1}(s/T_0)^2 c_0^2 \e^2-\delta_2^2N^{2\nu-1}.
$$
Invoking (\ref{4.23})
and using
$c_0>10$,$\delta_2<12^{-1}$, and $N^{\nu-1/2}\le \e$ we obtain
\begin{equation}\label{4.32}
\Vert h\Vert_2^2\ge (4/10)c_0^2\e^2.
\end{equation}

Now we are going to apply Lemma \ref{LA4.1} statement {\bf a)} to ${\tilde X}=vg+h$.
For this purpose we verify conditions of the lemma.
Firstly, note that
(\ref{4.32}), (\ref{4.20}) imply, $\|h_s\|_2^2\ge (8/10)c_0^2\e^2$. Furthermore, it follows from the
 simple inequality $\E|h(X_1)-h(X_2)|^r\le 2^r\E|h(X_1)|^r$ and (\ref{4.31}) that
 $\|h_s\|_r^r\le 3(2/3)^r\e^2$. Therefore, we obtain, for $4\le r\le 5$,
$$
\Vert h_s\Vert_r^r
\le
\frac{6}{10}\e^2
\le c_0^{-2}\|h_s\|_2^2\le c_r \Vert h_s\Vert_2^2,
\qquad
c_r=(7/24)2^{-(r-1)}.
$$
Furthermore,
the inequalities  (\ref{4.29}), (\ref{4.30}) and (\ref{4.23})
 imply
$$
|v|
\le
\delta_2+\delta_2\|g\|_2^{-1}(N^{\nu-1/2}
+
 2\e^{2/r}(11/10))
\le
\delta_2(1+4\|g\|_2^{-1}),
$$
for $N^{\nu-1/2}\le \e\le 1$.
Invoking (\ref{4.1}) and using the inequality $\|g_s\|_r^r\le 2^r\|g\|_r^r$
and the identity $\|g_s\|_2^2=2\|g\|_2^2$
 we obtain
 $$
 |v|^{r-2}
 \le
 \frac{c_r}{2^r}\frac{\Vert g\Vert_2^2}{\Vert g\Vert_r^{r}}
 \le
\frac{c_r}{2^r}\frac{2^{-1}\|g_s\|_2^2}{2^{-r}\|g_s\|_r^r}
\le
\frac{c_r}{2}\frac{\|g_s\|_2^2}{\|g_s\|_r^r}
$$
as required by Lemma \ref{LA4.1} {\bf a)}.
This lemma implies
$$
1-|\E e^{i{\tilde X}}|^2
\ge
6^{-1}\Vert h_s\Vert_2^2=3^{-1}\|h\|_2^2.
$$
In the last step we used (\ref{4.20}). Now (\ref{4.32}), for $c_0\ge 10$,  contradicts (\ref{4.25}).

\bigskip
{\bf Step 4.2.2.} Consider the case where
$\delta_2<|s-t|\le \delta_1 N^{-\nu+1/2}$.

It follows from (\ref{4.27}), (\ref{4.22})  and (\ref{4.23})
 that
\begin{eqnarray}
\nonumber
\E|h|
\le
\Vert h\Vert_r
&
\le
&
c_g \bigl(2\delta_2\e^{2/r}|s/T_0|+\delta_1\bigr)
\\
\label{4.33}
&
\le
&
c_g(\delta_1+3\delta_2\e^{2/r})
\le
c_g\delta_1+\e^{2/r}.
\end{eqnarray}
In the last step we used $\delta_2<1/3$.
From (\ref{4.29}), (\ref{4.30}) and
(\ref{4.23}), we obtain for $\delta_2\le |s-t|$ and $N^{\nu-1/2}\le \e$,
\begin{align*}
|v|
&
\ge
\delta_2 (1-N^{\nu-1/2}\|g\|_2^{-1})
-
2\delta_2\e^{2/r}|s/T_0|\|g\|_2^{-1}
\\
&
=
\delta_2(1-\|g\|_2^{-1}(\e+\e^{2/r}(22/10)))
\\
&
\ge
\delta_2(1-3\e^{2/r}\|g\|_2^{-1})\ge \delta_2/2,
\end{align*}
provided that $\e^{2/r}<\|g\|_2/6$. Similarly, using in addition,
$\delta_1, \delta_2<1/4$ and $\e<\|g\|_2$, we obtain,
 for
 $|s-t|\le \delta_1N^{-\nu+1/2}$,
\begin{align*}
|v|
&
\le
|s-t|(1+N^{\nu-1/2}\|g\|_2^{-1})
+
2\delta_2\e^{2/r}|s/T_0|\|g\|_2^{-1}
\\
&
\le |s-t|(1+\e\|g\|_2^{-1})+(22/10)\delta_2\e^{2/r}\|g\|_2^{-1}
\\
&
\le
2\,|s-t|+1
\le
N^{-\nu+1/2}.
\end{align*}
It follows from these inequalities, see (\ref{4.2}), that
$$
1-|\E e^{i{\tilde X}}|^2
\ge
1-|\E e^{i{\tilde X}}|
\ge
1-|\E e^{ivg}|-\E |h|
\ge
\delta'-\E|h|.
$$
Finally, invoking
(\ref{4.33}) and (\ref{3.10}), we get
$$
1-|\E e^{i{\tilde X}}|^2\ge
\delta'
-c_g\delta_1-\e^{2/r}
\ge
\delta'/2
>
4\e^2,
$$
Once again we obtain a contradiction to (\ref{4.25}), thus completing
the proof.

\section{ Expansions}

Here we prove the bound
\begin{equation}\label{5.1}
\int_{|t|\le t_1}\Bigl|\E e^{it{\tilde {\mathbb T}}}-{\hat G}(t)\Bigr|\frac{dt}{|t|}
\le c_*N^{-1-\nu},
\end{equation}
 where $t_1=N^{1/2}/10^3\beta_3$. For the definition of ${\tilde {\mathbb T}}$ and
${\hat G}$ see section {\bf 2.1.} Here and below $c_*$
denotes a constant depending on $A_*,M_*,D_*,r,s,\nu_1$ only.
We  prove (\ref{5.1}) for sufficiently large $N$, that is, we shall assume
that $N>C_*$, where $C_*$ is a number depending on $A_*,M_*,D_*,r,s,\nu_1$
only. Note that for $N<C_*$, the bound (\ref{5.1}) becomes trivial, since in this case
 the integral is bounded by a constant.

Let us first  introduce some notation. Denote $\Omega_m=\{1,\dots, m\}$.
For $A\subset \Omega_N$ write ${\mathbb U}_1(A)=\sum_{j\in A}g_1(X_j)$.
Given complex valued functions $f, h$ we write $f\prec \R$ if
$$
\int_{|t|\le t_1}|t^{-1}f(t)|dt\le c_*N^{-1-\nu}$$ 
and write
$f\sim h$ if $f-h\prec \R$.
In particular, (\ref{5.1}) can be written in short
$\E e^{it{\tilde {\mathbb T}}}\sim{\hat G}(t)$.

In order to prove (\ref{5.1}) we show that
\begin{equation}\label{5.2}
\E e^{it{\tilde {\mathbb T}}}\sim \E e^{it{\mathbb T}}
\qquad
{\text{and}}
\qquad
\E e^{it{\mathbb T}}\sim {\hat G}(t).
\end{equation}
In what follows we use the notation of Section 2 and assume that (\ref{2.3}) holds.

\bigskip
{\bf 5.1.} Let us prove the first part of (\ref{5.2}).
Write
 $$
{\mathbb T}={\tilde {\mathbb T}}+{\tilde \Lambda}_1+{\tilde \Lambda}_2,
\qquad
{\tilde \Lambda}_1=\Lambda_1+\Lambda_4,
\qquad
{\tilde \Lambda}_2=\Lambda_2+\Lambda_3+\Lambda_5,
$$
where the random variables  $\Lambda_j$ are introduced in {\bf 2.3}.
We shall show that
\begin{equation}\label{5.3}
 \E e^{it{\tilde {\mathbb T}}}
\sim
\E e^{it({\tilde {\mathbb T}}+{\tilde \Lambda}_1)}
\qquad
{\text{and}}
\qquad
\E e^{it({\tilde {\mathbb T}}+{\tilde \Lambda}_1)}
\sim
 \E e^{it {\mathbb T}}.
\end{equation}
The second relation follows from the moment bounds of
Lemma \ref{LA1.2} via Taylor expansion. We have
$$
 \E e^{it{\mathbb T}}
=
\E e^{it({\tilde {\mathbb T}}+{\tilde \Lambda}_1)}
+R,
\qquad
|R|\le |t|\E|{\tilde \Lambda}_2|,
$$
By Lyapunov's inequality,
$$
\E|{\tilde \Lambda}_2|
\le
(\E\Lambda_2^2)^{1/2}
+
(\E\Lambda_3^2)^{1/2}
+
(\E \Lambda_5^2)^{1/2}.
$$
Invoking the moment bounds of Lemma \ref{LA1.2} we obtain
$|t|\E|{\tilde \Lambda}_2|\prec \R$, thus, proving the second part
of (\ref{5.3}).

In order to prove the first part we combine Taylor's expansion with bounds
for characteristic functions.
Expanding the exponent we obtain
$$
\E e^{it({\tilde {\mathbb T}}+{\tilde \Lambda}_1)}
=
\E e^{it{\tilde {\mathbb T}}}
+
it\E e^{it{\tilde {\mathbb T}}}{\tilde \Lambda}_1
+
R,
\qquad
|R|\le t^2\E|{\tilde \Lambda}_1|^2.
$$
Invoking the identities
\begin{equation}\label{5.4}
\E\Lambda_1^2=\binom{m}{2}\frac{\gamma_2}{N^3},
\qquad
\E \Lambda_4^2=m \binom{N-m}{2}\frac{\zeta_2}{N^5}
\end{equation}
we obtain, for $\gamma_2<c_*$ and $\zeta_2<c_*$, see (\ref{1.3}),
 and $m\le N^{1/12}$, that
 $R\prec \R$.
We complete the proof of (\ref{5.3}) by showing that
\begin{equation}\label{5.5}
t\E e^{it{\tilde {\mathbb T}}}{\tilde \Lambda}_1\prec \R.
\end{equation}
Let us prove (\ref{5.5}).
Split ${\mathbb W}={\mathbb W}_1+{\mathbb W}_2+{\mathbb W}_3+R_{W}$, where
$$
{\mathbb W}_k=\sum_{A\subset \Omega',\,|A|=k}T_A,
\qquad
R_W=\sum_{A\subset \Omega', \, |A|\ge 4}T_A.
$$
Here $\Omega'=\{m+1,\dots, N\}$.
Denote ${\mathbb R}={\mathbb U}_2^*+{\mathbb W}_3+R_W$ and
${\mathbb U}_1=\sum_{j=1}^Ng_1(X_j)$. We have
 $
{\tilde {\mathbb T}}={\mathbb U}_1+{\mathbb W}_2+{\mathbb R}$.
Expanding the exponent in powers of $it{\mathbb R}$ we obtain
\begin{equation}\label{5.6}
 t\E e^{it{\tilde {\mathbb T}}}{\tilde \Lambda}_1
=
 t\E e^{it({\mathbb U}_1+{\mathbb W}_2)}{\tilde \Lambda}_1+t^2R,
\end{equation}
where
\begin{align*}
&
|R|
\le
\E|{\tilde \Lambda}_1{\mathbb R}|
\le
(r_1+r_2)(r_3+r_4+r_5),
\\
&
r_1^2=\E \Lambda_1^2,
\quad
r_2^2=\E \Lambda_4^2,
\quad
r_3^2=\E ({\mathbb U}_2^*)^2,
\quad
r_4^2=\E R_W^2,
\quad
r_5^2=\E {\mathbb W}_3^2.
\end{align*}
In the last step we applied the Cauchy-Schwartz inequality.
Combining (\ref{5.4}) with the
identities
$$
\E ({\mathbb U}_2^*)^2
=
\frac{m(N-m)}{N^3}\gamma_2,
\qquad
\E {\mathbb W}_3^2=\frac{\binom{N-m}{3}}{N^5}\zeta_2
$$
and invoking the simple bound
$$
\E R_W^2\le \frac{\Delta_4^2}{N^3}\le \frac{D_*}{N^{2+2\nu_1}},
$$
we obtain $t^2(r_1+r_2)(r_3+r_4+r_5)\prec \R$.
Therefore, (\ref{5.6}) implies
$$
t\E e^{it{\tilde {\mathbb T}}}{\tilde \Lambda}_1
\sim
 t\E e^{it({\mathbb U}_1+{\mathbb W}_2)}{\tilde \Lambda}_1.
$$
Let us show that $t\E e^{it({\mathbb U}_1+{\mathbb W}_2)}{\tilde \Lambda}_1\sim 0$.
Expanding the exponent in powers of $it{\mathbb W}_2$ we get
\begin{align*}
&
 t\E e^{it({\mathbb U}_1+{\mathbb W}_2)}{\tilde \Lambda}_1
= f_1(t)+f_2(t)+f_3(t)+f_4(t),
\\
&
f_1(t)= t\E e^{it{\mathbb U}_1}{\tilde \Lambda}_1,
\qquad
\qquad
\quad
f_2(t)= it^2\E e^{it{\mathbb U}_1}\Lambda_1{\mathbb W}_2,
\\
&
f_3(t)= t^2\E e^{it{\mathbb U}_1} \Lambda_4{\mathbb W}_2\theta_1,
\qquad
\
f_4(t)= t^3\E e^{it{\mathbb U}_1}\Lambda_1{\mathbb W}_2^2\theta_2/2,
\end{align*}
where $\theta_1, \theta_2$ are
 functions of ${\mathbb W}_2$ satisfying $|\theta_i|\le 1$.

Let us show that $f_i\prec \R$, for $i=1,2,3,4$.
 Split the set
$\Omega_m=\{1,\dots, m\}$ in three (non-intersecting) parts
$A_1\cup A_2\cup A_3=\Omega_m$ of
(almost) equal size $|A_i|\approx m/3$. The set of pairs
$\bigl\{\{i,j\}\subset \Omega_m\bigr\}$ splits into six (non-intersecting)
parts $B_{kr}$, $1\le k\le r\le 3$ (the pair $\{i,j\}$ belongs to $B_{kr}$ if
$i\in A_k$ and $j\in A_r$).
Write
\begin{align*}
&
\Lambda_1=\sum_{1\le k\le r\le 3}\Lambda_1(k,r),
\qquad
\Lambda_1(k,r)=\sum_{\{i,j\}\in B_{kr}} g_2(X_k,X_l),
\\
&
\Lambda_4=\sum_{1\le k\le 3}\Lambda_4(k),
\qquad
\Lambda_4(k)=\sum_{i\in A_k}\sum_{m+1\le j<l\le N}g_3(X_i,X_j,X_l).
\end{align*}

Let us prove $f_4\prec \R$. We shall show that
\begin{equation}\label{5.7}
t^3\E e^{it{\mathbb U}_1}\Lambda_1(k,r){\mathbb W}^2_2\theta_2\prec \R.
\end{equation}
 Given a pair $(k,r)$ denote
$A_i=\Omega_m\setminus (A_k\cup A_r)$ and write $k_i=|A_i|$. Note
that $k_i\approx m/3$. We shall assume that $k_i\ge m/4$. Since the random
variable
${\mathbb U}_1(A_i):=\sum_{j\in A_i}g_1(X_j)$ and the random variables
 $\Lambda_1(k,r)$, ${\mathbb W}_2$ are independent,
we have
$$
\E e^{it{\mathbb U}_1}\Lambda_1(k,r){\mathbb W}^2_2\theta_2
\le
\E e^{it{\mathbb U}_1(A_i)}
\
 \E \Lambda_1(k,r){\mathbb W}^2_2\theta_2.
$$
Therefore,
\begin{equation}\label{5.8}
|\E e^{it{\mathbb U}_1}\Lambda_1(k,r){\mathbb W}^2_2\theta_2|
\le
|\E e^{it{\mathbb U}_1(A_i)}|
\
 \E| \Lambda_1(k,r){\mathbb W}^2_2|.
\end{equation}
The first factor on the right
 is bounded from above by $\exp\{-mt^2/16N\}$, for
$k_i\ge m/4$, see (\ref{5.39}) below.
 The second factor  is bounded from above by $r$,
where
$$
r^2=\E \Lambda_1^2(k,r)\E {\mathbb W}_2^4\le c_*m^2N^{-5}.
$$
Here we combined the Cauchy-Schwartz inequality and  the bounds
$$
 \E \Lambda_1^2(k,r)\le c_*m^2N^{-3},
\qquad
\E {\mathbb W}_2^4\le c_*N^{-2}.
$$
Finally, (\ref{5.7}) follows from (\ref{5.8})
$$
\bigl| t^3\E e^{it{\mathbb U}_1}\Lambda_1(k,r){\mathbb W}_2^2\theta_2 \bigr|
\le
c_*|t|^3e^{-mt^2/16N}mN^{-5/2}\prec \R.
$$

The proof of $f_3\prec \R$ is almost the same as that of  $f_4\prec \R$.

Let us prove $f_2\prec \R$.
Split the set $\Omega'=\{m+1,\dots, N\}$ into three (non-intersecting) parts
$B_1\cup B_2\cup B_3=\Omega'$ of (almost) equal sizes $|B_i|\approx (N-m)/3$.
Split the set of pairs $\bigl\{ \{i,j\}:\, m+1\le i<j\le N\bigr\}$ into
(non-intersecting) groups $D(k,r)$, for $1\le k\le r\le 3$. The pair
$\{i,j\}\in D(k,r)$ if $i\in B_k$ and $j\in B_r$.
Write
\begin{align*}
&
{\mathbb W}_2=\sum_{1\le k\le r\le 3}{\mathbb W}_2(k,r),
\qquad
{\mathbb W}_2(k,r)=\sum_{\{i,j\}\in D(k,r)}g_2(X_i,X_j).
\\
&
\Lambda_4=\sum_{1\le  k\le r\le 3}\Lambda_4(k,r),
\qquad
\Lambda_4(k,r)=\sum_{1\le s\le m}\sum_{\{i,j\}\in D(k,r)}g_3(X_s,X_i,X_j),
\end{align*}
In order to prove $f_2\prec \R$ we shall show that
\begin{equation}\label{5.9}
t^2\E e^{it{\mathbb U}_1}\Lambda_1{\mathbb W}_2(k,r)\prec \R.
\end{equation}
Write $B_i=\Omega'\setminus(B_k\cup  B_r)$ and denote $m_i=|B_i|$.
We shall assume that $m_i\ge N/4$. Since the random variable
${\mathbb U}_1(B_i)=\sum_{j\in B_i}g_1(X_j)$ and the random variables $\Lambda_1$
and ${\mathbb W}_2(k,r)$ are independent, we have, cf. (\ref{5.8}),
\begin{equation}\label{5.10}
|\E e^{it{\mathbb U}_1}\Lambda_1{\mathbb W}_2(k,r)|
\le
|\E e^{it{\mathbb U}_1(B_i)}|
\
 \E|\Lambda_1{\mathbb W}_2(k,r)|.
\end{equation}
The first factor in the right is the product $|\alpha^{m_i}(t)|\le
e^{-m_it^2/4N}$, see the argument used in the proof of (\ref{5.7}) above.
 The second factor  is bounded from above by $r$,
where
$$
r^2=\E \Lambda_1^2\E {\mathbb W}_2^2(k,r)\le c_*m^2N^{-4}.
$$
Finally, we obtain, using the inequality $m_i\ge N/4$,
$$
|\E e^{it{\mathbb U}_1}|
\
\E|\Lambda_1{\mathbb W}_2(k,r)|
\le c_*\frac{m}{N^2}\exp\{-t^2\frac{m_i}{4N}\}
\le c_*\frac{m}{N^2}\exp\{-\frac{t^2}{16}\}.
$$
This in combination with (\ref{5.10}) shows (\ref{5.9}). We obtain $f_2\prec \R$.

Let us prove $f_1\prec \R$. We shall show that $f^*\prec \R$ and $f^{\star}\prec \R$, where
$$
f^{\star}=t\E e^{it{\mathbb U}_1}\Lambda_1
\qquad
{\text{and}}
\qquad
f^*=t\E e^{it{\mathbb U}_1}\Lambda_4
$$
satisfy $f^*+f^{\star}=f_1$.

Let us show $f^{\star}\prec \R$.
Denote ${\mathbb U}_1^{\star}=\sum_{j=m+1}^Ng_1(X_j)$. We obtain, by the
independence of ${\mathbb U}_1^{\star}$ and $\Lambda_1$ that
$$
|\E e^{it{\mathbb U}_1}\Lambda_1|
\le
|\E e^{it{\mathbb U}_1^{\star}}|
\
\E|\Lambda_1|.
$$
Invoking, for $N-m>N/2$, the bound
$
|\E e^{it{\mathbb U}_1^{\star}}|
\le e^{-t^2/8}
$, see (\ref{5.39}) below,
and the bound $\E|\Lambda_1|\le (\E \Lambda_1^2)^{1/2}\le c_*mN^{-3/2}$
we obtain
$$
|f^{\star}(t)|\le c_*|t|e^{-t^2/8}N^{-3/2}\prec \R.
$$

Let us prove $f^*\prec \R$. We shall show that, for $1\le k\le r\le 3$,
\begin{equation}\label{5.11}
t\E e^{it{\mathbb U}_1}\Lambda_4(k,r)\prec \R.
\end{equation}
Proceeding  as in the proof of (\ref{5.9}) we obtain the chain of inequalities
\begin{equation}\label{5.12}
|\E e^{it{\mathbb U}_1}\Lambda_4(k,r)|
\le
e^{-t^2/16}\E |\Lambda_4(k,r)|
\le
c_*
e^{-t^2/16}m^{1/2}N^{-3/2}.
\end{equation}
In the last step we applied Cauchy-Schwartz and the simple
bound $\E \Lambda_4^2(k,r)\le c_*mN^{-3}$.
Clearly, (\ref{5.12}) implies (\ref{5.11}).

\bigskip
{\bf 5.2.} Here we prove the second relation of (\ref{5.2}). Firstly, we
 shall show that
\begin{equation}\label{5.13}
\E e^{it{\mathbb T}}\sim \E \exp\{it({\mathbb U}_1+{\mathbb U}_2+{\mathbb U}_3)\},
\end{equation}
\begin{equation}\label{5.14}
\E \exp\{it({\mathbb U}_1+{\mathbb U}_2+{\mathbb U}_3)\}
\sim
\E \exp\{it({\mathbb U}_1+{\mathbb U}_2)\}+\binom{N}{3} e^{-t^2/2}(it)^4w,
\end{equation}
where $w=\E g_3(X_1,X_2,X_3)g_1(X_1)g_1(X_2)g_1(X_3)$.

Let $m(t)$ be an integer valued function such that
\begin{equation}\label{5.15}
m(t)\approx C_1Nt^{-2}\ln (t^2+1),
\qquad
C_1\le |t|\le t_1,
\end{equation}
and put $m(t)\equiv 10$, for $|t|\le C_1$. Here $C_1$ denotes a large
absolute constant (one can take, e.g.,  $C_1=200$).
 Assume, in addition, that the numbers $m=m(t)$ are even.

\bigskip
{\bf 5.2.1.} Here we prove (5B.1). Given $m$ write
$$
{\mathbb T}={\mathbb U}_1+{\mathbb U}_2+{\mathbb U}_3+{\mathbb H},
$$
where
$$
{\mathbb H}={\mathbb H}_1+{\mathbb H}_2,
\qquad
{\mathbb H}_1=\sum_{|A|\ge 4,\, A\cap \Omega_m=\emptyset}T_A,
\qquad
{\mathbb H}_2=\sum_{|A|\ge 4,\, A\cap \Omega_m\not=\emptyset}T_A.
$$
In order to show (\ref{5.13}) we expand the exponent in powers of $it{\mathbb H}$ and
$it{\mathbb U}_3$,
$$
\E \exp\{it{\mathbb T}\}
=
\E \exp\{it({\mathbb U}_1+{\mathbb U}_2+{\mathbb U}_3)\}
+
\E \exp\{it({\mathbb U}_1+{\mathbb U}_2)\}it{\mathbb H}
+
R,
$$
where
$
|R|\le t^2(\E {\mathbb H}^2+\E |{\mathbb U}_3{\mathbb H}|)
$. Invoking the bounds, see (\ref{A1.1}), (\ref{A1.2}), (\ref{1.3}), (\ref{1.4}),
\begin{equation}\label{5.16}
\E {\mathbb H}^2\le N^{-3}\Delta_4^2\le c_*N^{-2-2\nu_1},
\qquad
\E {\mathbb U}_3^2\le N^{-2}\zeta_2\le c_*N^{-2}
\end{equation}
we obtain, by Cauchy-Schwartz,
$|R|\le c_*t^2 N^{-2-\nu_1}\prec \R$.
We complete the proof of (\ref{5.13}) by showing that
\begin{equation}\label{5.17}
\E \exp\{it({\mathbb U}_1+{\mathbb U}_2)\}it{\mathbb H}\prec \R.
\end{equation}
Before proving (\ref{5.17}) we collect some auxiliary inequalities.
For $m=2k$ write
\begin{equation}\label{5.18}
\Omega_m=A_1\cup A_2,
{\text{ where}}
\qquad
A_1=\{1,\dots, k\},
\qquad
A_2=\{k+1,\dots, 2k\}.
\end{equation}
Furthermore, split the sum
\begin{eqnarray}\label{5.19}
{\mathbb U}_2
&
=
&{\mathbb Z}_1+{\mathbb Z}_2+{\mathbb Z}_3+{\mathbb Z}_4,
\\
\nonumber
{\mathbb Z}_1
&
=
&
\sum_{1\le i<j\le m}g_2(X_i,X_j),
\qquad
{\mathbb Z}_2
=
\sum_{i\in A_1}\sum_{m<j\le N}g_2(X_i,X_j),
\\
\nonumber
{\mathbb Z}_3
&
=
&
\sum_{i\in A_2}\sum_{m<j\le N}g_2(X_i,X_j),
\qquad
{\mathbb Z}_4=\sum_{m<i<j\le N}g_2(X_i,X_j).
\end{eqnarray}
In what follows we shall use the simple bounds, see (\ref{1.3}),
\begin{eqnarray}\label{5.20}
&&
\E {\mathbb Z}_1^2\le \frac{m^2}{N^3}\gamma_2\le c_*\frac{m^2}{N^3},
\qquad
\E {\mathbb Z}_4^2\le \frac{\gamma_2}{N}\le \frac{c_*}{N},
\\
\nonumber
&&
\E {\mathbb Z}_i^2\le \frac{m}{N^2}\gamma_2\le c_*\frac{m}{N^2},
\quad
\E {\mathbb Z}_i^4\le c\frac{m^2}{N^{4}}\gamma_4\le c_*\frac{m^2}{N^4},
\quad
i=2,3.
\end{eqnarray}

Let us prove (\ref{5.17}).
Expand the exponent $\exp\{it({\mathbb U}_1+{\mathbb Z}_1+\dots+{\mathbb Z}_4)\}$
in powers of
$it{\mathbb Z}_1$
to get
$$
\E \exp\{it({\mathbb U}_1+{\mathbb U}_2)\}it{\mathbb H}
=
h_1(t)+R
$$
where
$h_1(t)=\E \exp\{it({\mathbb U}_1+{\mathbb Z}_2+\dots+{\mathbb Z}_4)\}it{\mathbb H}$
and where
$$
|R|
\le
t^2\E|{\mathbb H}{\mathbb Z}_1|
\le
t^2(\E{\mathbb H}^2)^{1/2}(\E {\mathbb Z}_1^2)^{1/2}
\le
c_*t^2mN^{-(5+2\nu_1)/2}.
$$
For $m=m(t)$ satisfying (\ref{5.15}) we have $R\prec\R$. Therefore, we obtain
$$
\E \exp\{it({\mathbb U}_1+{\mathbb U}_2)\}it{\mathbb H}\sim
h_1.
$$
In order to prove $h_1\prec \R$ we write $h_1=h_2+h_3$ and show that
$h_2, h_3\prec\R$ , where
$$
h_2=\E \exp\{it({\mathbb U}_1+{\mathbb Z}_2+\dots+{\mathbb Z}_4)\}it{\mathbb H}_1,
\qquad
h_3=\E \exp\{it({\mathbb U}_1+{\mathbb Z}_2+\dots+{\mathbb Z}_4)\}it{\mathbb H}_2.
$$

Let us show that $h_2\prec\R$. Firstly,
 we prove that
\begin{equation}\label{5.21}
h_2\sim h_{2.1}+h_{2.2}+h_{2.3},
\end{equation}
where
$h_{2.1}(t)=\E \exp\{it({\mathbb U}_1+{\mathbb Z}_4)\}it{\mathbb H}_1$
and, for $j=2,3$,
$$
h_{2.j}(t)=\E\exp\{it({\mathbb U}_1+{\mathbb Z}_4)\}(it)^2{\mathbb H}_1{\mathbb Z}_j.
$$
Expanding the exponent in powers of
$it({\mathbb Z}_2+{\mathbb Z}_3)$ we obtain
$$
h_2=h_{2.1}+h_{2.2}+h_{2.3}+R,$$
where $|R|\le |t|^3\E |{\mathbb H}_1|({\mathbb Z}_2+{\mathbb Z}_3)^2$ is bounded
from above by
$$
|t|^3 (\E {\mathbb H}_1^2)^{1/2}(\E({\mathbb Z}_2+{\mathbb Z}_3)^4)^{1/2}
\le
c_*|t|^3mN^{-3-\nu_1}
\prec\R.
$$
In the last step we used $\E {\mathbb H}_1^2\le \E {\mathbb H}^2$ and applied
(\ref{5.16}) and (\ref{5.20}).
Therefore, (\ref{5.21}) follows.

Let us show  $h_{2.i}\prec \R$, for $i=1,\, 2,\,3$.
The random variable ${\mathbb U}_1(A_1)$
does not depend on the observations $X_j$, $j\in \Omega\setminus A_1$.
Therefore, we can write
$$
h_{2.3}
=
\E \exp\{it{\mathbb U}_1(A_1)\}
\E\exp\{it({\mathbb U}_1(\Omega\setminus A_1)+{\mathbb Z}_4)\}(it)^2
{\mathbb H}_1{\mathbb Z}_3.
$$
Furthermore, using (\ref{5.39}) we obtain, for $|A_1|=m/2$,
\begin{equation}\label{5.22}
|h_{2.3}|\le t^2|\alpha^{m/2}(t)|\E|{\mathbb H}_1{\mathbb Z}_3|
\le
c_*t^2\exp\{-t^2\frac{m}{8N}\}\frac{m^{1/2}}{N^{2+\nu_1}}.
\end{equation}
In the last step we combined the bound $\E {\mathbb H}_1^2\le c_*N^{-2-2\nu_1}$
and (\ref{5.20}) to get
$$
\E|{\mathbb H}_1{\mathbb Z}_3|
\le
(\E {\mathbb H}_1^2)^{1/2}(\E {\mathbb Z}_3^2)^{1/2}
\le
c_*m^{1/2}N^{-2-\nu_1}.
$$
Note that choosing of $C_1$ in (\ref{5.15}) sufficiently large implies,
for $|t|\ge C_1$,
$$
t^2m/12N\approx (C_1/12)\ln (t^2+1)\ge 10 \ln (t^2+1).
$$
An application of this bound to the argument of the exponent in (\ref{5.22}) shows
 $h_{2.3}\prec \R$.
The proof of $h_{2.i}\prec\R$, for $i=1,2$, is almost the same. Therefore,
we obtain $h_2\prec\R$.

Let us prove $h_3\prec \R$. Firstly we collect some auxiliary inequalities.
Write $m=2k$ (recall that
the number $m$ is even) and split $\Omega_m=B \cup D$, where $B$
denotes the set of odd numbers and $D$ denotes the set of even numbers.
Split ${\mathbb H}_2={\mathbb H}_B+{\mathbb H}_D+{\mathbb H}_C$.
Here, for $A\subset \Omega_N$ and $|A|\ge 4$, we denote by ${\mathbb H}_B$
the sum of $T_A$ such that $A\cap B=\emptyset$ and
$A\cap D\not=\emptyset$; ${\mathbb H}_D$ denotes the sum of $T_A$ such that
$A\cap B\not=\emptyset$ and
$A\cap D=\emptyset$;
${\mathbb H}_C$ denotes the sum of $T_A$ such that
$A\cap B\not=\emptyset$ and
$A\cap D\not=\emptyset$.
It follows from the inequalities (\ref{A1.12}) and (\ref{1.4}) that
\begin{equation}\label{5.23}
\E \,{\mathbb H}_C^2\le c_*m^2N^{-4-2\nu_1},
\qquad
\E \,{\mathbb H}_B^2=\E \,{\mathbb H}_D^2\le c_*mN^{-3-2\nu_1}.
\end{equation}

Using the notation
$z=it\exp\{it({\mathbb U}_1+{\mathbb Z}_2+{\mathbb Z}_3+{\mathbb Z}_4)\}$ write
\begin{align*}
&
h_3=\E z{\mathbb H}_2=h_{3.1}+h_{3.2}+h_{3.3},
\\
&
h_{3.1}=\E z{\mathbb H}_B,
\quad
h_{3.2}=\E z{\mathbb H}_D,
\quad
h_{3.3}=\E z{\mathbb H}_C.
\end{align*}
We shall show that $h_{3.i}\prec\R$, for $i=1,2,3$.
The relation $h_{3.3}\prec\R$ follows from (\ref{5.23}) and (\ref{5.20}), and
by Cauchy-Schwartz,
$
|h_{3.3}|\le c_*|t|\, mN^{-2-\nu_1}\prec\R$.

Let us show that $h_{3.2}\prec\R$. Expanding the exponent in powers of
$it({\mathbb Z}_2+{\mathbb Z}_3)$ we obtain
$$
h_{3.2}=h_{3.2}^*+R,
\qquad
h_{3.2}^*:=\E \exp\{it({\mathbb U}_1+{\mathbb Z}_4)\}it{\mathbb H}_D,
$$
where $|R|\le t^2\E |{\mathbb H}_D  ({\mathbb Z}_2+{\mathbb Z}_3)|$.
Combining the bounds (\ref{5.20}) and (\ref{5.23}) we obtain, by Cauchy-Schwartz,
$|R|\le c_*t^2mN^{-(5+2\nu_1)/2}\prec\R$.
Next we show that $h_{3.2}^*\prec\R$. The random variable
${\mathbb U}_1(D)=\sum_{j\in D}g_1(X_j)$ and the random variable ${\mathbb H}_D$
are independent. Therefore, we can write
$$
|h_{3.2}^*|\le |t|\,|\E \exp\{it{\mathbb U}_1(D)\}|\E |{\mathbb H}_D|.
$$
Combining (\ref{5.39}) and  (\ref{5.23}) we obtain using Cauchy-Schwartz,
$$
|h_{3.2}^*|
\le
c_*|t|\,e^{-mt^2/8N}m^{1/2}N^{-(3+2\nu_1)/2}\prec\R.
$$
The proof of $h_{3.1}\prec\R$ is similar. Therefore, we obtain $h_3\prec\R$.
This together with the relation $h_2\prec\R$, proved above, implies
$h_1\prec\R$. Thus we arrive at (\ref{5.17}) completing the proof of (\ref{5.13}).

\bigskip
{\bf 5.2.2.} Here we prove (\ref{5.14}).
We start with some auxiliary moment inequalities. Split
$$
{\mathbb U}_3=W+Z,
\qquad
W=\sum_{|A|=3,\, A\cap \Omega_m\not=\emptyset}T_A,
\qquad
Z=\sum_{|A|=3,\, A\cap \Omega_m=\emptyset}T_A.
$$
Using the orthogonality and moment bounds for $U$-statistics, see, e.g.,
Dharmadhikari et al (1968), one can show that
$$
\E W^2\le mN^2\E g_3^2(X_1,X_2,X_3),
\qquad
\E Z^2\le N^3\E g_3^2(X_1,X_2,X_3),
$$
and $\E |Z|^s\le cN^{3s/2}\E|g_3(X_1,X_2,X_3)|^s$.
Invoking (\ref{1.3}) we obtain
\begin{equation}\label{5.24}
\E W^2\le c_*mN^{-3},
\qquad
\E Z^2\le c_*N^{-2},
\qquad
\E|Z|^s\le c_*N^{-s}.
\end{equation}
For the sets $A_1,A_2\subset \Omega_m$ defined in (\ref{5.18}) write
\begin{align*}
&
{\D}=\{A\subset {\Omega}_N:\, |A|=3,\, A\cap {\Omega}_m\not=\emptyset\},
\\
&
{\D}_1=\{A\in {\D}:\, A\cap A_1=\emptyset\},
\\
&
{\D}_2=\{A\in {\D}:\, A\cap A_2=\emptyset\},
\\
&
{\D}_3=\{A\in {\D}:\, A\cap A_1\not=\emptyset,
\
A\cap A_2\not=\emptyset\}.
\end{align*}
We have
${\D}={\D}_1\cup{\D}_2\cup{\D}_3$ and $W=\sum_{A\in {\D}}T_A$. Therefore, we can write
$W=W_1+W_2+W_3$, where
$W_j=\sum_{A\in {\D}_j}T_A$.

A calculation shows that
$$
\E W_1^2=\E W_2^2\le kN^{2}\E g_3^2(X_1,X_2,X_3),
\qquad
\E W_3^2\le k^2N\E g_3^2(X_1,X_2,X_3).
$$
Therefore, we obtain form (\ref{1.3}) that
\begin{equation}\label{5.25}
\E W_1^2=\E W_2^2\le c_*mN^{-3},
\qquad
\E W_3^2\le c_*m^2N^{-4}.
\end{equation}

Let us prove (\ref{5.14}). Write ${\mathbb U}_3=W+Z$.
Expanding the exponent in powers of $itW$ we obtain
\begin{align*}
&
\E \exp\{it({\mathbb U}_1+{\mathbb U}_2+{\mathbb U}_3)\}=h_4+h_5+R,
\\
&
h_4
=\E \exp\{it({\mathbb U}_1+{\mathbb U}_2+Z)\},
\\
&
h_5=\E \exp\{it({\mathbb U}_1+{\mathbb U}_2+Z)\}itW,
\end{align*}
where, by (\ref{5.24}),  $|R|\le t^2\E W^2\le c_*t^2mN^{-3}\prec\R$.
This implies
$$
\E \exp\{it({\mathbb U}_1+{\mathbb U}_2+{\mathbb U}_3)\}\sim h_4+h_5.
$$
In order to prove (\ref{5.14}) we shall show that
\begin{eqnarray}\label{5.26}
&&
h_5\sim \E \exp\{it{\mathbb U}_1\}itW,
\\
\label{5.27}
&&
h_4
\sim
\E\exp\{it({\mathbb U}_1+{\mathbb U}_2)\}
+
\E\exp\{it{\mathbb U}_1\}itZ,
\\
\label{5.28}
&&
\E \exp\{it{\mathbb U}_1\}it{\mathbb U}_3\sim \binom{N}{3} e^{-t^2/2}(it)^4w.
\end{eqnarray}

Let us prove (\ref{5.26}).
Expanding the exponent (in $h_5$) in powers of $itZ$ we obtain
$$
h_5
=
h_6+R,
\qquad
h_6
=
\E\exp\{it({\mathbb U}_1+{\mathbb U}_2)\}itW,
$$
where, by (\ref{5.24}) and Cauchy-Schwartz,
$$
|R|
\le
t^2\E|WZ|\le c_*t^2m^{1/2}N^{-5/2}\prec\R.
$$
We have, $h_5\sim h_6$.

It remains to show that
$h_6\sim  \E \exp\{it{\mathbb U}_1\}itW$.
Split
\begin{equation}\label{5.29}
{\mathbb U}_2={\mathbb U}_2^*+{\mathbb U}_2^{\star},
\qquad
{\mathbb U}_2^*=\sum_{|A|=2,\, A\cap \Omega_m\not=\emptyset}T_A,
\qquad
{\mathbb U}_2^{\star}=\sum_{|A|=2,\, A\cap \Omega_m=\emptyset}T_A.
\end{equation}
We have, see (\ref{5.20}),
\begin{equation}\label{5.30}
\E ({\mathbb U}_2^*)^2\le c_*mN^{-2},
\qquad
\E ({\mathbb U}_2^{\star})^2\le c_*N^{-1}.
\end{equation}
Expanding
the exponent (in $h_6$) in powers of $it{\mathbb U}_2^*$ we obtain
$$
h_6=h_7+R,
\qquad
{\text{where}}
\qquad
h_7=\E\exp\{it({\mathbb U}_1+{\mathbb U}_2^{\star})\}itW,
$$
and where, by (\ref{5.24}), (\ref{5.30}) and Cauchy-Schwartz,
$$
|R|
\le
t^2\E|W{\mathbb U}_2^*|
\le
c_*t^2mN^{-5/2}
\prec
\R.
$$
Therefore, we obtain $h_6\sim h_7$.

We complete the proof of (\ref{5.26}) by showing
that $h_7\sim\E\exp\{it{\mathbb U}_1\}itW$.
Use the decomposition $W=W_1+W_2+W_3$ and write
$$
h_7=h_{7.1}+h_{7.2}+h_{7.3},
\qquad
h_{7.j}=\E\exp\{it({\mathbb U}_1+{\mathbb U}_2^{\star})\}itW_j.
$$
We shall show that
\begin{equation}\label{5.31}
h_{7.j}\sim\E\exp\{it{\mathbb U}_1\}itW_j,
\qquad
j=1,2,3.
\end{equation}
Expanding in powers of $it{\mathbb U}_2^{\star}$ we obtain
$$
h_{7.j}= \E\exp\{it{\mathbb U}_1\}itW_j+R_j,
$$
where $R_j=(it)^2\E\exp\{it{\mathbb U}_1\} W_j{\mathbb U}_2^{\star}\theta$
and where $\theta$ is a function of ${\mathbb U}_2^{\star}$ satisfying
$|\theta|\le 1$.
In order to prove (\ref{5.31}) we  show that
$R_j\prec \R$, for $j=1,2,3$.

Combining (\ref{5.25}) and (\ref{5.30}) we obtain via
Cauchy-Schwartz
$$
|R_3|\le c_*t^2mN^{-5/2}\prec\R.
$$

Furthermore, using the fact that the random variable
${\mathbb U}_1(A_2)$ and the random variables ${\mathbb U}_2^{\star}$ and $W_2$
are independent, we can write
$$
|R_2|
\le
t^2|\E \exp\{it{\mathbb U}_1(A_2)\}|\E|W_2{\mathbb U}_2^{\star}|
\le
c_*t^2e^{-mt^2/8N}m^{1/2}N^{-2}
\prec
\R.
$$
Here we used  (\ref{5.39}) and the moment inequalities (\ref{5.25}) and (\ref{5.30}).
The proof of $R_1\prec\R$ is similar. We arrive at (\ref{5.31}) and, thus, complete
 the proof of (\ref{5.26}).

Let us prove (\ref{5.27}). We proceed in two steps. Firstly we show
\begin{eqnarray}\label{5.32}
&&
h_4\sim h_8+h_9,
\\
\nonumber
&&
h_8=\E\exp\{it({\mathbb U}_1+{\mathbb U}_2)\},
\qquad
h_9=\E\exp\{it({\mathbb U}_1+{\mathbb U}_2)\}itZ.
\end{eqnarray}
Secondly, we show
\begin{equation}\label{5.33}
h_9\sim \E\exp\{it{\mathbb U}_1\}itZ.
\end{equation}

In order to prove (\ref{5.32}) we  write
$$
h_4=h_8+h_9+R,
\qquad
R=\E\exp\{it({\mathbb U}_1+{\mathbb U}_2)\}r,
\qquad
r=\exp\{itZ\}-1-itZ,
$$
and show that $R\prec \R$.
In order to bound the remainder $R$ we write
${\mathbb U}_2={\mathbb U}_2^*+{\mathbb U}_2^{\star}$, see (\ref{5.29}), and expand
the exponent in powers of $it{\mathbb U}_2^*$. We obtain
$R=R_1+R_2$, where
$$
R_1=\E \exp\{it({\mathbb U}_1+{\mathbb U}_2^{\star})\}r
\qquad
{\text{and}}
\qquad
|R_2|\le \E|it{\mathbb U}^*_2r|.
$$
Note that, for $2<s\le 3$, we have  $|r|\le c |tZ|^{s/2}$.
Combining (\ref{5.24}) and (\ref{5.30}) we obtain via Cauchy-Schwartz,
$$
|R_2|
\le
|t|^{1+s/2}\E|Z|^{s/2}|{\mathbb U}_2^*|
\le
c_*|t|^{1+s/2}m^{1/2}N^{-1-s/2}
\prec\R.
$$
In order to prove $R_1\prec\R$ we use the fact that the random variable
${\mathbb U}_1(\Omega_m)$ and the random variables ${\mathbb U}_2^{\star}$ and
 $r$ are
independent. Invoking the inequality $|r|\le t^2Z^2$  we obtain from (\ref{5.39})
and (\ref{5.24})
$$
|R_1|\le t^2|\alpha^m(t)|\E Z^2\le c_* t^2e^{-mt^2/4N}N^{-2}\prec\R.
$$
We thus arrive at (\ref{5.32}).

Let us prove (\ref{5.33}). Use the decomposition (\ref{5.19})  and expand the exponent
 (in $h_9$) in powers of $it{\mathbb Z}_1$ to get
$h_9=h_{10}+R$, where
$$
h_{10}=\E \exp\{it({\mathbb U}_1+{\mathbb Z}_2+{\mathbb Z}_3+{\mathbb Z}_4)\}itZ,
\quad
|R|\le t^2\E|Z{\mathbb Z}_1|.
$$
Combining (\ref{5.20}) and (\ref{5.24}) we obtain via Cauchy-Schwartz
$$
|R|\le c_* t^2 mN^{-5/2}\prec\R.
$$
Therefore, we have
$$
h_9\sim h_{10}.
$$

Now we expand the exponent in
$h_{10}$ in powers of $it({\mathbb Z}_2+{\mathbb Z}_3)$ and obtain
$
h_{10}=h_{11}+h_{12}+R$,
where
$$
h_{11}=\E\exp\{it({\mathbb U}_1+{\mathbb Z}_4)\}itZ,
\qquad
h_{12}=\E\exp\{it({\mathbb U}_1+{\mathbb Z}_4)\}(it)^2Z({\mathbb Z}_2+{\mathbb Z}_3),
$$
and where $|R|\le |t|^3\E |Z|\,|{\mathbb Z}_2+{\mathbb Z}_3|^2$.
Combining (\ref{5.20}) and (\ref{5.24}) we obtain via Cauchy-Schwartz
$|R|\le |t|^3mN^{-3}\prec\R$. Therefore, we have
$$
h_{10}\sim h_{11}+h_{12}.
$$
We complete the proof of (\ref{5.33}), by showing that
\begin{equation}\label{5.34}
h_{11}\sim\E\exp\{it{\mathbb U}_1\}itZ
\qquad
{\text{and}}
\qquad
h_{12}\prec\R.
\end{equation}
In order to prove the second bound  write
$$
h_{12}=R_2+R_3,
\qquad
{\text{where}}
\qquad
R_j=\E\exp\{it({\mathbb U}_1+{\mathbb Z}_4)\}(it)^2Z{\mathbb Z}_j.
$$
We shall show that $R_3\prec\R$.
Using the fact that the random variable ${\mathbb U}_1(A_1)$ and
the random variables $Z$, ${\mathbb Z}_3$ and ${\mathbb Z}_4$ are independent
we obtain from (\ref{5.39})
$$
|R_3|
\le
t^2|\alpha^{m/2}(t)|\E |Z{\mathbb Z}_3|
\le t^2e^{-mt^2/8}m^{1/2}N^{-2}
\prec\R.
$$
In the last step we combined (\ref{5.20}), (\ref{5.24}) and Cauchy-Schwartz.
The proof of $R_1\prec\R$ is similar.

In order to prove the first relation of (\ref{5.34})
we expand the exponent in powers
of $it{\mathbb Z}_4$ and obtain
$h_{11}=\E \exp\{it{\mathbb U}_1\}itZ+R$. Furthermore, combining
(\ref{5.39}), (\ref{5.20}) and (\ref{5.24}) we obtain
$$
|R|
\le
t^2|\alpha^m(t)| \E |Z{\mathbb Z}_4|
\le
c_*t^2e^{-mt^2/4N}N^{-3/2}
 \prec\R.
$$
Thus the proof of (\ref{5.27}) is complete.

Let us prove (\ref{5.28}).
By
symmetry and the independence,
\begin{equation}\label{5.35}
\E e^{it{\mathbb U}_1}it{\mathbb U}_3
=
\binom{N}{3}
h_{13}
\E e^{it{\mathbb U}_*},
\qquad
h_{13}=\E e^{itx_1}e^{itx_2}e^{itx_3}itz.
\end{equation}
Here we denote $z=g_3(X_1,X_2,X_3)$ and write,
$$
{\mathbb U}_1=x_1+x_2+x_3+{\mathbb U}_*,
\qquad
{\mathbb U}_*=\sum_{4\le j\le N}g_1(X_j),
\qquad
x_j=g_1(X_j).
$$
 Furthermore, write
$$
r_j=e^{itx_j}-1-itx_j,
\qquad
v_j=e^{itx_j}-1.
$$

In what follows we expand the exponents in powers of $itX_j$, $j=1,2,3$ and
use the fact that $\E\bigl( g_3(X_1,X_2,X_3)\bigl|X_1,X_2\bigr)=0$ as well as
the obvious symmetry. Thus, we have
\begin{align*}
&
h_{13}=h_{14}+R_1,
\qquad
h_{14}=\E e^{itx_2}e^{itx_3}(it)^2zx_1,
\qquad
R_1=\E e^{itx_2}e^{itx_3}itzr_1,
\\
&
h_{14}=h_{15}+R_2,
\qquad
h_{15}=\E e^{itx_3}(it)^3zx_1x_2,
\qquad
R_2=\E e^{itx_3}(it)^2zx_1r_2
\\
&
h_{15}=h_{16}+R_3,
\qquad
h_{16}=\E (it)^4zx_1x_2x_3,
\qquad
R_3=\E (it)^3zx_1x_2r_3.
\end{align*}
Furthermore, we have
$$
R_1=\E itz_1r_1v_2v_3,
\qquad
R_2=\E (it)^2zx_1r_2v_3.
$$
Invoking the bounds $|r_j|\le |tx_j|^2$
 and $|v_j|\le |tx_j|$ we obtain
\begin{equation}\label{5.36}
h_{13}=h_{16}+R,
\end{equation}
where $|R|\le c|t|^{5}\E |z x_1x_2|\,x_3^2$.
The bound, $|R|\le c_*|t|^5N^{-9/2}$ (which follows, by Cauchy-Schwartz)
in combination with (\ref{5.35}) and (\ref{5.36}) implies
\begin{equation}\label{5.37}
\E e^{it{\mathbb U}_1}it{\mathbb U}_3
\sim
\binom{N}{3}
\E e^{it{\mathbb U}_*} (it)^4w.
\end{equation}
Note that $\binom{N}{3}|w|\le c_*N^{-1}$.
In order to show (\ref{5.28}) we replace $\E e^{it{\mathbb U}_*}$ by $e^{-t^2/2}$.
Therefore, (\ref{5.28}) follows from (\ref{5.37}) and the inequalities
$$
\frac{(it)^4}{N}(\E e^{it{\mathbb U}_*}-e^{-t^2\sigma^2(N-3)/2N})\prec{\R},
\quad
\frac{(it)^4}{N}(e^{-t^2\sigma^2(N-3)/2N}-e^{-t^2/2})\prec{\R}.
$$
The second inequality is a direct consequence of (\ref{A1.4}). The proof of the
first inequality is routine and here omitted.
Thus the proof of (\ref{5.14}) is complete.

\bigskip
{\bf 5.2.3.} Here we show that
\begin{equation}\label{5.38}
\E \exp\{it{\mathbb U}_1+{\mathbb U}_2)\}+ \binom{N}{3} e^{-t^2/2}(it)^4w
\sim {\hat G}(t).
\end{equation}
This relation in combination with (\ref{5.13}) and (\ref{5.14}) implies
$\E e^{it{\mathbb T}}\sim {\hat G}(t)$.

Let $G_U(t)$ denote the two term Edgeworth expansion of the  $U$-
statistic ${\mathbb U}_1+{\mathbb U}_2$. That is, $G_U(t)$ is defined by (\ref{1.2}),
but with $\kappa_4$ replaced by
$\kappa_4^*$, where $\kappa_4^*$ is obtained from $\kappa_4$ after removing
the summand $4\E g(X_1)g(X_2)g(X_3)\chi(X_1,X_2,X_3)$.
Furthermore, let ${\hat G}_U(t)$ denote the Fourier transform of $G_U(t)$.
 It easy to show that
$$
{\hat G}(t)={\hat G}_U(t)+ \binom{N}{3} e^{-t^2/2}(it)^4w.
$$
Therefore, in order to prove (\ref{5.38}) it suffices to show
that ${\hat G}_U(t)\sim\E \exp\{it({\mathbb U}_1+{\mathbb U}_2)\}$.
The bound
$$
\int_{|t|\le t_1}
|{\hat G}_U(t)-\E \exp\{it({\mathbb U}_1+{\mathbb U}_2)\}|
\frac{dt}{|t|}
\le \e_NN^{-1}
$$
where $\e_N\downarrow 0$, was shown by Callaert, Janssen and Veraverbeke (1980)
\cite{Callaert_Janssen_Veraverbeke_1980}
and Bickel, G\"otze and van Zwet (1986) \cite{Bickel_Goetze_vanZwet_1986}.
An inspection
 of their proofs shows that under the moment conditions (\ref{1.3}) one can
replace $\e_n$ by $c_*N^{-\nu}$. This completes the proof of (\ref{5.1}).

\bigskip

For the reader convenience  we formulate in Lemma \ref{L5.1} a known result
 on upper bounds for characteristic functions.

\begin{lem}\label{L5.1} Assume that (\ref{2.3}) holds.
There exists a constant $c_*$ depending on
$D_*,M_*, r, s, \nu_1$ only such that, for $N>c_*$ and
$|t|\le N^{1/2}/10^3\beta_3$ and $B\subset \Omega_N$, we have
\begin{equation}\label{5.39}
|\alpha(t)|\le 1-t^2/4N,
\qquad
\E\exp\{it{\mathbb U}_1(B)\}|\le |\alpha(t)|^{|B|}\le e^{-|B|t^2/4N}.
\end{equation}
Here
$\alpha(t)=\E \exp\{itg_1(X_1)\}$ \
and
\
 ${\mathbb U}_1(B)=\sum_{j\in B}g_1(X_j)$.
\end{lem}

\begin{proof}
Let us prove the first inequality of (\ref{5.39}).
Expanding the exponent, see (\ref{A2.7}), we obtain
\begin{align*}
|\alpha(t)|
&
\le
\bigl|1-2^{-1}t^2\E g_1^2(X_1)\bigr|+6^{-1}|t|^3\E|g_1(X_1)|^3
\\
&
=
\bigl|1-\sigma^2t^2/2N\bigr|+\beta_3\sigma^3|t|^3/6N^{3/2}
\end{align*}
Invoking the inequality $1-10^{-3}\le \sigma^2\le 1$ which follows from
(\ref{A1.4}) for $N>c_*$, where $c_*$ is sufficiently large, we obtain
$|\alpha(t)|\le 1-t^2/4N$, for $|t|\le N^{1/2}/10^3\beta_3$.

The second inequality of (\ref{5.39}) follows from the first one via the
inequality $1+x\le e^x$, for $ x\in R$.
\end{proof}

\section{Appendix 1}

{\bf 1.1.} Here we compare the
moments $\Delta_m^2$ and $\E R_m^2$, where $R_m$ denotes the remainder
of the expansion (\ref{2.1}),
 $$
\TT=\E \TT+ \UU_1+\dots+\UU_{m-1}+R_m,
\qquad
R_m:=\UU_m+\dots+\UU_N.
$$
For $k=1,\dots, N$, write $\Omega_k=\{1,2,\dots, k\}$ and denote
$
\sigma_k^2:=\E g_k^2(X_1,\dots, X_k)=\E T_{\Omega_k}^2$.
It follows from
(\ref{2.1}), by the orthogonality property (\ref{2.2}), that
\begin{equation}\label{A1.1}
\sigma_\TT^2
=
\sum_{k=1}^N\E {\mathbb U}_k^2,
\qquad
\E R_m^2
=
\sum_{k=m}^N\E {\mathbb U}_k^2,
\qquad
\E {\mathbb U}_k^2
=
\binom{N}{k}\sigma_k^2.
\end{equation}

\begin{lem}\label{LA1.1} Assume that $\E \TT^2<\infty$. Then
\begin{eqnarray}\label{A1.2}
&&
\E R_m^2\le N^{-(m-1)}\Delta_m^2,
\\
\label{A1.3}
&&
\Delta_m^2\le N^{2m-1}\sigma_m^2+N^{-1}\Delta_{m+1}^2,
\end{eqnarray}
Assume that (\ref{1.3}) and (\ref{1.4}) hold, then there exists a constant
$c_*<\infty$ depending on $D_*,M_*,r,s, \nu_1$ such that
\begin{equation}\label{A1.4}
0\le 1-\sigma^2\sigma_\TT^{-2}\le c_*N^{-1}.
\end{equation}
\end{lem}

{\it Remark.} For $m=3$, the inequality (\ref{A1.3}) yields
$\Delta_3^2\le \zeta_2+N^{-1}\Delta_4^2$.

{\it Proof.} Let us prove (\ref{A1.2}).
The identity
$$
D_1\cdots D_m\TT
=
\sum_{A:\, \Omega_m\subset A\subset \Omega_N}T_A
=
\sum_{m\le k\le N}\UU_{k|m},
$$
where $\UU_{k|m}=\sum_{|A|=k,\, A\supset\Omega_m}T_A$, implies
\begin{equation}\label{A1.5}
\E(D_1\cdots D_m\TT)^2=\sum_{m\le k\le N}\E \UU_{k|m}^2,
\qquad
\E \UU_{k|m}^2=\sigma_k^2 \binom{N-m}{k-m}.
\end{equation}
Write
\begin{equation}\label{A1.6}
\E(D_1D_2\cdots D_m\TT)^2
=
\sum_{m\le k\le N}\sigma_k^2 \binom{N-m}{k-m}
=
\sum_{m\le k\le N}\sigma_k^2 \binom{N}{k} b_k,
\end{equation}
where $b_k=[k]_m/[N]_m$ satisfies $b_k\ge b_m\ge m!N^{-m}$. Here
we denote $[x]_m=x(x-1)\cdots(x-m+1)$.
A comparison of (\ref{A1.1}) and (\ref{A1.6}) shows (\ref{A1.2})
$$
\E R_m^2
\le
N^m\E(D_1\cdots D_m\TT)^2
=
N^{-(m-1)}\Delta_m^2.
$$
Let us prove (\ref{A1.3}). Write
\begin{align*}
\E(D_1\cdots D_m\TT)^2
&
=
\sigma_m^2
+
\sum_{m<k\le N}\sigma_k^2 \binom{N-m}{k-m}
\\
&
=
\sigma_m^2
+
\sum_{m<k\le N}\sigma_k^2 \binom{N-m-1}{k-m-1} {\tilde b}_k,
\end{align*}
where ${\tilde b}_k=(N-m)/(k-m)\le N$. We obtain the inequality
$$
\E(D_1\cdots D_m\TT)^2\le \sigma_m^2+N\E(D_1\cdots D_{m+1}\TT)^2
$$
which implies (\ref{A1.3}).

Let us prove (\ref{A1.4}). From (\ref{A1.1}), (\ref{A1.2}) we have, for
$\sigma^2=N\sigma_1^2$,
$$
0
\le
1-\frac{\sigma^2}{\sigma_{\TT}^2}
\le
\binom{N}{2}\frac{\sigma_2^2}{\sigma_{\TT}^2}
+
\binom{N}{3}\frac{\sigma_3^2}{\sigma_{\TT}^2}
+\frac{1}{N^3} \frac{\Delta_4^2}{\sigma_{\TT}^2}.
$$
Invoking the bounds, which follow from (\ref{1.3}),
$$
N^3\sigma_2^2=\E\psi^2(X_1,X_2)\le M_*^{2/r}\sigma_{\TT}^2,
\qquad
N^5\sigma_3^2=\E\chi^2(X_1,X_2,X_3)\le M_*^{2/s}\sigma_{\TT}^2
$$
and using (\ref{1.4}) we obtain
(\ref{A1.4}).

\bigskip
{\bf 1.2.} Here we prove moments bounds for
various parts of Hoeffding decomposition defined in Section 2.

\begin{lem}\label{LA1.2} Assume that $\sigma_\TT^2=1$. For $3\le m\le N$ and $s>2$, we have
\begin{eqnarray}\label{A1.7}
&&
\E\Lambda_3^2\le \frac{m^3}{N^{5}}\Delta_3^2,
\quad
\E\Lambda_2^2\le \frac{m^2}{N^{4}}\Delta_3^2,
\quad
\E|\Lambda_1|^3\le c\frac{m^3}{N^{9/2}}\gamma_3,
\\
\label{A1.8}
&&
\E |\Lambda_4|^s
\le 
c(s)\, m^{s/2}N^{-3s/2}\zeta_s,
\quad
\E\eta_i^2
\le N^{-4}\Delta_4^2,
\quad
\E \Lambda_5^2\le mN^{-4}\Delta_4^2.
\end{eqnarray}
Here $c$ denotes an absolute constant and $c(s)$ denotes a
constant which depends only on $s$.
\end{lem}
\begin{proof} The inequalities (\ref{A1.7}) are proved in \cite{Bentkus_Goetze_vanZwet_1997}.

Let us prove (\ref{A1.8}).
Split $\Lambda_4=z_1+\dots+z_m$, where
$$
z_i=\sum_{|A|=3,\, A\cap \Omega_m=i}T_A.
$$
Let $\E'$ denote the conditional expectation given $X_{m+1},\dots,
X_N$.
It follows from Rosenthal's inequality that almost surely
$$
\E'|\Lambda_4|^s
\le
c(s)\sum_{i=1}^m\E'|z_i|^s+c(s)\bigl(\sum_{i=1}^m\E' z_i^2\bigr)^{s/2}.
$$
Invoking H\"older's inequality we obtain, by symmetry,
\begin{equation}\label{A1.9}
\E|\Lambda_4|^s
=
\E\E'|\Lambda_4|^s
\le
c(s)m^{s/2}\E|z_1|^s.
\end{equation}
Using well known martingale moment inequalities (and their applications to $U$ statistics), see  Dharmadhikari, Fabian and Jogdeo
(1968) \cite{Dharmadhikari_Fabian_Jogdeo_1968}, one can show the bound
$\E|z_1|^s\le c(s)N^{-3s/2}\zeta_s$.
Invoking this bound in (\ref{A1.9}) we obtain the first bound of
(\ref{A1.8}).

In order to prove the second bound of (\ref{A1.8}) write
$$
 \eta_i=\sum_{k=4}^{N-m+1} U^*_k,
 \qquad
 U^*_k=\sum_{|A|=k,\, A\cap \Omega_m=\{i\}}T_A.
 $$
 A simple calculation shows
 $\E(U^*_k)^2= \binom{N-m}{k-1}\sigma_k^2$. Therefore, by
 orthogonality,
\begin{eqnarray}\nonumber
 \E
 \eta_i^2
&&
 =
 \sum_{k=4}^{N-m+1} \binom{N-m}{k-1}\sigma_k^2
 =
 \sum_{k=4}^{N-m+1} \binom{N-4}{k-4} b_k\sigma_k^2
\\
\label{A1.10}
&&
 \le
 N^3\E(D_1\cdots D_4\TT)^2.
\end{eqnarray}
 In the last step we invoke (\ref{A1.5}) and use  the bound $b_k\le N^3$, where
  $b_k= \binom{N-m}{k-1} \binom{N-4}{k-4}^{-1}$.
Clearly, (\ref{A1.10}) implies $\E \eta_i^2\le N^{-4}\Delta_4^2$. Finally, using the fact that
$\eta_1,\dots, \eta_m$ are uncorrelated we obtain
$$
\Lambda_5^2=\E\eta_1^2+\dots+\E\eta_m^2\le m\,N^{-4}\Delta_4^2.
$$
 thus completing the proof.
\end{proof}

Before formulating  next result we introduce some notation.
Given $m$ let $\D$ denote the class of subsets $A\subset \Omega_N$
satisfying $|A|\ge 4$ and $\Omega_m\cap A\not=\emptyset$.
Introduce the random variable ${\mathbb H}(m)=\sum_{A\in \D}T_A$.
Denote $x_i=2i-1$ and $y_i=2i$. For even integer $m=2k\le N$ write
$$
\Omega_m=A_k\cup B_k,
\qquad
A_k=\{x_1,\dots, x_k\},
\qquad
B_k=\{y_1,\dots, y_k\}
$$
and put $A_0=B_0=\emptyset$.
Let ${\A}(k)$ (respectively ${\B}(k)$)  denote the
 collection of those $A\in \D$ which satisfy $A\cap A_k=\emptyset$
(respectively $A\cap B_k=\emptyset$). Furthermore, let ${\C}(k)$ denote
the collection of  $A\in {\D}$ such that $A\cap A_k\not=\emptyset$ and
$A\cap B_k\not=\emptyset$. Write
$$
{\mathbb H}_A(k)=\sum_{A\in {\A}(k)}T_A,
\qquad
{\mathbb H}_B(k)=\sum_{A\in {\B}(k)}T_A,
\qquad
{\mathbb H}_C(k)=\sum_{A\in {\C}(k)}T_A.
$$
\begin{lem}\label{LA.1.3} There exists an absolute constant $c$ such that,
\begin{equation}\label{A1.11}
\E {\mathbb H}^2(m)\le c\frac{m}{N^4}\Delta_4^2,
\qquad
{\text{for}}
\qquad m=4,5,\dots, N.
\end{equation}
For even integer $m=2k<N$ we have
\begin{equation}\label{A1.12}
\E {\mathbb H}_A^2(k)=\E {\mathbb H}_B^2(k)\le c\frac{k}{N^4}\Delta_4^2,
\qquad
\E {\mathbb H}_C^2(k)\le c \frac{k^2}{N^5}\Delta_4^2.
\end{equation}
\end{lem}
\begin{proof}
Let us prove the first bound of (\ref{A1.11}). For $m=4$ we have
$$
{\mathbb H}(4)=H_1+H_2+H_3+H_4,
\qquad
H_k=\sum_{|A|\ge 4,\, |A\cap \Omega_4|=k}T_A.
$$
A calculation shows that, for $k=1,2,3,4$,
$$
\E H_k^2
=
\binom{4}{k}\sum_{j=4}^N\sigma_j^2 \binom{N-4}{j-k}
=
\binom{4}{k}\sum_{j=4}^N\sigma_j^2 \binom{N-4}{j-4} a_k(j),
$$
where the numbers
$$
a_k(j)
=
\frac{ \binom{N-4}{j-k} }{  \binom{N-4}{j-4} }
\le N^{4-k}.
$$
Invoking (\ref{A1.6}) we obtain
\begin{equation}\label{A1.13}
\E H_k^2\le c\,N^{4-k} \E(D_1\cdots D_4\TT)^2=cN^{-3-k}\Delta_4^2.
\end{equation}
Finally, we obtain (\ref{A1.11}) for $m=4$
$$
\E {\mathbb H}^2(4)=\E H_1^2+\dots+\E H_4^2\le cN^{-4}\Delta_4^2.
$$

In order to prove (\ref{A1.11}) for
$m=5,6,\dots$ we apply a recursive argument. Write
\begin{equation}\label{A1.14}
\E {\mathbb  H}^2(m+1)=\E {\mathbb  H}^2(m)+\E d_m^2,
\end{equation}
where $d_m={\mathbb H}(m+1)-{\mathbb H}(m)$ is the sum of those $T_A$ with $|A|\ge 4$
satisfying
$A\cap \Omega_m=\emptyset$ and $A\cap \Omega_{m+1}\not=\emptyset$.
In particular, we have
$$
d_m=\sum_{|A|\ge 3,\, A\cap \Omega_{m+1}=\emptyset}T_{A\cup\{m+1\}}.
$$
Therefore,
$$
\E d_m^2
=
\sum_{j=4}^N\sigma_j^2 \binom{N-m-1}{j-1}
=
\sum_{j=4}^N\sigma_j^2 \binom{N-4}{j-4} c_j,
$$
where the numbers
$$
c_j=\frac{ \binom{N-m-1}{j-1} }{  \binom{N-4}{j-4} }\le N^3.
$$
Invoking (\ref{A1.6}) we obtain $\E d_m^2\le N^{-4}\Delta_4^2$.
This bound together with  (\ref{A1.14}) implies (\ref{A1.11}).

Let us prove (\ref{A1.12}). Note that for $m=2k$ we have
${\mathbb H}(m)={\mathbb H}_A(k)+{\mathbb H}_B(k)+{\mathbb H}_C(k)$ and the summands are
 uncorrelated. Therefore, the first bound of (\ref{A1.12}) follows from (\ref{A1.11}).

Let us show the second inequality of (\ref{A1.12}).
For $k=2$ we have ${\C}(2)\subset {\C}$, where $\C$ denotes the class of subsets
$A\subset \Omega_N$ such that $|A|\ge 4$ and $|A\cap \Omega_4|\ge 2$.
Write ${\mathbb H}_C=\sum_{A\in \C}T_A$.
We have
$$
\E {\mathbb H}_C^2(2)\le \E {\mathbb H}_C^2=\E H^2_2+\E H^2_3+\E H_4^2
\le
cN^{-5}\Delta_4^2.
$$
In the last step we applied (\ref{A1.13}). We obtain (\ref{A1.12}), for $k=2$.

In order to prove the bound (\ref{A1.12}), for $k=3,4,\dots$, we apply a recursive
argument similar to that used in the proof of (\ref{A1.11}).
Denote
$$
d_{[k]}={\mathbb H}_C(k+1)-{\mathbb H}_C(k)=\sum_{A\in {\C}(k+1)\setminus{\C}(k)}T_A.
$$
We shall show that
\begin{equation}\label{A1.15}
\E d_{[k]}^2\le ckN^{-5}\Delta_4^2.
\end{equation}
This bound in combination with the identity
$\E {\mathbb H}_C^2(k+1)=\E {\mathbb H}_C^2(k)+\E d_{[k]}^2$ shows (\ref{A1.12})
for arbitrary $k$.

In order to show (\ref{A1.15}) split the set ${\C}(k+1)\setminus{\C}(k)$
 into $2k+1$ non-intersecting parts
$$
{\C}(k+1)\setminus{\C}(k)
=
\bigl(\cup_{i=1}^k{\C}_{x.i}\bigr)
\cup
\bigl(\cup_{i=1}^k{\C}_{y.i}\bigr)
\cup
{\C}_{x.y},
$$
where we denote
\begin{align*}
&
{\C}_{x.y}=\bigl\{A={\tilde A}\cup\{x_{k+1},y_{k+1}\}:
\ {\tilde A}\cap(B_k\cup A_k)=\emptyset,\  |\tilde A|\ge 2\bigr\},
\\
&
{\C}_{x.i}=\bigl\{A={\tilde A}\cup\{y_{k+1},x_i\}:
\ {\tilde A}\cap(B_k\cup A_{i-1})=\emptyset,\  |\tilde A|\ge 2\bigr\},
\\
&
{\C}_{y.i}=\bigl\{A={\tilde A}\cup\{x_{k+1},y_i\}:
\ {\tilde A}\cap(B_{i-1}\cup A_k)=\emptyset,\  |\tilde A|\ge 2\bigr\}.
\end{align*}
By the orthogonality property ($\E T_AT_{V}=0$ for $A\not=V$),
the random variables
$$
d_{x.i}=\sum_{A\in {\C}_{x.i}}T_A,
\qquad
d_{y.i}=\sum_{A\in {\C}_{y.i}}T_A,
\qquad
d_{x.y}=\sum_{A\in {\C}_{x.y}}T_A
$$
are uncorrelated. Therefore, we have
\begin{equation}\label{A1.16}
\E d_{[k]}^2=\E d_{x.y}^2+\sum_{i=1}^k(\E d_{x.i}^2+\E d_{y.i}^2).
\end{equation}
A calculation shows that
$$
\E d_{x.y}^2
=
\sum_{j=4}^N\sigma_j^2 \binom{N-2k-2}{j-2}
=
\sum_{j=4}^N\sigma_j^2 \binom{N-4}{j-4}v_j,
$$
where the coefficients
$$
v_j=\frac{ \binom{N-2k-2}{j-2} }{ \binom{N-4}{j-4} }\le N^2.
$$
Invoking (\ref{A1.6}) we obtain $\E d_{x.y}^2\le N^{-5}\Delta_4^2$.
The same argument shows $\E d_{x.i}^2=\E d_{y.i}^2\le N^{-5}\Delta_4^2$.
The latter bound in combination with (\ref{A1.16}) shows (\ref{A1.15}). The lemma is
proved.
\end{proof}


\section{Appendix 2}

Here we construct bounds for the probability density function (and its derivatives) of
random variables $g_k^*=(N/M)^{1/2}g_k$,  for $1\le k\le n-1$, where $g_k$ are defined
in (\ref{3.46}).  Since these random variables are
identically distributed it suffices to consider
$$
g_1^*
=
\bigl(\frac{N}{M}\bigr)^{1/2}g_1
=
\frac{1}{\sqrt{M}}\sum_{j=m+1}^{m+M}g(Y_j)
+
\frac{\xi_1}{R}.
$$
Here $R=\sqrt{n\,M\,N\,}$. Introduce the random variables
$$
g_2^*=g_1^*-M^{-1/2}g(Y_{m+1}),
\qquad
g_3^*=g_1^*-M^{-1/2}\bigl(g(Y_{m+1})+g(Y_{m+2})\bigr).
$$
Let $p_i(\cdot)$  denote the probability density function of $g_i^*$,
for $i=1,2,3$.
Recall that the integers $n\approx N^{50\nu}\le N^{\nu_2/10}$ and $M\approx N/n \ge N^{9/10}$
 are introduced in (\ref{3.2}) and the number
$\nu>0$ is defined by (\ref{2.4}).

\begin{lem}\label{LA2.1} Assume that conditions of Theorem 1 are
satisfied.
There exist positive constants $C_*, c_*, c_*'$ depending only on
$M_*, D_*, \delta, r$ and $\nu_1, \nu$ such that, for $i=1,2,3$, we have uniformly in $u\in \RR$
and $N>C_*$
\begin{equation}\label{A2.1}
|p_i(u)|\le c_*,
\quad
|p_i'(u)|\le c_*,
\quad
|p_i''(u)|\le c_*,
\quad
|p_i'''(u)|\le c_*.
\end{equation}
Furthermore, given $w>0$ there exists a constant $C_*(w)$ depending on
$M_*, D_*, \delta, r$, $\nu_1, \nu$  and $w$ such that
uniformly in $z_*\in [-2w,2w]$ and $N>C_*(w)$ we have
\begin{equation}\label{A2.2}
p_i(z_*)\ge c'_*,
\qquad
i=1,2,3.
\end{equation}
\end{lem}
{\it Proof.}
We shall prove (\ref{A2.1}) and (\ref{A2.2}) for $i=1$. For $i=2,3$, the proof is almost the same.
Before starting the proof we  introduce some notation and collect  auxiliary results.

Denote
\begin{align*}
&
\theta=\E g_1^*=M^{1/2}\theta_1,
\qquad
\theta_1=\E g(Y_{m+1}),
\\
&
s^2=\E(g(Y_{m+1})-\theta_1)^2,
\qquad
{\tilde \beta}_3=s^{-3}\E|g(Y_{m+1})-\theta_1|^3.
\end{align*}
It follows from $\E g(X_{m+1})=0$
that
$$
\theta_1
=
q_N^{-1}\E g(X_{m+1})\II_{A_{m+1}}
=
-q_N^{-1}\E g(X_{m+1})(1-\II_{A_{m+1}}).
$$
Therefore, by Chebyshev's inequality, for $\alpha=3/(r+2)$ we have
\begin{equation}\label{A2.3}
|\theta_1|\le
q_{N}^{-1}N^{-\alpha(r-1)}\E|g(X_{m+1})|\,\|Z_{m+1}'\|_r^{r-1}
\le
c_*N^{-3/2}.
\end{equation}
In the last step we invoke the inequalities $\alpha(r-1)\ge 1+(r-1)/(r+2)\ge 3/2$
and $q_N^{-1}\le c_*$, see (\ref{3.17}), and $\E|g(X_{m+1})|\,\|Z_{m+1}'\|_r^{r-1}\le M_*$,
where the latter inequality follows from (\ref{1.3}) by H\"older
inequality.

Similarly, the identities
$$
s^2
=
q_N^{-1}\E g^2(X_{m+1})\II_{A_{m+1}}-\theta_1^2
=
q_N^{-1}\sigma^2-q_N^{-1}\E
g^2(X_{m+1})(1-\II_{A_{m+1}})-\theta_1^2
$$
in combination with (\ref{3.16}) and the inequalities
$$
\E g^2(X_{m+1})(1-\II_{A_{m+1}})
\le
N^{-\alpha(r-2)}
\E g^2(X_{m+1})\,\|Z_{m+1}'\|_r^{r-2}\le N^{-\alpha(r-2)}M_*
$$
and  $\alpha(r-2)=1+2(r-4)/(r-2)\ge 1$ yield
\begin{equation}\label{A2.4}
|s^2-\sigma^2|\le c_*N^{-1}.
\end{equation}

Introduce the random variables
$$
g_*=S+\frac{\xi_1}{s\,R},
\quad
S=w_1+\dots+w_M,
\quad
w_j=\frac{g(Y_{m+j})-\theta_1}{M^{1/2}s}.
$$
We have  $g_*=s^{-1}(g_1^*-\theta)$. Let $p(\cdot)$ denote the density function
of $g_*$. Note that $p_1(u)=s^{-1}p\bigl(s^{-1}(u-\theta)\bigr)$. Furthermore, we
have, by (\ref{A2.3}),
$|\theta|\le c_*N^{-1}$ and, by (\ref{A2.4}), (\ref{A1.4}), $|s^2-1|\le c_*N^{-1}$.
Therefore, it suffices to prove (\ref{A2.1})
and (\ref{A2.2})  for $p(\cdot)$ (the latter inequality we verify for every $z_*\in [-3w,3w]$).

\smallskip

In order to prove (\ref{A2.1}) and (\ref{A2.2}) we approximate the
characteristic function ${\hat p}(t)=\E e^{itg_*}$ by $e^{-t^2/2}$
 and then apply  a Fourier inversion formula.
Write
$$
{\hat p}(t)
=
\E e^{itg_*}
=
\gamma^M(t)\tau\Bigl(\frac{t}{s R}\Bigr),
\qquad
\gamma(t):=\E e^{itw_1},
\quad
\tau(t):=\E e^{it\xi_1}.
$$
The fact that $\tau(t)=0$, for $|t|\ge 1$, implies
${\hat p}(t)=0$, for $|t|>s\, R$.
Therefore,
we obtain from  the Fourier inversion formula,
$$
p(x)=
\frac{1}{2\pi}\int_{-\infty}^{+\infty}
e^{-itx}{\hat p}(t)dt
=
\frac{1}{2\pi}\int_{-s\, R}^{s\, R}
e^{-itx}{\hat p}(t)dt.
$$

Write ${\hat p}(t)-e^{-t^2/2}=r_1(t)+r_2(t)$, where
$$
r_1(t)=(\gamma^M(t)-e^{-t^2/2})\tau(t/sR),
\qquad
r_2(t)=e^{-t^2/2}\bigl(\tau(t/sR)-1\bigr).
$$
We shall show below that
\begin{equation}\label{A2.5}
\int_{|t|\le sR}|r_i(t)|dt\le c_*M^{-1/2},
\qquad
i=1,2.
\end{equation}
These bounds in combination with the simple inequality
$$
\int_{|t|\ge sR}e^{-t^2/2}dt\le c_*M^{-1/2}
$$
show that
\begin{equation}\label{A2.6}
|p(x)-\varphi(x)|\le c_*M^{-1/2},
\qquad
 x\in \RR.
\end{equation}
 Here $\varphi$ denotes the standard normal density function
 $$
 \varphi(x)
 =
 \frac{1}{\sqrt{2\pi}}e^{-x^2/2}
 =
 \frac{1}{2\pi}\int_{-\infty}^{+\infty}e^{-itx}e^{-t^2/2}dt
$$

It follows from (\ref{A2.6}) that
$$
|p(x)|\le c_*,
\qquad
x\in \RR.
$$
 Furthermore,
given $w$  we have uniformly in
$|z_*|\le 3w$
$$
|p(z_*)|\ge \varphi(3w)-c_*M^{-1/2}\ge c_*'>0,
$$
for sufficiently large $M$ (for $N>C_*(w)$).

In order to prove an upper bounds for the $k-$th derivative,
$|p^{(k)}(x)|\le c_*$, write
$$
p^{(k)}(t)
=
\frac{1}{2\pi}\int_{-\infty}^{+\infty}(-it)^k\exp\{-itx\}{\hat p}(t)dt,
\qquad
k=1,2,3,
$$
and replace ${\hat p}(t)$ by $e^{-t^2/2}$ as in the proof of
(\ref{A2.6}).
We obtain
$$
p^{(k)}(x)
=
\frac{1}{2\pi}\int_{-\infty}^{+\infty}(-it)^k\exp\{-itx\}e^{-t^2/2}(t)dt
+r,
\qquad
|r|\le c_*M^{-1/2}.
$$
This implies $|p^{(k)}(x)-\varphi^{(k)}(x)|\le c_*M^{-1/2}$. We arrive at the desired bound
$|p^{(k)}(x)|\le c_*$, for $k=1,2,3$.

In the remaining part of the proof we verify (\ref{A2.5}).
For $i=2$ this bound follows from
$|\tau(t/sR)-1|\le ct^2/(sR)^2$. The latter inequality is a consequence of the short expansion
$$
\bigl|
\E \exp\{it\xi_1/sR\}-1-\E it\xi_1/sR
\bigr|
\le \E(t\xi_1)^2/2(sR)^2
$$
and  $\E \xi_1=0$ and  $\E \xi_1^2\le c$,
for some absolute constant $c$.

Let us prove (\ref{A2.5}) for $i=1$.
Introduce the sequence of i.i.d. centered Gaussian random variables
$\eta_1,\,\eta_2,\,\dots$ with variances $\E \eta_i^2=M^{-1}$.
Denote
$$
f(t)=\E e^{it\eta_1}= e^{-t^2/(2M)}
\qquad
{\text{and}}
\qquad
\delta(t)=\gamma(t)-f(t).
$$
We are going to apply the well known inequality
\begin{equation}\label{A2.7}
\bigl|
e^{iv}
-
\bigl(
1+\frac{iv}{1!}+\frac{(iv)^2}{2!}+\dots+\frac{(iv)^{k-1}}{(k-1)!}
\bigr)
\bigr|
\le
\frac{|v|^k}{k!}.
\end{equation}
It follows from (\ref{A2.7}) and   identities $\E \eta_1^i=\E w_1^i$,
$i=1,2$, that
\begin{equation}\label{A2.8}
|\delta(t)|
\le
\frac{|t|^3}{3!}\bigl(\E |w_1|^3+\E |\eta_1|^3\bigr)\le c|t|^3\E|w_1|^3.
\end{equation}
Here we use the inequality $\E|\eta_1|^3\le c\E|w_1|^3$, which follows
from $\E \eta_1^2=\E w_1^2$.

Combining (\ref{A2.8}) and the simple identity
$$
\gamma^M(t)-f^M(t)
=
\delta(t) \sum_{k=1}^M\gamma^{M-k}(t)f^{k-1}(t)
$$
we obtain
\begin{equation}\label{A2.9}
|\gamma^M(t)-f^M(t)|\le c|t|^3Z(t)\,M^{-1/2}{\tilde \beta}_3.
\end{equation}
Here we denote
$$
Z(t)=\max_{r+v=M-1}|f^r(t)\gamma^v(t)|.
$$
We shall show below that
\begin{equation}\label{A2.10}
Z(t)\le \exp\Bigl\{-\frac{t^2}{3}\frac{M-1}{M}\Bigr\}+\exp\{-\delta''(M-1)/2\},
\qquad
0\le |t|\le sR,
\end{equation}
where $\delta''>0$ depends on $\delta, A_*, D_*, M_*, \nu_1$ and it is
given in (\ref{3.9}).
This inequality in combination with (\ref{A2.9}) proves (\ref{A2.5}).

Let us prove (\ref{A2.10}).
Clearly, $Z\le |f^{M-1}(t)|+|\gamma^{M-1}(t)|$. Furthermore,
$f^M(t)=e^{-t^2/2}$. In order to prove (\ref{A2.10}) we shall show
\begin{eqnarray}\label{A2.11}
&&
|\gamma^M(t)|\le e^{-t^2/3},
\qquad
0\le |t|\le M^{1/2}/{\tilde \beta}_3,
\\
\label{A2.12}
&&
|\gamma(t)|\le e^{-\delta''/2},
\qquad
M^{1/2}/{\tilde \beta}_3\le |t|\le sR.
\end{eqnarray}
To show (\ref{A2.11}) we expand $e^{itw_1}$ using (\ref{A2.7}),
\begin{align*}
|\gamma(t)|
=
\bigl| \E e^{itw_1}\bigr|
&
\le
\bigl| 1-\frac{t^2}{2}\E w_1^2 \bigr|+\frac{|t|^3}{3!}\E |w_1|^3
\\
&
=
\bigl| 1-\frac{t^2}{2M} \bigr|+\frac{|t|^3}{3!}\frac{{\tilde \beta}_3}{M^{3/2}}
\\
&
=
1-\frac{t^2}{2M}\bigl(1-\frac{|t|}{3}\frac{{\tilde \beta}_3}{\sqrt M}\bigr)
\\
&
\le 1-\frac{t^2}{3M}.
\end{align*}
Here we used the
identity $|1-t^2/2M|=1-t^2/2M$,
which holds for $|t|<M^{1/2}/{\tilde \beta}_3$, since ${\tilde \beta}_3\ge 1$.
Finally, an application of the inequality $1-x\le e^{-x}$ to
$x=t^2/3M>0$
completes the proof of (\ref{A2.11}).

Let us prove (\ref{A2.12}).
For $\delta''$ defined by (\ref{3.9}) we shall show $\delta''\le 2{\tilde \delta}$, where
\begin{align*}
{\tilde \delta}
&
=
1-\sup\{|\gamma(t)|:\, M^{1/2}{\tilde \beta}_3^{-1}\le |t|\le sR\}
\\
&
=
1
-
\sup\{
|\E\exp\{iu\sigma^{-1}g(Y_{m+1})|:\, \sigma/s\,{\tilde \beta}_3\le |u|\le \sigma \sqrt {n\, N}
\}.
\end{align*}
We are going to replace $g(Y_{m+1}),\, {\tilde \beta}_3,\, s^2$
by $X_{m+1}, \,\beta_3, \,\sigma^2$
respectively. Write
\begin{align*}
&
\E e^{ivg(Y_{m+1})}
=
q_N^{-1}\E e^{ivg(X_{m+1})}\II_{A_{m+1}}
=
\E e^{ivg(X_{m+1})}+r_1+r_2,
\\
&
r_1=q_N^{-1}\E e^{ivg(X_{m+1})}\bigl(\II_{A_{m+1}}-1\bigr),
\qquad
r_2=(q_N^{-1}-1)\E e^{ivg(X_{m+1})}.
\end{align*}
It follows from (5.A5) that, for every $v\in \RR$,
$$
|r_1|\le q_N^{-1}\E|\II_{A_{m+1}}-1|=q_N^{-1}-1\le c_*N^{-2},
\qquad
|r_2|\le q_N^{-1}-1\le c_*N^{-2}.
$$
These bounds imply
\begin{equation}\label{A2.13}
|\E e^{ivg(Y_{m+1})}-\E e^{ivg(X_{m+1})}|\le c_*N^{-2},
\qquad
{\text{ for
every}}
\qquad
v\in \RR.
\end{equation}
One can show that, for sufficiently large $N$ (i.e., for $N>C_*$), we have
\begin{equation}\label{A2.14}
|{\tilde \beta}_3/\beta_3-1|<1/5,
\qquad
|s^2/\sigma^2-1|<1/5,
\qquad
|s^2-1|\le 1/5.
\end{equation}
Using (\ref{A2.13}), (\ref{A2.14}) we get, for $N>C_*$,
\begin{align*}
{\tilde \delta}
&
\ge
1
-
\sup\bigl\{
|\E e^{iu\sigma^{-1}g(Y_{m+1})}|:\, (2\beta_3)^{-1}\le |u|\le  N^{(1+50\nu)/2}
\bigr\}
\\
&
\ge
1-
\sup\bigl\{
|\E e^{iu\sigma^{-1}g(X_{m+1})}|:\, (2\beta_3)^{-1}\le |u|\le  N^{(\nu_2+1)/2}
\bigr\}-c_*N^{-2}
\\
&
\ge \delta''/2.
\end{align*}
We obtain  $|\gamma(t)|\le 1-{\tilde \delta}\le 1-\delta''/2$ and,
therefore, $|\gamma(t)|\le e^{-\delta''/2}$.


\section{Appendix 3}

The main results of this section are moment inequalities of Lemma
\ref{LA3.2} and corresponding inequalities for conditional moments of
Lemma \ref{LA3.3}. Lemma \ref{LA3.1} provides an auxiliary inequality.

We start with some notation.
We call $v=v(\cdot),u=u(\cdot)\in L^r$ orthogonal if $\langle u,v \rangle=0$, where
$$
\langle u,v \rangle=\int_{\X}u(x)v(x)P_X(dx)=\E u(X_1)v(X_1).
$$
Given $f\in L^2(P_X)$ we have for the kernel $\psi^{**}$ defined
in (\ref{3.14})
$$
\E
\psi^{**}(X_1,X_2)\bigl( f(X_1)g(X_2)+f(X_2)g(X_1)\bigr)=0
$$
and almost surely
\begin{eqnarray}\label{A3.1}
&&
\E(\psi^{**}(X_1,X_2)|X_1)=0,
\\
\label{A3.2}
&&
 \E\bigl(\psi^{**}(X_1,X_2)g(X_1)|X_2\bigr)
=
0.
\end{eqnarray}
The latter identity says that almost all values of the $L^r$-valued random variable
$\psi^{**}(\cdot,X_2)$ are orthogonal to the  vector $g(\cdot)\in L^r$.

Let
$
p_g:L^r\to L^r
$
denote the projection on the subspace of elements $u\in L^r$ which are
orthogonal to $g=g(\cdot)$. For $v\in L^r$, write $v^*=p_g(v)$.
It follows from  (\ref{A3.2}) that
\begin{equation}\label{A3.3}
\psi^*(\cdot, Y_j)
\,
\Bigl(=p_g\bigl( \psi(\cdot, Y_j)\bigr)\Bigr)
=
\psi^{**}(\cdot, Y_j)+g(Y_j) b^*(\cdot),
\end{equation}
where
$
b^*(\cdot)=p_g(b(\cdot))=
\sigma^{-2}p_g\bigl(\, \E(\psi(\cdot,X_1)g(X_1)\,
\bigr)$. Denote
$$
U_k^*
\,
\bigl(=
p_g(U_k)
\bigr)=\frac{1}{\sqrt N}\sum_{j\in O_k}\psi^*(\cdot, Y_j),
\qquad
U_k^{**}=N^{-1/2}\sum_{j\in O_k}\psi^{**}(\cdot,Y_j),
$$
where the $L^r$-valued random variables
$U_k$
are introduced in (\ref{3.63}). For  the random variables $g_k$ and $L_k$ introduced in (\ref{3.44}) and
(\ref{3.46}),
we have
\begin{equation}\label{A3.4}
U_k^*
=
U_k^{**}+L_k\, b^*(\cdot)
=
U_k^{**}+(g_k-\frac{1}{\sqrt n}\frac{\xi_k}{N})b^*(\cdot).
\end{equation}

Denote $K=\E|\psi(X_1,X_2)|^r$ and $K_s=\E|\psi^{**}(X_1,X_2)|^s$, $s\le r$.

\begin{lem}\label{LA3.1} Let $4<r\le 5$. For $s\le r$, we have
\begin{equation}\label{A3.5}
K_s^{r/s}
\le
K_r\le
c\,K\Bigl(1+\frac{\E|g(X_1)|^r}{\sigma^r}\Bigr)^2.
\end{equation}
\end{lem}
\begin{proof}
The first inequality of (\ref{A3.5}) is a consequence of
Lyapunov's inequality. Let us prove the second inequality.
The inequality $|a+b+c|^r\le 3^r(|a|^r+|b|^r+|c|^r)$ implies
$$
K_r
=
\E|\psi^{**}(X_1,X_2)|^r
\le
 3^r\bigl(K+2\E|b(X_1)|^r\E|g(X_2)|^r\bigr).
$$
Therefore, (\ref{A3.5}) is a consequence of the inequalities
\begin{align*}
&
\E|b(X_1)|^r\le
\frac{2^r}{\sigma^r}K+\frac{|\kappa|^r}{\sigma^{4r}}\E|g(X_1)|^r,
\\
&
\kappa^2\le \sigma^4\E\psi^2(X_1,X_2)\le \sigma^4K^{2/r}.
\end{align*}
Here $\kappa=\E \psi(X_1,X_2)g(X_1)g(X_2)$.
To prove the first inequality use
$|a+b|^r\le 2^r(|a|^r+|b|^r)$ to get
$$
\E|b(X_1)|^r\le
\frac{2^r}{\sigma^{2r}}\E\bigl|\E\bigl(\psi(X_1,X_2)g(X_2)\bigr| X_1\bigr)\bigr|^r
+
\frac{\kappa^r}{\sigma^{4r}}\E|g(X_1)|^r.
$$
Furthermore, by Cauchy--Schwartz,
$$
\bigl|\E\bigl(\psi(X_1,X_2)g(X_2)\,| \, X_1\bigr)\bigr|
\le
\bigl(\E(\psi^2(X_1,X_2)\,|\,X_1)\bigr)^{1/2}\sigma.
$$
Finally, Lyapunov's inequality implies
$$
\bigl(\E(\psi^2(X_1,X_2)\,|\,X_1)\bigr)^{r/2}
\le
\E(|\psi(X_1,X_2)|^r\,|\,X_1).
$$
We obtain
$\E\bigl|\E\bigl(\psi(X_1,X_2)g(X_2)\,| \,X_1\bigr)\bigr|^r\le K\sigma^r$
thus completing the proof.
\end{proof}

\smallskip
%
\begin{lem}\label{LA3.2} Let $1\le k\le n-1$. For ${\overline U}_k^*$,
 an independent copy of $U_k^*$,  we have
\begin{eqnarray}\label{A3.6}
&&
2\delta_3^2-\frac{c_*}{N^{\alpha(r-4)}}
\le
\frac{N}{M}
\E\|U_k^*-{\overline U}_k^*\|_2^2
\le
2\delta_3^2+c_*,
\\
\label{A3.7}
&&
\E\|U_k-{\overline U}_k\|_r^r\le c_*\Bigl(\frac{M}{N}\Bigr)^{r/2}.
\end{eqnarray}
Recall that $\delta_3^2=\E|\psi^{**}(X_1,X_2)|^2$.
\end{lem}
\begin{proof}
Let us prove (\ref{A3.6}).
By symmetry, we have, for $i,j\in O_1$,
\begin{align*}
&
\E\|U_1^*-{\overline U}_1^*\|_2^2
=
2\frac{M}{N}H_1-2\frac{M}{N}H_2,
\\
&
H_1
:=
\E\|\psi^*(\cdot, Y_j)\|_2^2,
\qquad
H_2
:=
\E\langle\psi^*(\cdot, Y_j),\psi^*(\cdot, Y_i)\rangle,
\quad
i\not=j.
\end{align*}

The inequality (\ref{A3.6}) follows from the inequalities
\begin{eqnarray}\label{A3.8}
&&
\delta_3^2-c_*N^{-\alpha(r-4)}
\le
H_1
\le
\delta_3^2+c_*,
\\
\label{A3.9}
&&
H_2
\le
 c_*N^{-\alpha(r-2)}.
\end{eqnarray}
Let us prove (\ref{A3.8}). From (\ref{A3.3}) we have
$H_1=V_1+V_2+2V_3$, where
$$
V_1=
\E\|\psi^{**}(\cdot, Y_j)\|_2^2,
\quad
V_2=\|b^*(\cdot)\|_2^2\E g^2(Y_j),
\quad
V_3
=
\E
g(Y_j)\langle\psi^{**}(\cdot,Y_j), b^*(\cdot)\rangle.
$$
Let us show that
\begin{equation}\label{A3.10}
\delta_3^2-c_*N^{-\alpha(r-2)}\le V_1\le \delta_3^2+c_*N^{-\alpha\,r}.
\end{equation}
This inequality follows from (\ref{3.16}), (\ref{3.17}) and the identity
$$
V_1
=
q_{N}^{-1}\E|\psi^{**}(X_1,X_j)|^2{\mathbb  I}_{A_j}
=
q_{N}^{-1}\E|\psi^{**}(X_1,X_j)|^2-q_{N}^{-1}V_1',
$$
where $V_1'=\E|\psi^{**}(X_1,X_j)|^2(1-{\mathbb  I}_{A_j})$
satisfies, by (\ref{3.15}),
\begin{equation}\label{A3.11}
0
\le
V_1'
\le
N^{-\alpha(r-2)}\E|\psi^{**}(X_1,X_j)|^2 \|Z_j'\|_r^{r-2}
\le
c_*N^{-\alpha(r-2)}.
\end{equation}
 In the last step we applied
H\"older's inequality and Lemma \ref{LA3.1} to get
$$
\E|\psi^{**}(X_1,X_j)|^2 \|Z_j'\|_r^{r-2}
\le
K_r^{2/r}(\E \|Z_j'\|_r^r)^{(r-2)/r}
\le
K_r^{2/r}K^{(r-2)/r}
\le
c_*.
$$

Let us show that
\begin{equation}\label{A3.12}
0\le V_2\le c_*.
\end{equation}
For
${\tilde b}(\cdot):=
\E\psi(\cdot,X_1)g(X_1)$ we have, by Cauchy-Schwartz,
$$
\|{\tilde b}(\cdot)\|_2^2
=
\E\bigl(\E(\psi(X,X_1)g(X_1)\bigl|\, X)\bigr)^2
\le
\E\psi^2(X,X_1)\, \sigma^2
\le
c_*\sigma^2.
$$
Now the identity $b^*=\sigma^{-2}p_g({\tilde b})$ implies
\begin{equation}\label{A3.13}
\|b^*\|_2^2
\le
\sigma^{-4}\|{\tilde b}\|_2^2
\le
\sigma^{-2}c_*.
\end{equation}
Invoking the bound $\E g^2(Y_j)\le c_*\sigma^2$, see (\ref{3.19}), we
obtain (\ref{A3.12}).

Finally, write
$$
V_3
=
q_N^{-1}\E {\tilde V}{\mathbb I}_{A_j},
\qquad
{\tilde V}=g(X_j)\psi^{**}(X_1,X_j)b^{*}(X_1).
$$
The identity (\ref{A3.2}) implies $\E {\tilde V}=0$.
Therefore
$V_3=q_N^{-1}\E{\tilde V}({\mathbb I}_{A_j}-1)$.
Invoking (\ref{3.15}) and using  $q_N^{-1}\le c_*$, see
(\ref{3.17}), we obtain
\begin{equation}\label{A3.14}
|V_3|
\le
c_*N^{-\alpha(r-4)}\E|{\tilde V}|\|Z_j'\|_r^{r-4}
\le c_*N^{-\alpha(r-4)}.
\end{equation}
In the last step we used the bound
$\E|{\tilde V}|\|Z_j'\|_r^{r-4}\le c_*$.
In order to prove this bound we invoke the inequalities
$$
|abc|
\le
(ab)^2+c^2
\le
a^4+b^4+c^2
$$
to show that
$$
|{\tilde V}|
\le
|g(X_j)|^4+|\psi^{**}(X_1,X_j)|^4+|b^*(X_1)|^2.
$$
Furthermore, by H\"older's inequality and (\ref{A3.5}),
$$
\E|g(X_j)|^4\|Z_j'\|_r^{r-4}\le c_*,
\qquad
\E|\psi^{**}(X_1,X_j)|^4\|Z_j'\|_r^{r-4}\le c_*.
$$
By the independence and (\ref{A3.13}),
$$
\E|b^*(X_1)|^2\|Z_j'\|_r^{r-4}
=
\|b^*\|_2^2\E\|Z_j'\|_r^{r-2}\le c_*.
$$
Thus we arrive at (\ref{A3.14}).
Combining (\ref{A3.10}), (\ref{A3.12}) and (\ref{A3.14}) we obtain (\ref{A3.8}).

Let us prove (\ref{A3.9}). Using (\ref{A3.3}) write
$H_2=Q_1+Q_2+2Q_3$, where
\begin{align*}
&
Q_1=\E\psi^{**}(X_1,Y_j)\psi^{**}(X_1,Y_i),
\qquad
Q_2=\|b^*\|_2^2\E g(Y_j)g(Y_i),
\\
&
Q_3=\E \psi^{**}(X_1,Y_j)g(Y_i)b^*(X_1).
\end{align*}

It follows from the identity (\ref{A3.1}) that
$$
Q_1
=
q_N^{-2}
\E
\psi^{**}(X_1,X_j)\psi^{**}(X_1,X_i)
({\mathbb I}_{A_j}-1)({\mathbb I}_{A_i}-1).
$$
The simple  inequality
$|\psi^{**}(X_1,X_j)\psi^{**}(X_1,X_i)|\le|\psi^{**}(X_1,X_j)|^2+|\psi^{**}(X_1,X_i)|^2$
 yields, by symmetry,
\begin{equation}\label{A3.15}
|Q_1|
\le
2q_N^{-2}\E|\psi^{**}(X_1,X_j)|^2(1-{\mathbb I}_{A_j})
\le
c_*N^{-\alpha(r-2)}.
\end{equation}
In the last step we applied (\ref{A3.11}) and $q_N^{-1}\le c_*$, see (\ref{3.16}).

Furthermore, using the identity $\E g(X_i)=0$ we obtain from
(\ref{3.15})
\begin{eqnarray}\label{A3.16}
|\E g(Y_i)|
&&
=
q_N^{-1}|\E  g(X_i)({\mathbb I}_{A_i}-1)|
\\
\nonumber
&&
\le
q_{N}^{-1}N^{-\alpha(r-1)}
\E|g(X_i)|\|Z_i\|_r^{r-1}\le c_*N^{-\alpha(r-1)}.
\end{eqnarray}
In the last step we applied H\"older's inequality to show
$\E|g(X_i)|\|Z_i\|_r^{r-1}\le c_*$.

The  bounds (\ref{A3.16}), (\ref{3.16})  and (\ref{A3.13}) together imply
\begin{equation}\label{A3.17}
|Q_k|\le c_*N^{-\alpha(r-1)},
\qquad
k=2,3.
\end{equation}
The bound (\ref{A3.9}) follows from (\ref{A3.15}) and (\ref{A3.17}).

Let us prove (\ref{A3.7}).
For this purpose we shall show  that
\begin{equation}\label{A3.18}
\E \big\|\sum_{j\in O_k}\frac{V_j}{\sqrt M}\big\|_r^r\le c_*,
\qquad
{\text{where}}
\qquad
V_j=\psi(\cdot,Y_j)-\psi(\cdot,{\overline Y}_j),
\end{equation}
and where ${\overline Y}_j$
 denote independent copies of $Y_j$, $j\in O_k$.
Using 
$$
\E\|\psi(\cdot,X_j)\|_r^r=\E|\psi(X_1,X_j)|^r\le c_*
$$ 
we obtain, by symmetry and
(\ref{3.19}),
$$
\E\|V_j\|_r^r
\le
2^r\E\|\psi(\cdot,Y_j)\|_r^r
\le
c_*\E\|\psi(\cdot, X_j)\|_r^r
\le
c_*.
$$
Now (\ref{A3.18}) follows from the well known inequality
\begin{equation}\label{A3.19}
\|\xi_1+\dots+\xi_k\|_r^r\le
c(r)\sum_{i=1}^k\E\|\xi_i\|_r^r+c(r)\bigl(\sum_{i=1}^k\E\|\xi_i\|_r^2\bigr)^{r/2},
\quad k=1,2,\dots
\end{equation}
which is valid for independent centered random elements $\xi_i$  with
values in $L^r$. One can derive this inequality from  Hoffmann --
Jorgensen's inequality (see e.g., Proposition 6.8 in  Ledoux and Talagrand (1991)
\cite{Ledoux_Talagrand_1991})
using the type $2$ property of the Banach space $L^r$ and
 the
symmetrization lemma (see formula (9.8) and Lemma 6.3 ibidem).
The proof of the lemma is complete.
\end{proof}

\smallskip
Before formulating and proving Lemma \ref{LA3.3} we introduce some more
notation.
Let ${\B}(L^r)$ denote the class of Borel sets of $L^r$.
  Consider
the regular conditional probability $P_k:{\mathbb R}\times {\B}(L^r)\to [0,1]$,
defined, for $z_k\in {\mathbb R}$ and $B\in {\B}(L^r)$,
$$
P_k(z_k;B):=\P\bigl(U_k\in B\,\bigr|g_k=z_k\bigr)=\E(\II_{U_k\in B}|g_k=z_k).
$$
Recall, see (\ref{3.54}), that $\psi_k$ denotes a $L^r$ valued random variable with the distribution
$\P\{\psi_k\in B\}=P_k(z_k;B)$.
Note that the $L^r$ valued random variable $\psi_k^*=p_g(\psi_k)$ has distribution
\begin{eqnarray}\nonumber
\P\{\psi_k^*\in B\}
&&
=
\P\{p_g(\psi_k)\in B\}
=
\P\{\psi_k\in p_g^{-1}(B)\}
\\
\label{A3.20}
&&
=
\P\bigl(U_k\in p_g^{-1}(B)\,\bigr|g_k=z_k\bigr)
=
\P\bigl(U_k^*\in B\,\bigr|\,g_k=z_k\bigr).
\end{eqnarray}
Furthermore,
using (\ref{A3.4}) we  write (\ref{A3.20}) in
the  form
$$
\P\{\psi_k^*\in B\}
=
\P
\Bigl(
U_k^{**}
+
(z_k-\frac{1}{N}\frac{\xi_k}{n^{1/2}})\,b^*
\in B
\,
\Bigr|
\,
g_k=z_k
\Bigr).
$$

Let ${\overline \psi}_k$ respectively ${\overline \psi}^*_k$   denote an
independent copy of $\psi_k$ respectively $\psi^*_k$.
Denote 
$$\tau_N=M^{-(r-4)/2}+N^{-\alpha(r-2)}M.$$
\smallskip

\begin{lem}\label{LA3.3} Let $k=1,\dots, n-1$. Let $|z_k|\le w\,n^{-1/2}$.
There exist positive constants $c_*^{(i)}$, $i=0,1,2,3$,
which depend on $w, r,\nu_1,\nu_2, \delta, A_*,D_*,M_*$ only such that
for
\begin{equation}\label{A3.21}
\tau_N\le c_*^{(0)}\delta_3^2,
\end{equation}
we have
\begin{eqnarray}\label{A3.22}
&&
c_*^{(1)}\delta_3^2
\le
n\E\|\psi_k^*-{\overline \psi}_k^*\|_2^2
\le c_*^{(2)}\delta_3^2
\\
\label{A3.23}
&&
\E\|\psi_k-{\overline \psi}_k\|_r^r
\le
c_*^{(3)}n^{-r/2}.
\end{eqnarray}
\end{lem}
The condition (\ref{A3.21}) requires $N$ to be
large enough.  A simple calculation shows $\tau_N\le N^{-75\nu}$,
for $\nu$ satisfying (\ref{2.4}). Therefore, (\ref{3.59}) implies
$\tau_N\le N^{-65\nu}\delta_3^2$. In particular, under (\ref{3.59}) the inequality
(\ref{A3.21}) is satisfied provided that $N>c_*$,
where $c_*$ does not depend on $\delta_3^2$.
\smallskip

\begin{proof}
By ${\tilde c}_*, {\tilde c}_*'$ we
denote positive constants which depend only on
 $w, r,\nu_1,\nu_2,  \delta, A_*,D_*,M_*$. These constants can be different in different places
of the text.
Given $i,j\in O_k$, $i\not= j$, introduce random variables
\begin{align*}
&
g_*=\eta+\zeta,
\qquad
\eta=\frac{\xi_k}{R},
\quad
\zeta=\frac{1}{\sqrt M}\sum_{j\in O_k}g(Y_j),
\\
&
\zeta_i=\zeta-\frac{g(Y_i)}{\sqrt M},
\qquad
\zeta_{ij}=\zeta-\frac{g(Y_i)}{\sqrt M}-\frac{g(Y_j)}{\sqrt M}.
\end{align*}
Here $R=\sqrt{n\,M\,N}$ satisfies $N/2\le R\le N$, by the choice of $n$ and $M$.
Let $p$, $p_0$, $p_1$, and $p_2$ denote the densities of random variables
$\eta$, $\zeta+\eta$, $\zeta_i+\eta$, and $\zeta_{ij}+\eta$ respectively.

Note that $g_*=\sqrt{N/M}g_k$. Therefore, the condition
$g_k=z_k$ is equivalent to $g_*=z_*$, where $z_*=\sqrt{N/M}z_k$.
Furthermore, $|z_k|\le w\, n^{-1/2}\Leftrightarrow|z_*|\le w_*$, where
$w_*=w\sqrt{N/Mn}\le 2w$.

Given a random variable $Y$, we  denote
the conditional expectation
$\E(Y|g_*=z_*)=\E(Y|g_k=z_k)$ by $\E_*Y$.
 Similarly, for an event $A$,
we have $P(A|g_k=z_k)=P(A|g_*=z_*)$.

\bigskip
{\bf Proof of}  (\ref{A3.22}).
For the $L^r$ valued random variable ${\hat \psi}^*=\psi_k^*-z_kb^*$ we have
\begin{equation}\label{A3.24}
\P\{{\hat \psi}^*\in B\}=
P
\Bigl(
U_k^{**}
-
\frac{1}{N}\frac{\xi_k}{n^{1/2}}\,b^*
\in B
\,
\Bigr|
\, g_*=z_*
\Bigr).
\end{equation}
Note that for an independent copy ${\overline \psi}^*_k$ of $\psi^*_k$
the distributions of $\psi_k^*-{\overline \psi}^*_k$ and
${\hat\psi}^*-{\hat\psi}^*_c$ are the same. Here ${\hat\psi}^*_c$
denotes an independent copy of ${\hat \psi}^*$. Therefore,
\begin{equation}\label{A3.25}
\E\|\psi^*_k-{\overline \psi}^*_k\|_2^2
=
\E\|{\hat \psi}^*-{\hat\psi}^*_c\|_2^2
=
2\E\|{\hat\psi}^*\|_2^2-2\|\E {\hat\psi}^*\|_2^2.
\end{equation}
In order to prove (\ref{A3.22}) we show that
\begin{equation}\label{A3.26}
\|\E {\hat \psi}^*\|_2^2
\le
{\tilde c}_* N^{-1}
\end{equation}
and, for $\tau_N\le c_*^{(0)}\delta_3^2$ (i.e., for sufficiently large $N$),
\begin{equation}\label{A3.27}
{\tilde c}_*\delta_3^2
\le
n\E\|{\hat \psi}^*\|_2^2
\le
{\tilde c}_*'\delta_3^2.
\end{equation}
Since $N^{-1}n<\tau_N$, we can choose $c_*^{(0)}$ small enough such that
  the inequalities (\ref{A3.25}), (\ref{A3.26}) and
(\ref{A3.27}) together
imply (\ref{A3.22})

\bigskip

{\bf Proof of} (\ref{A3.26}).
Recall that an element $m=m(\cdot)\in L^2(P_X)$ is called mean of an $L^2(P_X)$ valued
random variable
${\hat \psi}^*={\hat \psi}^*(\cdot)$
if for every $f=f(\cdot)\in L^2(P_X)$
$$
\left< f,m \right>=\E\left< f, {\hat \psi}^*\right>.
$$
We shall show below that $\E\|{\hat \psi}^*\|_2^2<\infty$. Then,
by Fubini,
$$
\E \left<f,{\hat \psi}^*\right>
=
\int f(x)\E \,{\hat \psi}^*(x)P_X(dx).
$$
Therefore,
 $m(x)=\E {\hat \psi}^*(x)$, for $P_X$ almost all $x$.

For $f\in L^2(P_X)$ it follows from (\ref{A3.24}) that
\begin{eqnarray}\label{A3.28}
\E \left<f,{\hat \psi}^*\right>
&&
 =\E_*\left<f,
U_k^{**}
-
\frac{1}{N}\frac{\xi_k}{n^{1/2}}\,b^*\right>
\\
\nonumber
&&
=
\E_*\left<f,U_k^{**}\right>
-
\frac{\sqrt M}{\sqrt N} \left<f,b^*\right> \E_*\eta .
\end{eqnarray}
Fix $i\in O_k$. By symmetry,
\begin{equation}\label{A3.29}
\E_*\left<f,U_k^{**}\right>
=
\frac{M}{\sqrt N}\E_*\left<f,\psi^{**}(\cdot, Y_i)\right>
\end{equation}
An application of (\ref{A3.57}) yields
\begin{eqnarray}\nonumber
\E_*\left<f,\psi^{**}(\cdot,Y_i)\right>
&&
=
\frac{1}{p_0(z_*)}\E\left<f,\psi^{**}(\cdot,Y_i)\right>
p_1\bigl(
z_*-\frac{g(Y_i)}{\sqrt M}
\bigr)
\\
\label{A3.30}
&&
=
\left<f,a_{z_*}\right>,
\end{eqnarray}
where
$$
a_{z_*}(\cdot)=\frac{b_{z_*}(\cdot)}{p_0(z_*)},
\qquad
b_{z_*}(\cdot)
=
\E \psi^{**}(\cdot,Y_i) p_1\bigl(z_*-\frac{g(Y_i)}{\sqrt M}\bigr)
$$
are non-random elements of $L^r$.

It follows from (\ref{A3.28}), (\ref{A3.29}), (\ref{A3.30}) that
$$
m(\cdot)
=
\frac{M}{\sqrt N}a_{z_*}(\cdot)
-
\frac{\sqrt M}{\sqrt N}b^*(\cdot)\,\E_*\eta.
$$

In order to prove (\ref{A3.26}) we show that, for $|z_*|\le w_*$,
\begin{eqnarray}\label{A3.31}
&&
\|b_{z_*}\|_2
\le
c_*M^{-1},
\\
\label{A3.32}
&&
|\E_* \eta |
\le
{\tilde c}_*M^{-1/2}+{\tilde c}_* R^{-1/2},
\\
\label{A3.33}
&&
p_i(z_*)\ge {\tilde c}_*,
\qquad
i=0,1,2,
\end{eqnarray}
and apply (\ref{A3.13}).
Note that, by Lemma \ref{LA2.1}, there exist  positive
constants ${\tilde c}_*, {\tilde c}_*'$ such
that, for $M,N>{\tilde c}_*'$, the inequality (\ref{A3.33}) holds.

\smallskip
{\it Let us prove} (\ref{A3.31}).
In Lemma \ref{LA2.1} we show, for $i=1,2$, that  $p_i$ and its derivatives  are bounded
functions. That is,
\begin{equation}\label{A3.34}
|p_i|\le c_*,
\qquad
|p_i'|\le c_*,
\qquad
|p_i''|\le c_*,
\qquad
|p_i'''|\le c_*,
\quad i=1,2.
\end{equation}
Expanding in powers of $M^{-1/2}g(Y_i)$
we obtain
\begin{equation}\label{A3.35}
p_1\bigl(z_*-\frac{g(Y_i)}{\sqrt M}\bigr)
=
p_1(z_*)-\frac{g(Y_i)}{\sqrt M}p_1'(z_*)+\frac{g^2(Y_i)}{M}\frac{p_1''(\theta)}{2}.
\end{equation}
It follows from the identities (\ref{A3.1}) and (\ref{A3.2}) that
  for $P_X$ almost all $x$
\begin{align*}
\E \psi^{**}(x,Y_i)
&
=
q_N^{-1} \E \psi^{**}(x,X_i)\II_{A_i}
\\
&
=
q_N^{-1}\E \psi^{**}(x,X_i)(\II_{A_i}-1)
\\
&
=:q_N^{-1}a_0(x)
\\
\E \psi^{**}(x,Y_i)g(Y_i)
&
=
 q_N^{-1} \E \psi^{**}(x,X_i)g(X_i)\II_{A_i}
\\
&
=
q_N^{-1}\E \psi^{**}(x,X_i)g(X_i)(\II_{A_i}-1)
\\
&
=:q_N^{-1}a_1(x).
\end{align*}
Using (\ref{A3.34}) and the inequality $q_N^{-1}\le c_*$, see (\ref{3.16}), we obtain from (\ref{A3.35})
$$
\|b_{z_*}(\cdot)\|_2
\le
c_*\|a_0(\cdot)\|_2
+
\frac{c_*}{\sqrt M}\|a_1(\cdot)\|_2
+
\frac{c_*}{M}\|a_2(\cdot)\|_2,
$$
where we denote
$a_2(\cdot)=\E \psi^{**}(\cdot,Y_i)g^2(Y_i)$.
In order to prove (\ref{A3.31}) we show that
\begin{equation}\label{A3.36}
\|a_0(\cdot)\|_2\le \frac{c_*}{N^{\alpha(r-1)}},
\qquad
\|a_1(\cdot)\|_2\le \frac{c_*}{N^{\alpha(r-2)}},
\qquad
\|a_2(\cdot)\|_2\le c_*.
\end{equation}

Let us prove (\ref{A3.36}). Invoking (\ref{3.15}) we  obtain, by H\"older's inequality,
\begin{equation}\label{A3.37}
|a_0(x)|\le \E|\psi^{**}(x,X_i)|\frac{\|Z_i'\|_r^{r-1}}{N^{\alpha(r-1)}}
\le w^{1/r}(x)\frac{K^{(r-1)/r}}{N^{\alpha(r-1)}},
\end{equation}
where we denote $w(x)=\E |\psi^{**}(x,X_i)|^r$. Furthermore, by Lyapunov's
inequality,
\begin{equation}\label{A3.38}
\|w^{1/r}(\cdot)\|_2^2=\E w^{2/r}(X)\le \bigl(\E w(X)\bigr)^{2/r}=K_r^{2/r}.
\end{equation}
Clearly, the first bound of (\ref{A3.36}) follows from (\ref{A3.37}), (\ref{A3.38}) and
(\ref{A3.5}). A similar argument shows the second bound of (\ref{A3.36}). We have
\begin{equation}\label{A3.39}
|a_1(x)|
\le
\E|\psi^{**}(x,X_i)g(X_i)|\frac{\|Z_i'\|_r^{r-2}}{N^{\alpha(r-2)}}
\le
w^{1/r}(x) \frac{V^{(r-1)/r}}{N^{\alpha(r-2)}},
\end{equation}
where we denote $V=\E \bigl( \|Z_i'\|_r^{r-2} |g(X_i)|\bigr)^{r/(r-1)}$.
By H\"older's inequality,
\begin{equation}\label{A3.40}
V
\le
\bigl(\E |g(X_i)|^r\bigl)^{1/(r-1)}\bigl(\E\|Z_i'\|_r^r\bigr)^{(r-2)/(r-1)}
\le
c_*.
\end{equation}
Clearly, (\ref{A3.38}), (\ref{A3.39}) and (\ref{A3.40}) imply the second bound of (\ref{A3.36}). The last bound of
(\ref{A3.36}) follows from (\ref{3.19}), by
Cauchy-Shwartz. Indeed, we have
$$
|a_2(x)|
\le
c_*\E|\psi^{**}(x,X_i)|g^2(X_i)
\le c_*
\bigl(\E|\psi^{**}(x,X_i)|^2\E g^4(X_i)\bigr)^{1/2}.
$$
Therefore, $\|a_2(\cdot)\|_2^2\le c_*K_2\E g^4(X_i)\le c_*$, by (\ref{A3.5}).

\smallskip
{\it Let us prove} (\ref{A3.32}).
We have, by (\ref{A3.56}),
$$
\E_*\eta=p_0^{-1}(z_*)\E(z_*-\zeta)p(z_*-\zeta).
$$
 In order to prove (\ref{A3.32}) it suffices to show in view of (\ref{A3.33}) that
\begin{equation}\label{A3.41}
|\E(z_*-\zeta)p(z_*-\zeta)|
\le
c_*R^{-1/2}+ c_*M^{-1/2}.
\end{equation}
Let ${\tilde p}$ denotes the density function  of $\xi_k$.
Then  $p(u)=R\,{\tilde p}(R\, u)$.
We have
$$
\E (z_*-\zeta)p(z_*-\zeta)
=
6 c_{\xi}\E
\frac{\sin^6 (R(z_*-\zeta)/6)}{\bigl(R\,(z_*-\zeta)/6\bigr)^5}.
$$
Therefore, denoting $H(z_*)=1+|R\,(z_*-\zeta)|^5$, we obtain
\begin{equation}\label{A3.42}
\E |(z_*-\zeta)p(z_*-\zeta)|
\le c
\E H^{-1}(z_*).
\end{equation}
On the event $|\zeta-z_*|\ge R^{-1/2}$ we have
$
H^{-1}(z_*)
\le
R^{-5/2}$.
Furthermore, a bound for the probability of the complementary event
$$
\P\{|\zeta-z_*|\le R^{-1/2}\}\le c_*R^{-1/2}+c_*M^{-1/2},
$$
follows by the Berry-Esseen bound applied to
the sum $\zeta$.
Therefore, $\E H^{-1}(z_*)$ is bounded by the right hand side of
(\ref{A3.41}). Now (\ref{A3.41}) follows from (\ref{A3.42}).

\bigskip

{\bf Proof of } (\ref{A3.27}).  Write
\begin{align*}
&
U_k^{**}-\frac{1}{N}\frac{\xi_k}{\sqrt n}b^*=\frac{\sqrt M}{\sqrt N}(T_1-T_2),
\\
&
T_1:= \frac{1}{\sqrt M}\sum_{j\in O_k}\psi^{**}(\cdot,Y_j),
\qquad
T_2:=\eta b^*.
\end{align*}

It follows from (\ref{A3.24}), by the inequality
$\|u+v\|_2^2\ge\|u\|_2^2/2-\|v\|_2^2$, for $u,v\in L^2(P_X)$, that
\begin{align*}
\E \|{\hat \psi}^*\|_2^2
&
=
\frac{M}{N}
\E_*\|T_1-T_2\|_2^2
\ge
\frac{M}{2N}\E_* \|T_1\|_2^2
-
\frac{M}{N}\E_*  \|T_2\|_2^2.
\end{align*}
We shall show
that
\begin{eqnarray}\label{A3.43}
&&
\E_* \|T_2\|_2^2
\le
p_0^{-1}(z_*)\bigl(c_*R^{-1}M^{-1/2}+c_*R^{-3/2}\bigr),
\\
\label{A3.44}
&&
\E_* \|T_1\|_2^2
\ge
p_0^{-1}(z_*)\bigl(p_1(z_*)\delta_3^2-c_*\tau_N\bigr).
\\
\label{A3.45}
&&
\E_* \|T_1\|_2^2
\le
p_0^{-1}(z_*)\bigl(p_1(z_*)\delta_3^2+c_*\tau_N\bigr).
\end{eqnarray}
The inequalities (\ref{A3.43}) and (\ref{A3.44}) imply the lower bound in (\ref{A3.27}).
Indeed, by (\ref{A3.33}), we have, for small $c_*^{(0)}$,
$$
\frac{c_*}{M^{1/2}R}+\frac{c_*}{R^{3/2}}
\le
c_*\tau_N
\le
c_*c_*^{(0)}\delta_3^2
\le
p_1(z_*)\delta_3^2/4.
$$
Similarly, the inequalities (\ref{A3.43}) and (\ref{A3.45}) imply the upper bound in (\ref{A3.27}).

{\it Proof of} (\ref{A3.43}). We have, by  (\ref{A3.56}),
$$
\E_* \eta^2
=
p_0^{-1}(z_*)W,
\qquad
W:=\E(z_*-\zeta)^2p(z_*-\zeta).
$$
Proceeding as in proof of (\ref{A3.41}), we obtain
$$
W
=
\frac{36}{R}c_{\xi}\E\frac{\sin^6(R(z_*-\zeta)/6)}{(R(z_*-\zeta)/6)^{4}}
\le
\frac{c}{R}\E {\tilde H}^{-1}(z_*),
$$
where ${\tilde H}(z_*)=1+|R(z_*-\zeta)|^{4}$ satisfies
$$
\E {\tilde H}^{-1}(z_*)\le c_*R^{-1/2}+c_*M^{-1/2}.
$$
Therefore,
$
W
\le
c_*R^{-3/2}+c_*R^{-1}M^{-1/2}$.
This inequality in combination with (\ref{A3.13}) implies (\ref{A3.43}).

{\it Proof of} (\ref{A3.44}). Fix $i,j\in O_k$, $i\not=j$. By symmetry,
\begin{eqnarray}\label{A3.46}
&&
\E_* \|T_1\|_2^2
=
\E_*T_{11}+(M-1)\E_*T_{12},
\\
\nonumber
&&
T_{11}=\|\psi^{**}(\cdot,Y_i)\|_2^2,
\qquad
T_{12}=\left<\psi^{**}(\cdot,Y_i),\,\psi^{**}(\cdot,Y_j)\right>.
\end{eqnarray}
We have, by (\ref{A3.57}),
\begin{align*}
&
\E_*T_{11}=
p_0^{-1}(z_*)H_1,
\qquad
\E_*T_{12}=
p_0^{-1}(z_*)H_2,
\qquad
\\
&
H_1=\E T_{11}p_1\bigl(z_*-\frac{g(Y_i)}{\sqrt M}\bigr),
\qquad
H_2=\E T_{12}p_2\bigl(z_*-\frac{g(Y_i)+g(Y_j)}{\sqrt M}\bigr).
\end{align*}
The inequality (\ref{A3.44}) follows from (\ref{A3.46}) and the bounds
\begin{eqnarray}\label{A3.47}
&&
H_1
\ge
p_1(z_*)\delta_3^2-c_*M^{-1/2},
\\
\label{A3.48}
&&
|H_2|
\le
c_*N^{-\alpha(r-2)}
+
c_*M^{-(r-2)/2}.
\quad
\end{eqnarray}

Let us prove (\ref{A3.47}). It follows from (\ref{A3.34}), by the mean value theorem, that
\begin{equation}\label{A3.49}
H_1
=
p_1(z_*)\E T_{11}
+
Q,
\qquad
|Q|\le c_*\E T_{11}\frac{|g(Y_i)|}{\sqrt M},
\end{equation}
where $|Q|\le c_*M^{-1/2}$.
Indeed, by (\ref{3.19}) and Cauchy-Schwartz,
$$
\E T_{11}|g(Y_i)|
\le
\E \|\psi^{**}(\cdot,X_i)\|_2^2|g(X_i)|
\le
K_4^{1/2}\sigma
\le
c_*.
$$
In the last step we applied (\ref{A3.5}).
Furthermore, the identity
\begin{align*}
&
\E T_{11}
=
q_N^{-1}\E|\psi^{**}(X,X_i)|^2{\mathbb I}_{A_i}
=\E|\psi^{**}(X,X_i)|^2-b_1-b_2,
\\
&
b_1=(1-q_N^{-1})\E|\psi^{**}(X,X_i)|^2,
\qquad
b_2=q_N^{-1}\E|\psi^{**}(X,X_i)|^2(1-{\mathbb I}_{A_i})
\end{align*}
combined with (\ref{3.15}), (\ref{3.16}) and (\ref{3.17}) yields $\E T_{11}\ge
\delta_3^2-c_*M^{-1/2}$. This bound together with (\ref{A3.49}) shows
(\ref{A3.47}).

Let us prove (\ref{A3.48}). Write $y_i=g(Y_i)$ and expand
$$
p_2\bigl(z_*-\frac{y_i+y_j}{\sqrt M}\bigr)
=
p_2(z_*)-p_2'(z_*)\frac{y_i+y_j}{\sqrt M}+\frac{p_2''(z_*)}{2}\frac{(y_i+y_j)^2}{M}
+
{\tilde Q}.
$$
From (\ref{A3.34}) it follows, for $2<r-2\le 3$, the bound
$$
|{\tilde Q}|\le c_*|y_i+y_j|^{r-2}/M^{(r-2)/2}.
$$
Furthermore, denote
\begin{align*}
&
h_1=\E T_{12},
\qquad
\qquad
\quad
h_2=\E T_{12}g(Y_i),
\\
&
h_3=\E T_{12}g^2(Y_i),
\qquad
h_4=\E T_{12}g(Y_i)g(Y_j).
\end{align*}
We obtain, by symmetry,
\begin{eqnarray}\label{A3.50}
&&
H_2=p_2(z_*)h_1-2\frac{p_2'(z_*)}{\sqrt M}h_2 +\frac{p_2''(z_*)}{M}(h_3+h_4)+
\E T_{12}{\tilde Q},
\\
\nonumber
&&
\E|T_{12}{\tilde Q}|\le c_*M^{-(r-2)/2}\E |g(Y_i)|^{r-2}|T_{12}|.
\end{eqnarray}
Denote
$$
{\tilde T}_{12}=q_N^{-2}\psi^{**}(X,X_i)\psi^{**}(X,X_j).
$$
It follows from (\ref{3.17}), by H\"olders inequality and (\ref{A3.5}), that
$$
\E |g(Y_i)|^{r-2}|T_{12}|
\le
\E|g(X_i)|^{r-2}|{\tilde T}_{12}|
\le
c_*.
$$
Therefore,
\begin{equation}\label{A3.51}
\E|T_{12}{\tilde Q}|\le c_*M^{-(r-2)/2}.
\end{equation}
Furthermore,
(\ref{A3.1}) and (\ref{A3.2}) imply
\begin{align*}
&
h_1=\E{\tilde T}_{12}(\II_{A_i}-1)(\II_{A_j}-1),
\qquad
h_2=\E{\tilde T}_{12}g(X_i)(\II_{A_i}-1)(\II_{A_j}-1),
\\
&
h_3=\E{\tilde T}_{12}g^2(X_i)\II_{A_i}(\II_{A_j}-1),
\qquad
h_4=\E{\tilde T}_{12}g(X_i)g(X_j)(\II_{A_i}-1)(\II_{A_j}-1).
\end{align*}
Invoking the inequalities $q_N^{-2}\le c_*$, see (\ref{3.16}),
 and $1-\II_{A_i}\le V_i^s$, $s>0$, where $V_i:= \|Z_i'\|_r/N^{\alpha}$, see (\ref{3.15}),
we obtain, by H\"older's inequality,
\begin{eqnarray}\label{A3.52}
&&
|h_1|
\le
\E|{\tilde T}_{12}|(V_iV_j)^{(r-2)/2}
\le
c_*N^{-\alpha(r-2)},
\\
\nonumber
&&
|h_2|
\le
\E|{\tilde T}_{12}g(X_i)|V_i^{(r-4)/2}V_j^{(r-2)/2}
\le
c_*N^{-\alpha(r-3)},
\\
\nonumber
&&
|h_3|
\le
\E|{\tilde T}_{12}|g^2(X_i)V_j^{r-4}
\le
c_*N^{-\alpha(r-4)},
\\
\nonumber
&&
|h_4|
\le
\E|{\tilde T}_{12}g(X_i)g(X_j)|(V_iV_j)^{(r-4)/2}
\le
c_*N^{-\alpha(r-4)}.
\end{eqnarray}
Combining (\ref{A3.50}), (\ref{A3.52}), (\ref{A3.51})  and using the simple inequalities
$$
\frac{1}{N^{\alpha(r-3)}M^{1/2}}\le \frac{1}{{\tilde N}},
\qquad
\frac{1}{N^{\alpha(r-4)}M}\le \frac{1}{{\tilde N}},
\qquad
{\tilde N}=\min\{N^{\alpha(r-2)}, M^{(r-2)/2}\}
$$
and the inequalities (\ref{A3.34}), we obtain (\ref{A3.48}).

{\it Proof  of }(\ref{A3.45}). The inequality follows from (\ref{A3.46}), (\ref{A3.48})
and the inequality
\begin{align*}
H_1
&
\le
p_1(z_*)\E T_{11}+c_*M^{-1/2}
\le
p_1(z_*)\delta_3^2+c_*N^{-\alpha r}+c_*M^{-1/2}
\\
&
\le
p_1(z_*)\delta_3^2+c_*M^{-1/2},
\end{align*}
which is obtained in the same way as (\ref{A3.47}) above.

\bigskip
{\bf Proof of} (\ref{A3.23}). In order to prove (\ref{A3.23}) we shall show that
\begin{equation}\label{A3.53}
\E\|\psi_k\|_r^r\le {\tilde c}_*n^{-r/2}.
\end{equation}
Split $O_k=B\cup D$, where $B\cap D=\emptyset$ and $|B|=[M/2]$ and write
\begin{align*}
&
U_k=\frac{\sqrt M}{\sqrt N}(U_B+U_D),
\qquad
 U_B=\frac{1}{\sqrt M}\sum_{j\in B}\psi(\cdot,Y_j),
\\
&
\zeta=\zeta_B+\zeta_D,
\qquad
\zeta_B=\frac{1}{\sqrt M}\sum_{j\in B}g(Y_j).
\end{align*}
In particular, we have $g_*=\eta+\zeta_B+\zeta_D$.

The inequality
$$
 \E\|\psi_k\|_r^r
 =
 \E_*\|U_k\|_r^r
 \le
 2^r\Bigl(\frac{M}{N}\Bigr)^{r/2}\bigl(\E_*\|U_B\|_r^r+\E_*\|U_D\|_r^r\bigr)
 $$
combined with the bounds
\begin{equation}\label{A3.54}
 \E_*\|U_B\|_r^r\le c_*,
 \qquad
 \E_*\|U_D\|_r^r\le c_*
\end{equation}
 imply (\ref{A3.53}).
%
Let us prove the first bound of (\ref{A3.54}). By (\ref{A3.57}), we have
$$
\E_*\|U_B\|_r^r=p_0^{-1}(z_*)\E \|U_B\|_r^rp_3(z_*-\zeta_B),
$$
where $p_3$ denotes the density of $\eta+\zeta_D$. Furthermore, invoking the
bound
\linebreak
$\sup_{x\in {\mathbb R}}|p_3(x)|\le c_*$,
(which is obtained using the same argument as in the proof of Lemma \ref{LA2.1}) and the inequality
(\ref{A3.33}), we obtain
$\E_*\|U_B\|_r^r\le {\tilde c}_*\E\|U_B\|_r^r$.
Finally, invoking the bound
\begin{equation}\label{A3.55}
\E\|U_B\|_r^r
\le
c_*\E\|\frac{1}{\sqrt M}\sum_{j\in B}\psi(\cdot, X_j)\|_r^r
\le c_*,
\end{equation}
see (\ref{3.19}) and (\ref{A3.19}),
we obtain the first bound of (\ref{A3.54}). The second bound is obtained in the same way.
This completes the proof of the lemma.
%
\end{proof}


\bigskip

We collect some facts about conditional
moments in a separate lemma.

\begin{lem}\label{LA3.4} Let $\eta$ and $\zeta$ be independent
 random variables. Assume that $\eta$ is real valued and  has a density, say
$x\to p(x)$.

(i) Assume that $\zeta$ is real valued.  Then the function
$$
x\to \E p(x-\zeta),
\qquad
x\in {\mathbb R},
$$
is a density of the distribution $P_{\eta+\zeta}$ of $\eta+\zeta$.
Let $w:{\mathbb R}\to {\mathbb R}$ be a  measurable function such that $\E|w(\eta)|<\infty$.
For $P_{\eta+\zeta}$
 almost all $x\in {\mathbb R}$, we have
\begin{equation}\label{A3.56}
\E\bigl(w(\eta)\,\bigr|\, \eta+\zeta=x\bigr)
=
\frac{\E w(x-\zeta)p(x-\zeta)}{\E p(x-\zeta)}.
\end{equation}

(ii) Assume that $\zeta$ takes values in a measurable space, say $\Y$.
Assume that $u,v:\Y\to {\mathbb R}$ are measurable functions
and denote $P_{\eta+u(\zeta)}$ the distribution of $\eta+u(\zeta)$.
If $\E|v(\zeta)|<\infty$, then for $P_{\eta+u(\zeta)}$ almost all $x\in {\mathbb R}$,
\begin{equation}\label{A3.57}
\E\bigl(v(\zeta)\,\bigr|\, \eta+u(\zeta)=x\bigr)
=
\frac{\E v(\zeta)p\bigl(x-u(\zeta)\bigr)}{\E p\bigl(x-u(\zeta)\bigr)}.
\end{equation}
\end{lem}


\section{Appendix 4}

In the next lemma we consider independent and identically distributed
random vectors $(\xi,\,\eta)$ and $(\xi',\, \eta')$ with values in
${\mathbb R}^2$ and the symmetrization $(\xi_s,\eta_s)$ where $\xi_s=\xi-\xi'$
and  $\eta_s=\eta-\eta'$.
Note that in the main text we  apply this
lemma to $\xi=g(X_1)$ and $\eta=N^{-1/2}\sum_{j=m+1}^N\psi(X_1,Y_j)$.

\begin{lem}\label{LA4.1}
Let $0<\nu<1/2$ and $r>2$. Assume that $\E|\xi|^r+\E|\eta|^r<\infty$.
The following statements hold.

{\bf a)} For $c_r=(7/12)2^{-r}$ the conditions
$$
|t|^{r-2}\E|\xi_s|^r\le c_r\E \xi_s^2,
\qquad
\E \xi_s\eta_s=0,
\qquad
\E|\eta_s|^r\le c_r\E\eta_s^2
$$
imply
$1-|\E \exp\{i(t\xi+\eta)\}|^2\ge 6^{-1}(t^2\E\xi_s^2+\E \eta_s^2)$.


{\bf b)}
Assume that for some  ${\tilde c}_1, {\tilde c}_2 >0$
we have
\begin{equation}\label{A4.1}
\E \xi_s^2/12-N^{-1}\E \eta_s^2>{\tilde c}_1^2
\quad
{\text{and}}
\quad
c_r\E\xi_s^2/\E|\xi_s|^r\ge {\tilde c}_2^{r-2}.
\end{equation}
Let $\e>0$ be such that
\begin{equation}\label{A4.2}
\e<1/6{\tilde c}_3,
\qquad
\e^{(r-2)/2}<\sigma_z^2/4A,
\qquad
\e^{r-2}<\sigma_z^2/4B,
\end{equation}
where ${\tilde c_3}=2+(5/{\tilde c}_1)^2\sigma_z^2$ and where the numbers $A,B$ are
defined in (\ref{A4.13}). Here
$\sigma_z^2=\E(\xi_s+N^{-1/2}\eta_s)^2$.
Assume that for some $0<\delta<{\tilde c}_2$
and
$\delta'>10\e^2$,
\begin{equation}\label{A4.3}
\sup_{\delta<|t|<N^{-\nu+1/2}}|\E e^{it\xi_s}|\le 1-\delta'
\qquad
{\text and}
\qquad
\E|\eta_s|\le \delta' N^{\nu}/2.
\end{equation}
Then for every $T^*$, satisfying
 $N^{1/2-\nu}\le |T^*|\le N^{\nu+1/2}$, the set
$$
I^*
=
\bigl\{
T^*\le t\le T^*+N^{1/2-\nu}: |\E e^{it(\xi+N^{-1/2}\eta)}|^2\ge 1-\e^2
\bigr\}
$$
is an interval of size  at most $5{\tilde c}_1^{-1} \e$.
\end{lem}

\begin{proof}


Proof  {\bf a)}.
Invoking the inequality $1-\cos x\ge x^2/2-x^2/24-|x|^r$ and using
the simple inequality $|a+b|^r\le 2^{r-1}(|a|^r+|b|^r)$ we obtain
\begin{align*}
1-|\E \exp\{i(t\xi+\eta)\}|^2
&
=
1-\E\cos(t\xi_s+\eta_s)
\\
&
\ge \frac{11}{24}\E(t\xi_s+\eta_s)^2
-
2^{r-1}(\E|t\xi_s|^r+\E|\eta_s|^r)
\\
&
\ge 6^{-1}(t^2\E\xi_s^2+\E \eta_s^2).
\end{align*}
In the last step we use the conditions {\bf a)}.

Proof  {\bf b)}. Introduce the function
$t\to \tau_t^*=1-|\E e^{it(\xi+N^{-1/2}\eta)}|^2$.
Assume that the set $I^*$ is non-empty and choose
$s,t\in I^*$, i.e., we have $\tau_t^*,\tau_s^*\le \e^2$.
Firstly we show that $|s-t|\le 5{\tilde c}_1^{-1} \e$, thus proving
the bound for the size of the set $I^*$.

The inequality $1-\cos(x+y)\ge (1-\cos x)/2-(1-\cos y)$ implies
\begin{equation}\label{A4.4}
1-|\E e^{i(X+Y)}|^2\ge 2^{-1}(1-|\E e^{iX}|^2)-(1-|\E e^{iY}|^2),
\end{equation}
for arbitrary random variables $X,Y$. Choosing
 ${\tilde Y}=t(\xi+N^{-1/2}\eta)$ and
${\tilde X}=(s-t)(\xi+N^{-1/2}\eta)$ shows
\begin{equation}\label{A4.5}
\tau_s^*\ge (1-|\E e^{i{\tilde X}}|^2)/2-\tau_t^*.
\end{equation}
Now we show that the inequality $|t-s|>5{\tilde c}_1^{-1}\e$
implies $1-|\E e^{i{\tilde X}}|^2>5\e^2$, thus, contradicting to our choice
$\tau_s^*,\tau_t^*<\e^2$ and (\ref{A4.5}).
In what follows
the cases of ``large'' and ``small'' values of $|t-s|$ are treated separately.

For $5{\tilde c}_1^{-1}\e<|t-s|\le \delta$ we shall apply (\ref{A4.4})
 to ${\tilde X}=X+Y$, where
$X=(s-t)\xi$ and $Y=(s-t)N^{-1/2}\eta$.
Note that the statement {\bf a)} implies
\begin{equation}\label{A4.6}
1-|\E e^{iX}|^2
\ge
(t-s)^2\E \xi_s^2/6.
\end{equation}
Indeed, in view of the second inequality of (\ref{A4.1}),
the conditions of {\bf a)} are satisfied for
$|t-s|\le\delta\le {\tilde c}_2$.
Furthermore, we have
\begin{equation}\label{A4.7}
0
\le
1-|\E e^{iY}|^2
=
1-\cos\bigl(N^{-1/2}(s-t)\eta_s\bigr)
\le
(s-t)^2N^{-1}\E \eta_s^2.
\end{equation}

Invoking the bounds (\ref{A4.6}) and (\ref{A4.7}) in (\ref{A4.4}) we obtain
$$
1-|\E e^{i{\tilde X}}|^2
\ge
(s-t)^2\E \xi_s^2/12
-
(s-t)^2N^{-1}\E \eta_s^2
\ge
{\tilde c}_1^2(s-t)^2\ge 25\e^2.
$$
In the last step we used (\ref{A4.1}).

For $\delta<|t-s|\le N^{-\nu+1/2}$ we
 expand in powers of $a=i(s-t)N^{-1/2}\eta_s$
to get
\begin{align*}
1-|\E e^{i{\tilde X}}|^2
&
=
1-\E\exp\{i(s-t)\xi_s+a\}
\ge
1-\E\exp\{i(s-t)\xi_s\}
-\E|a|
\\
&
\ge
\delta'-\E|(t-s)N^{-1/2}\eta_s|
\ge
\delta'-N^{-\nu}\E|\eta_s|
\\
&
\ge
 \delta'/2
\ge 5\e^2.
\end{align*}
In the last step we applied (\ref{A4.3}).

Let us prove that $I^*$ is indeed an interval.
Assume the contrary, i.e. there exist $s<u<t$ such that $s,t\in
I^*$ and $u\notin I^*$. In particular, $\tau_t^*\le \e^2<\tau_u^*$.
Clearly, we can choose $u$ to be a local
maximum (stationary) point of the function $t\to \tau_t^*$.
Denote
$$
z=\xi_s+N^{-1/2}\eta_s,
\qquad
\sigma_z^2=\E z^2.
$$
An
application of (\ref{A4.4}) to $Y'=(t-u)(\xi+N^{-1/2}\eta)$ and $X'=u(\xi+N^{-1/2}\eta)$ gives
$$
\tau_t^*\ge \tau_u^*/2-\bigl(1-\E e^{i(t-u)z}\bigr) =\tau_u^*/2-\bigl(1-\E \cos(t-u)z\bigr).
$$
Invoking the inequalities $\tau_t^*\le \e^2$ and $1-\cos(t-u)z\le (t-u)^2z^2/2$ we obtain
\begin{equation}\label{A4.8}
\tau_u^*\le 2\e^2+(t-u)^2\sigma_z^2\le \e^2{\tilde c}_3,
\qquad
{\tilde c}_3=2+(5/{\tilde c}_1)^2\sigma_z^2.
\end{equation}
Here we used  the bound $|t-u|\le |t-s|\le 5\e/{\tilde c}_1$
proved above.

Denoting
$y=(t-u)z$ we
 have $\tau_t^*=1-\E e^{iuz}e^{iy}$. Invoking the expansion
$e^{iy}=1+iy+(iy)^2/2+R'$, where $|R'|\le y^2/6+|y|^r$, we obtain
\begin{equation}\label{A4.9}
\tau_t^*=\tau_u^*-i\E ye^{iuz}+2^{-1}\E y^2e^{iuz}+R,
\qquad
|R|\le \E y^2/6+\E |y|^r=:R_0.
\end{equation}
For a stationary point $u$  we have
$0=\frac{\partial}{\partial t}\tau_t^*\bigl|_{t=u}=-i\E z e^{iuz}$.
Therefore, $\E ye^{iuz}=0$ and (\ref{A4.9}) implies
$$
\tau_t^*\ge \tau_u^*+2^{-1}(t-u)^2\E z^2e^{iuz}-R_0.
$$
Write the right hand side in the form $\tau_u^*+2^{-1}(t-u)^2R_1$, where
$$
R_1=\E z^2e^{iuz}-3^{-1}\sigma_z^2-2\E|z|^r|t-u|^{r-2}.
$$
Note that the inequality $R_1>0$ contradicts
to our assumption $\tau_t^*< \tau_u^*$. We complete the proof by showing that $R_1>0$.

Since the random variable $z$ is symmetric we have $\E z^2\sin uz=0$. Therefore,
\begin{equation}\label{A4.10}
\E z^2e^{iuz}=\E z^2\cos uz=\sigma_z^2-\E z^2(1-\cos uz).
\end{equation}
Given $\lambda>0$ split
\begin{eqnarray}
\nonumber
\E z^2(1-\cos uz)
&
=
&
\E z^2(1-\cos uz)\Bigl({\mathbb I}_{\{z^2<\lambda^2\}}+{\mathbb I}_{\{z^2\ge \lambda^2\}}\Bigr)
\\
\label{A4.11}
&
\le
&
\lambda^2\E(1-\cos uz)+2\E|z|^r\lambda^{2-r}.
\end{eqnarray}
In the last step we used Chebyshev's inequality.
Furthermore, invoking the inequality $\E (1-\cos uz)=\tau_u^*\le {\tilde c}_3\e^2$, see
(\ref{A4.8}), we obtain from (\ref{A4.10}) and (\ref{A4.11}) for $\lambda^2=\e^{-1}\sigma_z^2$
\begin{equation}\label{A4.12}
\E z^2e^{iuz}\ge \sigma_z^2-\e {\tilde c}_3\sigma_z^2 -\e^{(r-2)/2}2\E |z|^r \sigma_z^{2-r}.
\end{equation}
Finally, invoking the inequality $|t-u|\le|t-s|\le 5{\tilde c}_1^{-1}\e$ we obtain from (\ref{A4.12})
$$
R_1\ge \sigma_z^2(1-3^{-1}-\e {\tilde c}_3)-\e^{(r-2)/2}A-\e^{r-2}B,
$$
where for random variable $z=\xi_s+N^{-1/2}\eta_s$ we write
\begin{equation}\label{A4.13}
A=2\E|z|^r\sigma_z^{2-r}
\qquad
{\text{and}}
\qquad
B=2\E|z|^r(5/{\tilde c}_1)^{r-2}.
\end{equation}
Thus, for $\e$ satisfying (\ref{A4.2}) we have $R_1>0$.
\end{proof}

\section{Appendix 5}

Let $Z_1,\dots, Z_N$ be independent copies of the $L^r$ valued random
element $Z=\{x\to \psi(x,Y)\}$. Recall that almost surely
$\|Z\|\le N^{\alpha}$. Here $\|\cdot\|$ denotes the norm of the
Banach space $L^r$, where $r>4$ and $1/2>\alpha>0$. Write
$M_p=\E|\psi(X_1,X_2)|^p$.

\begin{lem}\label{LA5.1}
(i) Assume that $\|\E Z\|^2\le \E\|Z\|^2/N$. Then there exists a constant $c(r)>0$ such
that for $k\le N$ and $x>c(r)$
we have
\begin{equation}\label{A5.1}
\P\{\Vert Z_1+\dots+Z_k\Vert>k^{1/2} u\,x\}
\le
\exp\{-2^{-5}x^2(1+xN^{\alpha}/k^{1/2}u)^{-1}\}.
\end{equation}
Here $u^2=\E\|Z\|^2$.

(ii) The following inequalities hold
\begin{eqnarray}\label{A5.2}
&&
\|\E Z\|
\le
M_r/q_N N^{(r-1)\alpha}
\\
\label{A5.3}
&&
q_N^{-1}(M_2-M_rN^{-(r-2)\alpha})
\le
\E \|Z\|^2
\le
 q_N^{-1}(M_r^{2/r}+M_rN^{-(r-2)\alpha}).
\end{eqnarray}
\end{lem}

{\it Remark.} Assume that
$$
M_2\ge 2M_rN^{-(r-2)\alpha},
\qquad
M_r^2\le (q_N/2)M_2N^\varkappa,
\qquad
{\text{ where}}
\qquad
\varkappa=2(r-1)\alpha-1.
$$
Then (\ref{A5.2}) and (\ref{A5.3}) imply
the inequality
 $\|\E Z\|^2\le \E\|Z\|^2/N$. Note that $r\alpha>2$ implies $\varkappa>2$.
Furthermore, by (\ref{3.16}), the probability $q_N>1-M_rN^{-r\alpha}$.

{\it Proof.} We derive (i) from Yurinskii's (1976) inequality.
 Denote $\zeta_k=Z_1+\dots+Z_k$. Using the type$-2$ inequality for an $L^r$
valued  random variable $\zeta_k-\E \zeta_k$,
$$
\E\|\zeta_k-\E \zeta_k\|^2\le k {\tilde c}(r)\E\|Z_1-\E Z_1\|^2,
$$
and the inequality $\|Z_1-\E Z_1\|^2\le 2\|Z_1\|^2+2\|\E Z_1\|^2$, we obtain
$$
\E\|\zeta_k-\E \zeta_k\|
\le
\bigl(\E\|\zeta_k-\E \zeta_k\|^2\bigr)^{1/2}
\le
k^{1/2}c'(r)(u+\|\E Z_1\|).
$$
 We have
\begin{align*}
\E\|\zeta_k\|
&
\le
\E\|\zeta_k-\E \zeta_k\|+k\|\E Z_1\|
\\
&
\le
c'(r)k^{1/2}u+k(1+c'(r)k^{-1/2})\|\E Z_1\|=:\beta_k.
\end{align*}

It follows from the inequality $\|Z_1\|\le N^{\alpha}$ that
$$
\E \Vert Z_1\Vert^L\le 2^{-1}L!u^2N^{\alpha(L-2)},
\qquad
L=2,3,\dots .
$$
Write $B_k^2=ku^2$.
Theorem 2.1 of Yurinskii (1976) shows
\begin{equation}\label{A5.4}
\P\{\|\zeta_k\|\ge x B_k \}
\le
\exp\{-B\},
\quad
B=\frac{{\overline x}^2}{8}(1+({\overline x}N^\alpha/2B_k))^{-1}\},
\end{equation}
provided that ${\overline x}=x-\beta_k/B_k>0$.

Since $\beta_k/B_k\le 1+c'(r)(1+k^{-1/2})$ we have, for $x> c(r):=4c'(r)+2$,
$$
x>2\beta_k/B_k
\qquad
{\text{and}}
\qquad
x>{\overline x}>x/2.
$$
The latter inequality implies
$$
B\ge B':=
\frac{(x/2)^2}{8}(1+(xN^\alpha/B_k))^{-1}\}.
$$
Finally, replacing $B$ by $B'$ in  (\ref{A5.4}) we obtain (\ref{A5.1}).

Let us prove (ii).
The mean value $\E Z=\{x\to \E\psi(x,Y)\}$ is an element of $L^r$. For $P_X$ almost
all $x\in \cal X$
we have $\E \psi(x,X)=0$. Therefore,
$$
\E Z=q_N^{-1}\E\psi(x,X){\mathbb I}_A
=
q_N^{-1}\E\psi(x,X)({\mathbb I}_A-1).
$$
Invoking (\ref{3.15}) and using Chebyshev and H\"older inequalities, we
obtain, for $P_X$ almost all $x$,
$$
|\E Z|
\le
\frac{1}{q_N N^{\alpha(r-1)}}
\E\|Z'\|_r^{r-1}|\psi(x,X)|
\le
\frac{1}{q_N N^{\alpha (r-1)}}
(\E\|Z'\|_r^{r})^{(r-1)/r} a(x),
$$
where $a(x)=(\E|\psi(x,X)|^{r})^{1/r}$. Note that
$\E\|Z'\|_r^{r}=M_r$ and $\|a\|^r=M_r$. Finally,
$$
\|\E Z\|
\le
\|a\|M_r^{(r-1)/r}/q_N N^{\alpha(r-1)}=M_r/q_N N^{\alpha(r-1)}.
$$

Let us prove (\ref{A5.3}). Denote $b_p(x)=(\E_{X_1}|\psi(X_1,x)|^p)^{1/p}$. Here $\E_{X_1}$
denotes the conditional expectation given all the random variables, but $X_1$.
We have
\begin{equation}\label{A5.5}
\E\|Z\|^2
=
q_N^{-1}\E{\mathbb I}_A b_r^2(X)
=
q_N^{-1}\E b_r^2(X)
+
q_N^{-1}R,
\qquad
R=\E({\mathbb I}_A-1)b_r^2(X).
\end{equation}
By H\"older's inequality  $b_r(x)\ge b_2(x)$, for $P_X$ almost all
$x$. Therefore,
\begin{equation}\label{A5.6}
M_2=\E b_2^2(X)\le \E b_r^2(X)\le M_r^{2/r}
\end{equation}
Combining (\ref{A5.6}) and (\ref{A5.5}) and the bound  $|R|\le M_rN^{-(r-2)\alpha}$ we
obtain (\ref{A5.3}).
In order to  bound $|R|$ we use (\ref{3.15}),
$|R|
\le
N^{-(r-2)\alpha}\E\|Z'\|_r^{r-2}b_r^2(X)$ and apply H\"older inequality,
$$
\E\|Z'\|_r^{r-2}b_r^2(X)
\le
(\E\|Z'\|_r^r)^{(r-2)/r}(\E b_r^r(X))^{2/r}=M_r.
$$

\end{document}